\documentclass[11pt,a4paper]{article}
\usepackage{graphicx}
\usepackage{ifpdf}
\usepackage{amsthm,amsfonts,amssymb,amsmath,euscript,comment,bbm}
\usepackage{mathabx}
\usepackage[utf8]{inputenc}
\usepackage[T1]{fontenc}
\usepackage{color}
\usepackage{textcomp}
\usepackage{bbm}
\usepackage{slashed}
\usepackage[colorlinks=true,linkcolor=red,citecolor=blue]{hyperref}
\usepackage{mathrsfs}
\usepackage{dsfont}
\usepackage{geometry}
\usepackage{epsfig}
\usepackage{tikz}
\numberwithin{equation}{section}
\renewcommand{\c}{\cdot}
\newcommand{\unl}{\underline{L}}
\newcommand{\pa}{\partial}
\newcommand{\pr}{\pa}
\newcommand{\unu}{{\underline{u}}}

\newcommand{\mo}{\mathcal{O}}
\newcommand{\unchi}{\underline{\chi}}
\newcommand{\omb}{\underline{\om}}
\newcommand{\etab}{\underline{\eta}}
\newcommand{\sld}{{\slashed{d}}}
\def\pf{\mathcal{P}}
\def\Wb{\underline{W}}
\def\omd{\dual\om}
\def\ombd{\dual\omb}
\def\muc{\widecheck{\mu}}
\def\cuvi{H_{u_\infty}^{(0,\ub)}}
\def\mg{\mathcal{G}}
\def\AAb{\underline{\mathcal{A}}}
\def\ug{\underline{\mg}}
\def\mw{\mathcal{W}}
\def\uw{\underline{\mw}}
\def\pfb{{\underline{\pf}}}
\def\Ssl{\slashed{\S}}
\def\Tr{\mathbf{Tr}}
\def\ges{\gtrsim}
\def\Sslh{\widehat{\Ssl}}
\def\Kc{\widetilde{\K}}
\def\PP{{\mathcal{P}}}

\def\Jh{{\widehat{J}}}
\def\ins{\slashed{\in}}

\def\mf{\mathcal{F}}

\def\trcht{\widetilde{\trch}}

\def\TT{\mathcal{T}}

\def\bt{\widetilde{\b}}
\def\bbt{\widetilde{\bb}}

\def\trchbt{\widetilde{\trchb}}

\def\Des{\slashed{\De}}
\def\bu{\bullet}
\def\eps{\slashed{\in}}
\def\slD{\slashed{D}}
\def\sdivs{\slashed{\sdiv}}
\def\curls{\slashed{\curl}}

\def\af{a^\frac{1}{2}}

\def\Ub{\underline{H}}
\def\aa{{\underline{\a}}}
\def\sk{\mathfrak{s}}
\def\th{\theta}

\def\slJ{\widehat{J}}

\def\dual{{}^*}
\def\bF{{}^{(F)}\hspace{-1.5pt}\b}
\def\bbF{{}^{(F)}\hspace{-1.5pt}\bb}

\def\rhoF{{}^{(F)}\hspace{-1.5pt}\rho}
\def\siF{{}^{(F)}\hspace{-1.5pt}\si}

\def\nabs{\slashed{\nab}}
\def\dag{\dagger}

\def\slg{{\slashed{g}}}
\def\afd{a^{-\frac{1}{2}}}
\def\sic{\widetilde{\si}}

\def\vkp{\varkappa}
\def\vkpb{{\underline{\vkp}}}

\def\Hb{{\underline{H}}}

\newcommand{\mr}{\mathcal{R}}

\newcommand{\M}{\mathcal{M}}
\newcommand{\D}{\mathbf{D}}
\newcommand{\uf}{\underline{\mf}}

\newcommand{\cuv}{{H_u^{(0,\ub)}}}
\newcommand{\ucuv}{{\Hb_\ub^{(u_\infty,u)}}}

\newcommand{\g}{\mathbf{g}}
\newcommand{\F}{\mathbf{F}}
\newcommand{\R}{{\mathbf{R}}}
\newcommand{\K}{\mathbf{K}}

\def\bo{b^\frac{1}{4}}
\def\hot{\widehat{\otimes}}

\def\Lb{\unl}
\DeclareMathOperator{\lot}{l.o.t.}
\def\II{\mathcal{I}}
\def\la{\lambda}
\def\fc{\widecheck{f}}
\def\MM{\mathcal{M}}
\def\ub{\unu}
\def\chib{\unchi}

\def\om{\omega}
\def\ze{\zeta}
\def\Om{\Omega}
\def\bb{{\underline{\b}}}
\def\si{\sigma}
\def\mub{{\underline{\mu}}}
\def\dk{\mathfrak{d}}
\def\dkb{\slashed{\dk}}
\def\ga{\gamma}
\def\Ga{\Gamma}
\def\xib{\underline{\xi}}
\def\a{\alpha}
\def\b{\beta}
\def\hch{\widehat{\chi}}
\def\hchb{\widehat{\chib}}
\def\trch{\tr\chi}
\def\trchb{\tr\chib}
\def\trchbc{\widecheck{\trchb}}
\def\de{\delta}
\def\De{\Delta}
\def\nab{\nabla}
\def\ov{\overline}

\def\bbb{\underline{b}}
\def\Gab{\Ga_b}
\def\Gag{\Ga_g}
\def\les{\lesssim}

\def\OO{\mo}

\def\Up{\Upsilon}

\def\Fb{\underline{F}}
\def\T{\mathbf{T}}
\DeclareMathOperator{\curl}{curl}
\DeclareMathOperator{\grad}{grad}
\DeclareMathOperator{\tr}{tr}
\DeclareMathOperator{\sdiv}{div}
\def\bdiv{\mathbf{Div}}
\def\Ric{\mathbf{Ric}}
\def\S{\mathbf{S}}
\def\ef{\mathfrak{e}}
\def\A{\mathbf{A}}
\def\slA{\Asl}
\def\Asl{\slashed{A}}
\def\Ub{\underline{U}}
\def\slep{\slashed{\in}}
\def\Psisl{\slashed{\Psi}}
\def\Psit{\widetilde{\Psi}}

\def\W{\mathbf{W}}
\newtheorem{thm}{Theorem}[section]
\newtheorem{prop}[thm]{Proposition}
\newtheorem{lem}[thm]{Lemma}
\newtheorem{cor}[thm]{Corollary}
\newtheorem{rk}[thm]{Remark}
\newtheorem{df}[thm]{Definition}
\allowdisplaybreaks[3]
\usepackage{tikz}
\usetikzlibrary{decorations.pathmorphing}
\usetikzlibrary{arrows.meta, decorations.markings}
\usepackage{float}
\title{Formation of trapped surfaces for the Einstein--Maxwell--charged scalar field system}
\author{Dawei Shen and Jingbo Wan}
\begin{document}
\maketitle
\begin{abstract}
In this paper, we prove a scale-critical trapped surface formation result for the Einstein--Maxwell--charged scalar field (EMCSF) system, without any symmetry assumptions. Specifically, we establish a scale-critical semi-global existence theorem from past null infinity and show that the focusing of gravitational waves, the concentration of electromagnetic fields, or the condensation of complex scalar fields, each individually, can lead to the formation of a trapped surface. In addition, we capture a nontrivial charging process along past null infinity, which introduces new difficulties due to the abnormal behavior of the matter fields. Nevertheless, the semi-global existence result and the formation of a trapped surface remain valid.

{\centering\subsubsection*{\small Keywords}}
\small\noindent double null foliation, Einstein--Maxwell--charged scalar field, signature $s_2$, $|u|^p$--weighted estimates, trapped surface formation, charging process.
\end{abstract}
\tableofcontents
\section{Introduction}
The study of trapped surface formation for the Einstein vacuum equations began with the seminal work of Christodoulou \cite{Chr} and continued in \cite{An12,An,AnLuk,Chen-K,KLR,kr}. Coupling Einstein’s equations with matter fields introduces new trapping mechanisms, as explored for the Einstein--scalar field, Einstein--Maxwell, and Einstein--Yang--Mills systems in \cite{AnAth,AMY,LL,Yu,Zhao}. In this paper, we establish a scale-critical result on trapped surface formation for the \emph{Einstein--Maxwell--charged scalar field (EMCSF) system}, where, in addition to gravitational wave focusing, both the electromagnetic field and the charged scalar field contribute nontrivially through a new charging process along the initial outgoing null hypersurface. The result is based on a unified treatment of the energy estimates for Bianchi pair equations of Weyl curvature and matter fields without any symmetry assumption.
\subsection{Einstein--Maxwell--charged scalar field system}
In this paper, we study the evolution of the Einstein--Maxwell--charged scalar field (EMCSF) system, with a fixed coupling constant $\ef$, involving a $1+3$--dimensional Lorentzian manifold $(\M,\g)$, an electromagnetic $2$-form $\F$, and a complex scalar field $\psi$:
\begin{align}
\begin{split}\label{EMCSF}
    \Ric_{\mu\nu}-\frac{1}{2}\g_{\mu\nu}\R &= \T_{\mu\nu}^{EM}+\T_{\mu\nu}^{CSF},\\
    \D^\mu \F_{\mu\nu} &= -2\ef\, \Im\left(\psi(D_\nu\psi)^\dag\right),\\
    \g^{\mu\nu} D_\mu D_\nu \psi &= 0,
\end{split}
\end{align}
where $D = \D + i\ef \A$ denotes the gauge covariant derivative associated to a gauge potential $1$-form $\A$, and $\F = d\A$ is the corresponding Faraday tensor.\footnote{More precisely, $\F$ is the \emph{Faraday tensor} associated with the Maxwell field.} The energy-momentum tensors are given by
\begin{align}
\begin{split}\label{Tdfeqintro}
    \T^{EM}_{\mu\nu} &= 2\left( \F_{\mu\a} {\F_{\nu}}^\a - \frac{1}{4} \g_{\mu\nu} \F_{\a\b} \F^{\a\b} \right),\\
    \T^{CSF}_{\mu\nu} &= 2\left( \Re\left( D_\mu \psi (D_\nu \psi)^\dag \right) - \frac{1}{2} \g_{\mu\nu} D_\a \psi (D^\a \psi)^\dag \right),
\end{split}
\end{align}
representing, respectively, the contributions of the Maxwell field $\F$ and the charged scalar field $\psi$. \\ \\
Throughout this paper, we fix the gauge by requiring the outgoing component $\A_4 := \A(e_4) = 0$ everywhere. We will introduce coordinates $u$ and $\ub$ on $(\M,\g)$ through a \emph{double null foliation},\footnote{See Section \ref{doublenullfoliation} for the detailed construction of the double null foliation.} and prescribe characteristic initial data along the incoming null hypersurface $\Hb_0$ (where $\ub=0$) and the outgoing null hypersurface $H_{u_\infty}$ (where $u=u_\infty$).
\subsection{Small data regime}
For characteristic initial data prescribed along the incoming null cone $\Hb_0$ and the outgoing null cone $H_\infty$, the study of \eqref{EMCSF} in the small data regime has been highly successful. For the Einstein vacuum equations (i.e. $\F = 0$ and $\psi = 0$ in \eqref{EMCSF}), Christodoulou--Klainerman's monumental work \cite{Ch-Kl} established that sufficiently small initial data lead to the completeness of all future-directed geodesics, implying no singularity formation in the interior region. More recent proofs of the stability of Minkowski space have shown that even weaker decay and regularity assumptions suffice to guarantee spacetime completeness; see \cite{Bieri,lr2,Shen23}.

There are also numerous Minkowski stability results for Einstein's equations coupled with matter fields. Zipser \cite{zipser} extended the result of \cite{Ch-Kl} to the Einstein--Maxwell system, while Speck \cite{speck} generalized it to a broader class of electromagnetic fields, including standard Maxwell fields. The global stability of the Einstein--Klein--Gordon system has been studied by Lefloch--Ma \cite{lefloch,leflochMa}, Wang \cite{wang}, and Ionescu--Pausader \cite{Io-Pau}. We also refer the reader to the introductions of \cite{Shen22,Shen23} for further results on the stability of Minkowski spacetime.
\subsection{Trapped surface formation}
The stability results discussed in the previous section imply the completeness of solutions to \eqref{EMCSF} when the initial data along $\Hb_0$ and $H_\infty$ are sufficiently small. However, for large initial data, a geometric object known as a \emph{trapped surface} may form dynamically within their domain of influence. In 1965, Penrose \cite{Penrose} proved the following celebrated incompleteness theorem.
\begin{thm}[Penrose \cite{Penrose}]\label{Penrosesingularity}
Let $(\M,\g,\F,\psi)$ be a spacetime satisfying \eqref{EMCSF}, containing a non-compact Cauchy hypersurface. If $\M$ contains a compact trapped surface, then it is future causally geodesically incomplete.
\end{thm}
Thus, in this setting, the proof of singularity formation in general relativity is reduced to deriving the formation of trapped surfaces. According to the stability results for the Minkowski space, e.g. \cite{Ch-Kl,Shen23,zipser}, the initial data chosen cannot be small to form a trapped surface. Solving \eqref{EMCSF} with large initial data is significantly more difficult: for the general large-data problem, only local existence results are known. However, forming a trapped surface at a later time requires a mathematical result beyond the local existence theory.
\subsubsection{Einstein vacuum equations}
In \cite{Chr}, Christodoulou discovered a remarkable mechanism for the dynamical formation of trapped surfaces. He showed that trapped surfaces can form, even in vacuum spacetimes, from regular, concentrated initial configurations. Christodoulou introduced an open set of initial data (the \emph{short pulse ansatz}) on an outgoing null hypersurface and proved that strongly focused gravitational waves arriving from past null infinity lead to the formation of a trapped surface. In \cite{kr}, Klainerman--Rodnianski extended Christodoulou's work by introducing a parabolic scaling and studying a broader class of initial data. Their new scaling captured hidden smallness in the nonlinear interactions of the Einstein equations, an idea further generalized by An in \cite{An12}. In \cite{LY}, Li--Yu showed the existence of spacelike initial data, initially free of trapped surfaces, that evolve into a trapped surface in its future development. Klainerman--Luk--Rodnianski \cite{KLR} significantly relaxed the lower bound condition in \cite{Chr,kr} by developing a fully anisotropic mechanism for the formation of trapped surfaces. The first scale-critical result was established by An--Luk in \cite{AnLuk}, leading to subsequent studies of the apparent horizon; see \cite{An17,AnHan}. In \cite{An}, An provided a simplified proof of the scale-critical result in the far-field regime by designing a scale-invariant norm based on the signature $s_2$ and the corresponding decay rates. All of the above results are based on an adapted double null foliation. More recently, Chen--Klainerman \cite{Chen-K} reproved trapped surface formation using a non-integrable Principal Temporal (PT) frame introduced in \cite{GKS,KS:main}.
\subsubsection{Einstein equations coupled with matter fields}
The formation of trapped surfaces for Einstein equations coupled with matter fields also plays an important role in both mathematics and physics. In a finite region, Yu \cite{Yucmp,Yu} extended the work of Klainerman--Rodnianski \cite{kr} and proved trapped surface formation for the Einstein--Maxwell system. Li--Liu \cite{LL} studied trapped surface formation for the Einstein--scalar field system. In \cite{AnAth}, An--Athanasiou extended Yu's work and established a scale-critical criterion for trapped surface formation from past null infinity by incorporating new ingredients from \cite{An}. Athanasiou--Mondal--Yau \cite{AMY} proved a scale-invariant trapped surface formation result for the coupled Einstein--Yang--Mills system. More recently, Zhao--Hilditch--Valiente Kroon \cite{Zhao} established trapped surface formation for the Einstein--scalar field system in a general setting using scale-critical analysis.
\subsubsection{Einstein--Maxwell--charged scalar field}
The formation of trapped surfaces for the Einstein–Maxwell–charged scalar field (EMCSF) system \eqref{EMCSF}, which describes charged gravitational collapse, has been studied extensively under spherical symmetry. In \cite{Kom}, Kommemi classified global structures of spherically symmetric EMCSF spacetimes and ruled out certain violations of cosmic censorship. In \cite{Anlim}, An–Lim established a sharp criterion for trapped surface formation using modern double null techniques. Later, An–Tan \cite{AnTan} gave the first full proof of weak cosmic censorship for spherically symmetric collapse. In \cite{KU}, Kehle–Unger constructed examples of collapse to extremal Reissner–Nordström black holes, demonstrating violations of the third law of black hole thermodynamics; see \cite{KU2} for further developments. The global dynamics of the spherically symmetric EMCSF system has also been explored, including the decay of \emph{weakly charged} solutions on sub-extremal Reissner–Nordstr\"om backgrounds by Van de Moortel~\cite{Van}.

Most existing results have focused on spherically symmetric spacetimes. Although these studies have deepened our understanding of charged collapse and cosmic censorship under symmetry, the behavior of the EMCSF system beyond spherical symmetry remains largely unexplored. To the best of our knowledge, this paper provides the first mathematically rigorous result on trapped surface formation for the EMCSF system without any symmetry assumptions. We prove a scale-critical semi-global existence theorem, a trapped surface formation result, and establish a charging process. Our work extends Christodoulou’s short pulse framework to a richer matter model and paves the way for future studies of charged radiative collapse in full generality.
\subsection{Rough versions of the main theorems}
In this section, we state simplified versions of our main results.
\begin{thm}[Main Theorem (first version)]\label{Maintheoremintro}
    There exist constants $a_0>0$ sufficiently large and $e_0>0$ sufficiently small such that the following hold. For $a>a_0$ and $0<\ef<e_0$, with initial data:\footnote{See Section \ref{doublenullfoliation} for the definition of $\hch$ and $\bF$.}
    \begin{itemize}
        \item $\displaystyle\sum_{i\leq 10,k\leq 4}\afd|u_\infty|\left\|\nabs_4^k(|u_\infty|\nabs)^i(\hch,\bF,\psi)\right\|_{L^\infty(S_{u_\infty,\ub})}\leq1$ along $u=u_\infty$,
        \item Minkowskian initial data along $\ub=0$.
    \end{itemize}
    Einstein--Maxwell--charged scalar field system \eqref{EMCSF} admits a unique smooth solution in the colored region of Figure \ref{maintheorempicture}. Moreover, we obtain detailed control of all the quantities associated with the double null foliations of spacetime; see Theorem \ref{maintheorem}.
\end{thm}
\begin{figure}[ht]
    \centering
    \begin{tikzpicture}[scale=0.9, decorate]
\fill[yellow] (1,1)--(2,2)--(4,0)--(3,-1)--cycle;

\draw[very thick] (1,1)--(2,2)--(4,0)--(3,-1)--cycle;

\node[scale=1, black] at (1.1,1.9) {$H_{-\frac{a}{4}}$};
\node[scale=1, black] at (3.3,1.3) {$\Hb_1$};
\node[scale=1, black] at (1.7,-0.3) {$\Hb_0$};

\node[scale=0.8, black] at (1.2,-1) {Minkowski};

\draw[dashed] (0,-4)--(0,4);
\draw[dashed] (0,-4)--(4,0)--(0,4);
\draw[dashed] (0,0)--(2,2);
\draw[dashed] (0,2)--(3,-1);

\draw[->] (3.3,-0.6)--(3.6,-0.3) node[midway, sloped, above, scale=1] {$e_4$};
\draw[->] (2.4,-0.3)--(1.7,0.4) node[midway, sloped, above, scale=1] {$e_3$};

\filldraw[red] (2,2) circle (2pt);

\draw[thin] (2.1,2.5) node[right, scale=1] {$S_{-\frac{ a}{4},1}$, trapped surface} -- (2,2);

\filldraw[red] (4,0) circle (2pt);

\draw[thin] (4.1,0.5) node[right, scale=1] {\shortstack{$S_{u_\infty,1}$, $m \simeq  a$, $|Q| \simeq \ef a$}} -- (4,0);

\draw[->, decorate, decoration={snake, amplitude=0.5mm, segment length=2mm}, thin, red]  (4.3,-0.7) -- (3.3,0.3);
\draw[->, decorate, decoration={snake, amplitude=0.5mm, segment length=2mm}, thin, red]  (4.0,-1) -- (3.0,0);
\draw[->, decorate, decoration={snake, amplitude=0.5mm, segment length=2mm}, thin, red]  (3.7,-1.3) -- (2.7,-0.3);

\node[scale=1, anchor=west] at (4.2,-1.2) {Short pulse across $H_{u_{\infty}}$ of size $ a$};
\end{tikzpicture}
\caption{{\footnotesize The initial conditions in Theorem \ref{Maintheoremintro} can lead to the formation of a trapped surface $S_{-\frac{a}{4},1}$ in the future of $H_{u_\infty}\cup\Hb_0$. Moreover, the Hawking mass and electric charge of the final sphere $S_{u_\infty,1}$ on $H_{u_\infty}$ satisfy $m(S_{u_\infty,1}) \simeq a$ and $|Q(S_{u_\infty,1})| \simeq \ef a$, respectively.}}\label{maintheorempicture}
\end{figure}
\begin{rk}
    Theorem \ref{Maintheoremintro} can be interpreted as a semi-global existence result near the past null infinity. In fact, it holds for all $u_\infty \ll -a$, and all the estimates obtained are uniform as $u_\infty \to -\infty$.
\end{rk}
Theorem \ref{Maintheoremintro} is restated as Theorem \ref{maintheorem} in Section \ref{secmainstate} and proved in Section \ref{proofmain}. Theorem \ref{Maintheoremintro} implies the existence of the background spacetime.
\begin{thm}[Trapping Mechanism and Charging Process (first version)]\label{trappedsurfaceintro}
Under the assumptions of Theorem \ref{Maintheoremintro}, suppose in addition that the following lower bound conditions hold:\footnote{See Section \ref{secmatterfields} for detailed definitions of null components of matter fields. In particular, $\Psi_4=e_4\psi$ under the Gauge choice $\A_4=0$.}
    \begin{align}
        \int_0^1|u_\infty|^2\left(|\hch|^2+|\bF|^2+|\Psi_4|^2\right)(u_\infty,\ub',x^1,x^2)d\ub'&\geq a,\qquad \forall \;(x^1,x^2),\label{geqaintro}\\
        \left|\int_{H_{u_\infty}^{(0,1)}}\Om\Im(\psi\Psi_4^\dag)\right|&\geq a.\label{chargelowerboundintro}
    \end{align}
    Then, as illustrated in Figure \ref{maintheorempicture},  we have:
    \begin{itemize}
        \item the $2$--sphere $S_{-\frac{a}{4},1}$ is a trapped surface,
        \item the Hawking mass of the sphere $S_{u_\infty,1}$ satisfies
    \begin{equation}\label{massintro}
        m(S_{u_\infty,1})\simeq a,
    \end{equation}
        \item the electric charge of the sphere $S_{u_\infty,1}$ satisfies
    \begin{equation}\label{chargeintro}
        |Q(S_{u_\infty,1})|\simeq \ef a.
    \end{equation}
    \end{itemize}
\end{thm}
Theorem \ref{trappedsurfaceintro} is restated as Theorem \ref{thmtrappedsurface} in Section \ref{sectrapping}. It shows that a trapped surface $S_{-\frac{a}{4},1}$ can form dynamically, although the initial data along $H_{u_\infty}$ and $\Hb_0$ are initially free of trapped surfaces. Moreover, \eqref{chargeintro} shows that the spacetime becomes charged after a short pulse. Combining \eqref{massintro} and \eqref{chargeintro}, we find that the charge-to-mass ratio of $S_{u_\infty,1}$ is bounded above by $e_0$.
\begin{thm}[Main Theorem (rescaled version)]\label{scalingintro}
Let $\de>0$ be a fixed constant. Then there exist constants $a_0(\de)$ sufficiently large and $e_0>0$ sufficiently small\footnote{Here, $e_0$ coincides with the $e_0$ in Theorem \ref{Maintheoremintro} and is independent of $a_0$ and $\de$. In particular, the charge-to-mass ratio $\de \ef$ of $S_{u_\infty,\de}$ remains bounded above by $e_0$.} such that the following holds. For $a>a_0$ and $0<\de\ef<e_0$, with initial data satisfying:
    \begin{itemize}
        \item $\displaystyle \sum_{i\leq 10,k\leq 4}\afd|u_\infty|\left\|(\de\nabs_4)^k(|u_\infty|\nabs)^i(\hch,\bF,\psi)\right\|_{L^\infty(S_{u_\infty,\ub})}\leq1$ along $u=u_\infty$,
        \item Minkowskian initial data along $\ub=0$,
        \item the following lower bound condition holds:
    \begin{align}
        \int_0^\de|u_\infty|^2\left(|\hch|^2+|\bF|^2+|\Psi_4|^2\right)(u_\infty,\ub',x^1,x^2)d\ub'&\geq \de a,\qquad \forall \;(x^1,x^2),\label{scalinglowerboundintro}\\
        \left|\int_{H_{u_\infty}^{(0,\de)}}\Im(\psi\Om\Psi_4^\dag)\right|&\geq\de a.
    \end{align}
    \end{itemize}
    \begin{figure}[H]
    \centering
    \begin{tikzpicture}[scale=0.9, decorate]
\fill[yellow] (1,1)--(2,2)--(4,0)--(3,-1)--cycle;

\draw[very thick] (1,1)--(2,2)--(4,0)--(3,-1)--cycle;

\node[scale=1, black] at (1.1,1.9) {$H_{-\frac{\de a}{4}}$};
\node[scale=1, black] at (3.3,1.3) {$\Hb_\de$};
\node[scale=1, black] at (1.7,-0.3) {$\Hb_0$};

\node[scale=0.8, black] at (1.2,-1) {Minkowski};

\draw[dashed] (0,-4)--(0,4);
\draw[dashed] (0,-4)--(4,0)--(0,4);
\draw[dashed] (0,0)--(2,2);
\draw[dashed] (0,2)--(3,-1);

\draw[->] (3.3,-0.6)--(3.6,-0.3) node[midway, sloped, above, scale=1] {$e_4$};
\draw[->] (2.4,-0.3)--(1.7,0.4) node[midway, sloped, above, scale=1] {$e_3$};

\filldraw[red] (2,2) circle (2pt);

\draw[thin] (2.1,2.5) node[right, scale=1] {$S_{-\frac{\de a}{4},\de}$, trapped surface} -- (2,2);

\filldraw[red] (4,0) circle (2pt);

\draw[thin] (4.1,0.5) node[right, scale=1] {\shortstack{$S_{u_\infty,\de}$, $m \simeq \de a$, $|Q| \simeq \de^2 \ef a$}} -- (4,0);

\draw[->, decorate, decoration={snake, amplitude=0.5mm, segment length=2mm}, thin, red]  (4.3,-0.7) -- (3.3,0.3);
\draw[->, decorate, decoration={snake, amplitude=0.5mm, segment length=2mm}, thin, red]  (4.0,-1) -- (3.0,0);
\draw[->, decorate, decoration={snake, amplitude=0.5mm, segment length=2mm}, thin, red]  (3.7,-1.3) -- (2.7,-0.3);

\node[scale=1, anchor=west] at (4.2,-1.2) {Short pulse across $H_{u_{\infty}}$ of size $\de a$};
\end{tikzpicture}
    \caption{{\footnotesize The initial conditions in Theorem \ref{scalingintro} can lead to the formation of a trapped surface $S_{-\frac{\de a}{4},\de}$ in the future of $H_{u_\infty}\cup\Hb_0$. Moreover, the Hawking mass and electric charge of the final sphere $S_{u_\infty,\de}$ on $H_{u_\infty}$ satisfy $m(S_{u_\infty,\de}) \simeq \de a$ and $|Q(S_{u_\infty,\de})| \simeq \de^2\ef a$, respectively.}}
    \label{scalingpicture}
\end{figure}
    Then the Einstein--Maxwell--charged scalar field system \eqref{EMCSF} admits a unique smooth solution in the colored region of Figure \ref{scalingpicture}. Moreover:
    \begin{itemize}
    \item the $2$--sphere $S_{-\frac{\de a}{4},\de}$ is a trapped surface,
    \item the Hawking mass of the sphere $S_{u_\infty,\de}$ has size $\de a$,
    \item the electric charge of the sphere $S_{u_\infty,\de}$ has size $\de^2\ef a$,
    \item we obtain detailed control of all the quantities associated with the double null foliations of spacetime; see Theorem \ref{thmscaling}.
    \end{itemize}
\end{thm}
Theorem \ref{scalingintro} is restated as Theorem \ref{thmscaling} in Section \ref{secscaling}. This provides a rescaled version of Theorems \ref{Maintheoremintro} and \ref{trappedsurfaceintro}, related to the main results of \cite{AnLuk,Chr}. 
\begin{rk}
In terms of $L^2$--based Sobolev spaces, there exist initial data satisfying the assumptions of Theorem \ref{scalingintro} such that the metric is only large in $H^\frac{3}{2}$, but can be small in $H^s$ for any $s<\frac{3}{2}$. Thus, we establish a scale-critical trapped surface formation result; see Remarks 1 and 2 in \cite{An12} and the discussion after Corollary 1.4 in Section 1 of \cite{AnLuk} for further details.
\end{rk}
\subsection{Main difficulty and new ingredients}\label{overview}
This section summarizes the proof of the main Theorem \ref{Maintheoremintro} and highlights the key ideas and new ingredients introduced in this paper. Our analysis is primarily based on the derivation of correct estimates for the Weyl curvature components, Ricci coefficients, the scalar field and the Maxwell field in the Einstein--Maxwell--charged scalar field system \eqref{EMCSF}. The definitions of all quantities used in this section are provided in Section \ref{doublenullfoliation}.
\subsubsection{\texorpdfstring{$|u|^p$}{}--weighted estimates}\label{secupintro}
The core of this paper is an energy-type estimate built upon a general treatment of Bianchi pairs without any symmetry assumptions. Specifically, we establish the $|u|^p$--weighted estimate method for general Bianchi equations in Proposition \ref{keyintegral}. This method combines the advantages of the signature techniques introduced by Klainerman--Rodnianski \cite{kr} and further developed by An \cite{An12,An}. Similar ideas also appear in \cite{AnLuk,Chen-K,holzegel,Shen22}. In addition, we prove Theorem \ref{Junkmanthm}, which provides a unified treatment of all nonlinear error terms arising in the energy estimates under the bootstrap assumption. Theorem \ref{Junkmanthm} can be regarded as a semi-global extension of Theorem 5.8 in \cite{Shen23}, which unified the treatment of nonlinear error terms for Minkowski stability.
\subsubsection{Maxwell fields and complex scalar fields}\label{secFintro}
We recast the Maxwell equations and the complex scalar wave equation in \eqref{EMCSF} into systems of Bianchi pairs.\footnote{See Remark \ref{Bianchitype} for more details.} Applying the $|u|^p$--weighted estimate framework yields $L^2$ estimates for the electromagnetic field $\F$ and the gauge covariant derivatives $\Psi$ of the charged scalar field, up to fluxes across the relevant (incoming or outgoing) null cones. See Section \ref{secmatterfields} for detailed definition of their null components. Taking components $(\bF, (\rhoF, \siF))$ of Maxwell fields as an example, they satisfy the following Bianchi pair system
\begin{align}
\nabs_3\bF+\frac{1}{2}\trchb\bF&=-\sld_1^*(\rhoF,\siF) -2\ef\Im\left(\psi\Psisl^\dag\right)+\lot,\label{ex1}\\
\nabs_4(\rhoF,\siF)&=\sld_1\bF+2\ef\left(\Im(\psi\Psi_4^\dag),0\right)+\lot,\label{ex2}
\end{align}
where all quantities are defined in Section \ref{secmatterfields} and $\lot$ denotes the lower order terms, which are ignored here for simplicity. Multiplying \eqref{ex1} by $\bF$ and \eqref{ex2} by $(\rhoF,\siF)$ and taking the sum, we infer
\begin{align*}
&\bdiv\left(|\bF|^2e_3\right)+\bdiv\left(|(\rhoF,\siF)|^2e_4\right)\\
=\;&2\sdivs\left(\bF\rhoF\right)+2\curls\left(\bF\siF\right)+4\ef\Im\left(\rhoF\psi\Psi_4^\dag-\bF\psi\Psisl^\dag\right)+\lot
\end{align*}
Integrating it in the region $V:=V(u,\ub):=\{[u_\infty,u]\times[0,\ub]\}$, we obtain
\begin{align*}
    \int_{H_u\cap V} |\bF|^2+\int_{\Hb_\ub\cap V} |(\rhoF,\siF)|^2&\les \int_{H_{u_\infty}\cap V}|\bF|^2+\ef\int_V|\rhoF||\psi||\Psi_4|\\
    &+\ef\int_V|\bF||\psi||\Psisl|+\lot
\end{align*}
Based on the initial assumption in Theorem \ref{Maintheoremintro}, the matter fields have the following abnormal behavior along the initial hypersurface $H_{u_\infty}$:
\begin{align}
|\psi,\Psi_4,\bF|\simeq \frac{\af}{|u_\infty|},\qquad\qquad |\rhoF,\Psisl|\simeq \frac{\af+\ef a}{|u_\infty|^2},
\end{align}
see Section \ref{secsignatures}, We then make the following bootstrap assumptions in $V$:
\begin{align}\label{abnormalintro}
    |\psi,\Psi_4,\bF|\les \frac{\bo\af}{|u|},\qquad\qquad |\rhoF,\Psisl|\les \frac{\bo\af+\bo\ef a}{|u|^2},
\end{align}
where $b\gg 1$ is an auxiliary constant; see Section \ref{smallconstant}. Thus, we deduce
\begin{align*}
    \ef\int_V|\rhoF||\psi||\Psi_4|+|\bF||\psi||\Psisl|\les\ef\int_V\frac{b^\frac{3}{4}a^\frac{3}{2}}{|u|^4}+\frac{\ef b^\frac{3}{4}a^2}{|u|^4}\les b(\ef\af+\ef^2a).
\end{align*}
Combining with the initial assumption, we infer
\begin{align*}
    \int_{H_u\cap V}|\bF|^2+\int_{\Hb_\ub\cap V}|\rhoF,\siF|^2\les a+b(\ef\af+\ef^2a).
\end{align*}
In order to the deduce desired decay rates to close the bootstrap argument \eqref{abnormalintro}, it is necessary to impose the following smallness condition for the coupling constant $\ef$:
\begin{align}\label{efsmallintro}
    b\ef\ll 1.
\end{align}
Under the assumption \eqref{efsmallintro}, we have
\begin{align*}
     \int_{H_u\cap V}|\bF|^2+\int_{\Hb_\ub\cap V}|\rhoF,\siF|^2\les a.
\end{align*}
The higher order estimates for $\bF$, $\rhoF$ and $\siF$ can be deduced similarly by applying commutation identities. Combining with Sobolev inequalities, we obtain the following bounds in $V$:
\begin{align}\label{abnormal}
|\bF| \lesssim \frac{\af}{|u|}, \qquad\qquad |\rhoF, \siF| \lesssim \frac{\af+\ef a}{|u|^2}.
\end{align}
The other components of the matter fields, $\Psi_4$, $\Psisl$, and $\bbF$, are treated similarly, except that $\Psi_3$ requires an appropriate renormalization; see~\eqref{dfPsit3}. This renormalization is necessary due to the insufficient decay of $\Psi_3$ and is inspired by~\cite{AAG}; see Remark~\ref{AAGmethod} for further details.
\begin{rk}
    Compared to the corresponding problem for the Einstein--Maxwell system and the Einstein--scalar field system, the Einstein--Maxwell charged scalar field (EMCSF) system \eqref{EMCSF} causes additional technical difficulties not encountered in the previous cases due to the abnormal properties of the coupling constant $\ef$:
    \begin{enumerate}
    \item The terms involving $\ef$ behaves $\af$ worse than the other terms; see the estimate for $\rhoF$ and $\Psisl$ in \eqref{abnormal}. In other words, the estimates for the components of Maxwell fields $\rhoF$ and scalar fields $\Psisl$ behaves normally if $\ef\simeq\afd$. However, we only assumed $0<\ef<e_0\ll 1$ in Theorem \ref{Maintheoremintro}, where $e_0$ is a constant independent of $a$. This leads to additional difficulties in the estimations of nonlinear error terms.
    \item The signature $s_2$, introduced in \cite{An}, is a powerful tool that simplifies the analysis in the Einstein--Vacuum and Einstein--Maxwell cases by being conserved in all equations. To maintain this conservation in the Einstein--Maxwell--charged scalar field system, we must set $s_2(\ef) = 0.5$, as deduced from \eqref{ex1}.\footnote{See Section \ref{secsignatures} for the signature $s_2$ of all quantities and their desired decay rates.} However, since $\ef$ is a fixed constant throughout the paper, this assignment leads to abnormal behavior in all quantities involving $\ef$, causing additional difficulties in estimating nonlinear error terms via $s_2$ and scale-invariant norms.
    \end{enumerate}
\end{rk}
\subsubsection{Electromagnetic potential}
Given the Faraday tensor $\F$, there exists an electromagnetic potential $\A$ satisfying
\begin{align}\label{dA=F}
\D_\mu \A_\nu - \D_\nu \A_\mu = \F_{\mu\nu}.
\end{align}
Denoting the null components of $\A$ by
\begin{align*}
U := \A_4 = \A(e_4), \qquad \Ub := \A_3 = \A(e_3), \qquad \Asl_B := \A_B = \A(e_B),
\end{align*}
we deduce the following consequences of \eqref{dA=F}:\footnote{See Proposition \ref{lem:FdAeq} for the precise equations.}
\begin{align}
\begin{split}\label{dA=Feqintro}
\rhoF &= \frac{1}{2} \nabs_3 U - \frac{1}{2} \nabs_4 \Ub + \lot,\\
\siF &= -\curls \Asl,\\
\bF &= \nabs U - \nabs_4 \Asl + \lot,\\
\bbF &= \nabs \Ub - \nabs_3 \Asl - \chib \c \Asl + \lot.
\end{split}
\end{align}
Since the gauge choice for $\A$ is not unique, we first consider the natural Lorenz gauge condition $\bdiv \A = 0$, which implies
\begin{align*}
\frac{1}{2} \nabs_3 U + \frac{1}{2} \nabs_4 \Ub + \frac{1}{2} \trchb U = \sdivs \Asl + \lot.
\end{align*}
Combining this with \eqref{dA=Feqintro}, we derive the following pairing Bianchi system:
\begin{align*}
\nabs_3(U,0) + \frac{1}{2} \trchb (U,0) &= \sld_1 \Asl + (\rhoF,\siF) + \lot,\\
\nabs_4 \Asl &= -\sld_1^*(U,0) - \bF + \lot,\\
\nabs_4(\Ub,0) &= \sld_1 \Asl + (-\rhoF, \siF) + \lot,\\
\nabs_3 \Asl + \frac{1}{2} \trchb \Asl &= -\sld_1^*(\Ub,0) - \bbF + \lot.
\end{align*}
However, estimating this system is difficult due to the lack of integrability of the bulk terms involving components of $\F$ when applying the $|u|^p$--weighted estimates from Section \ref{secupintro}. Motivated by the spherically symmetric results for the EMCSF system in \cite{Anlim,KU}, we adopt the gauge choice
\begin{equation}\label{GaugeU=0}
U=0.
\end{equation}
Substituting \eqref{GaugeU=0} into \eqref{dA=Feqintro}, the system for $(\A,\F)$ simplifies to an elliptic-transport structure analogous to the null structure equations for $(\Ga, \R)$:
\begin{align*}
\nabs_4 \Ub &= -2\rhoF + \lot,\\
\nabs_4 \Asl &= -\bF + \lot.
\end{align*}
Hence, the estimates for $\Asl$ and $\Ub$ follow directly from the estimates for $\bF$ and $\rhoF$. As an immediate consequence, $\A$ has the same regularity as $\F$ and $\Psi$.
\subsubsection{Weyl components}\label{secRintro}
The Bianchi equations for Weyl components take the following schematic form:
\begin{align}
    \begin{split}\label{WBianchiintro}
        \nabs_3\W_{(1)}&=\slD \W_{(2)}+\Ga\c\W+\D\T,\\
        \nabs_4\W_{(2)}&=-\slD^*\W_{(1)}+\Ga\c\W+\D\T.
    \end{split}
\end{align}
With the matter fields $\T$ controlled, we can derive energy-type estimates for the Weyl curvature components $\W$. The strategy again involves recasting the Bianchi equations for $\W$ into Bianchi pairs and applying the $|u|^p$--weighted estimate method introduced in Section \ref{secupintro}. However, due to the presence of $\D \T$ in \eqref{WBianchiintro} and the definition of $\T$ in \eqref{Tdfeqintro}, the Weyl components $\W$ are one order less regular than $\F$, $\Psi$, $\A$ and $\Ga$.\footnote{The fact that $\Ga$ has the same regularity as $\F$ and $\Psi$ follows directly from the Maxwell equations and the wave equations.} Moreover, suitable renormalizations of $\W$ are needed to handle problematic terms arising from matter couplings. We now briefly illustrate the key ideas in the estimate for each Bianchi pairs.\footnote{See \eqref{defr} for the definitions of the null components of $\W$.}
\paragraph{The Bianchi pair $(\a,\b)$}
The Bianchi equations for $(\a,\b)$ take the following form:
\begin{align*}
    \nabs_3\a+\frac{1}{2}\trchb\,\a&=-2\sld_2^*\b+\lot\\
    \nabs_4\b&=\sld_2\a+\frac{1}{2}\nabs_4\Ric_{4\bu}+\lot,
\end{align*}
where $\lot$ denotes the lower order terms we ignore for instance. Notice that $\nabs_4\Ric_{4A}$ contains the term $e_4(e_4\psi)$, which has not been controlled in Section \ref{secFintro}. In order to eliminate this term, we introduce the following renormalization:\footnote{This renormalization is also used in \cite{AnAth,AMY,LL}.}
\begin{align*}
    \bt:=\b-\frac{1}{2}\Ric_{4A}.
\end{align*}
Then, we deduce the following Bianchi equations for $(\a,\bt)$:
\begin{align*}
\nabs_3\a&=-2\sld_2^*\bt+\lot,\\
\nabs_4\bt&=\sld_2\a+\lot
\end{align*}
Applying $|u|^p$--weighted estimates, we can deduce the desired estimates for $(\a,\bt)$; see Proposition \ref{estab}.
\paragraph{The Bianchi pair $(\b,(\rho,-\si))$}
The Bianchi equations for $(\b,(\rho,-\si))$ take the following form:
\begin{align*}
\nabs_3\b+\trchb\,\b&=-\sld_1^*(\rho,-\si)+\lot,\\
\nabs_4(\rho,-\si)&=\sld_1\b-\frac{1}{2}(\hchb\c\a,-\hchb\wedge\a)+\lot
\end{align*}
Under the bootstrap assumption, we have the following estimates in $V$:
\begin{align}\label{bootrho}
    |\a|\les \frac{\bo\af}{|u|},\qquad |\hchb|\les \frac{\bo\af}{|u|^2},\qquad |\rho,\si|\les \frac{\bo a}{|u|^3}.
\end{align}
Observing the equation of $\nabs_4(\rho,-\si)$, we can only deduce the following estimate:
\begin{align*}
    |\rho,\si|\les \frac{b^\frac{1}{2}a}{|u|^3},
\end{align*}
which cannot close the bootstrap arguments due to the largeness of $b$. To overcome this difficult, we first notice that the bootstrap assumption of $\a$ is already improved in the estimate for $(\a,\b)$. Moreover, we can also improve the bootstrap assumptions for $\hch$ and $\hchb$ by their equations in the direction $e_4$ and the bounds obtained for $\a$. More precisely, we have
\begin{align}
    |\a|\les \frac{\af}{|u|},\qquad\quad |\hch|\les \frac{\af}{|u|},\qquad\quad |\hchb|\les\frac{\af}{|u|^2},\label{Gabimprovedintro}
\end{align}
see Proposition \ref{esthchhchb} for the precise statements for \eqref{Gabimprovedintro}. See also Propositions 4.3 and 4.4 in \cite{An} for similar arguments. Applying \eqref{Gabimprovedintro}, we can deduce
\begin{align*}
    |\b|\les \frac{\af}{|u|^2},\qquad |\rho,\si|\les \frac{a}{|u|^3},
\end{align*}
which improves the bootstrap bounds \eqref{bootrho}; see Proposition \ref{estbr}.
\paragraph{The Bianchi pair $((\rho,\si),\bb)$}
The Bianchi equations for $((\rho,\si),\bb)$ take the following form:
\begin{align*}
\nabs_3(\rho,\si)+\frac{3}{2}\trchb(\rho,\si)&=-\sld_1\bb-\frac{1}{2}(\hch\c\aa,\hch\wedge\aa)-\frac{1}{2}\nabs_3\Ric_{34}+\lot,\\
\nabs_4\bb+\trch\,\bb&=\sld_1^*(\rho,\si)+\lot
\end{align*}
Notice that the term $e_3(e_3\psi)$ appears in $\nabs_3\Ric_{34}$, which has not been controlled in Section \ref{secFintro}. Moreover, the appearance of $\aa$ in these equations may also lead additional difficulties. To overcome these difficulties, we proceed as in \cite{LL}. We define the following renormalized quantities:
\begin{align*}
    \Kc:=\K-\frac{1}{|u|^2},\qquad\quad \sic:=\si-\frac{1}{2}\hch\wedge\hchb,\qquad\quad \mu:=\Kc-\sdivs\eta.
\end{align*}
We then have the following Bianchi equations for $((\Kc,-\sic),\bb)$:
\begin{align}
\begin{split}\label{Kcsicintro}
\nabs_3(\Kc,-\sic)+\frac{3}{2}\trchb(\Kc,-\sic)&=\sld_1\bb+\frac{1}{2}\trchb\,\mu+\lot,\\
\nabs_4\bb+\trch\,\bb&=-\sld_1^*(\Kc,-\sic)+\lot
\end{split}
\end{align}
It is easy to verify that the term $\trchb\,\mu$ in \eqref{Kcsicintro} has borderline decay in $|u|$ when we apply the $|u|^p$--weighted estimates for $((\Kc,\sic),\bb)$. To overcome this lack of integrability, we first notice that the transport equation for $\mu$ takes the following form:
\begin{align*}
    \nabs_4\mu=-\frac{1}{|u|^2}\trch+\lot.
\end{align*}
According to the behavior of $\trch$, we can easily deduce that $\muc$ has better decay than $\ov{\mu}$,\footnote{$\ov{\mu}$ and $\muc$ are defined in Definition \ref{average}.} see Proposition \ref{fluxmu}. Thus, the $\ell\geq 1$ modes of $(\Kc,\sic)$ can be control by the improved estimate for $\muc$. Next, we have from Gauss-Bonnet theorem that $\ov{\Kc}$ is sufficiently small. Moreover, we have from the torsion equation \eqref{torsion} that $\ov{\sic}=0$. Combining the above facts, we can obtain the energy estimate for $(\Kc,\sic)$; see Proposition \ref{estrb}.
\paragraph{The Bianchi pair $(\bb,\aa)$} 
The last Bianchi pair $(\bb,\aa)$ satisfies the following equations:
\begin{align}
\begin{split}\label{bbaaintro}
\nabs_3 \bb + 2\trchb\,\bb&=-\sld_2\aa+\frac{1}{2}\nabs\Ric_{33}-\frac{1}{2}\nabs_3 \Ric_{3\bu}+\lot,\\
\nabs_4 \aa &= 2\sld_2^*\bb+\lot.
\end{split}
\end{align}
It can be easily verified that the term $\nabs \Ric_{33}$ in \eqref{bbaaintro} has borderline decay in $|u|$, which would introduce additional difficulties if treated directly. However, we observe that $\aa$ is decoupled from the Bianchi systems for $(\a, \b)$, $(\b, (\rho, -\si))$, and $((\Kc, -\sic), \bb)$. Moreover, $\aa$ appears only in the null structure equation for $\nabs_3 \hchb$, which is unnecessary for our estimates since $\hchb$ can be controlled through the $\nabs_4 \hchb$ equation and the Codazzi equation. Therefore, we can eliminate $\aa$ from the bootstrap assumptions and close all estimates in $V$ without analyzing the Bianchi system \eqref{bbaaintro} for $(\bb,\aa)$. This idea of removing $\aa$ from the system was first introduced by An--Luk in \cite{AnLuk}; see the discussion in Section 1.3.5 of \cite{AnLuk}.
\subsubsection{Ricci coefficients}
As mentioned in Section \ref{secRintro}, after estimating the first Bianchi pair $(\a, \bt)$, we further improve the bootstrap bounds for $\hch$ and $\hchb$. These quantities are particularly delicate in Christodoulou’s short pulse ansatz, as they act as the source of the pulse and must be controlled first in the hierarchy. To control all other quantities, we first establish lower-order estimates by applying the null structure equations viewed as transport equations. The top-order estimates for Ricci coefficients are more subtle. In particular, the top-order estimates for $\trch$ and $\trchb$ can be obtained directly from their Raychaudhuri equations, as these are decoupled from the Weyl components. The remaining null structure equations take the following schematic form:
\begin{align}\label{nullintro}
    \D \Ga=\W+\T+\Ga\c\Ga.
\end{align}
Recall from Section \ref{secRintro} that the Weyl components $\W$ are of one order less regular than $\F$, $\Psi$, $\A$ and $\Ga$. Thus, we cannot get the top-order estimates for $\Ga$ by considering \eqref{nullintro} as transport equations directly. To overcome this issue, we proceed as in \cite{kr} to introduce the following renormalized quantities:
\begin{align}
\begin{split}\label{renorvkpintro}
\vkp&:=\sld_1^*(-\om,\omd)-\frac{1}{2}\b,\qquad\quad\vkpb:=-\sld_1^*(\omb,\ombd)+\frac{1}{2}\bb.
\end{split}
\end{align}
where $\omd$ and $\ombd$ are defined by the solution of the following transport equations:
\begin{align}
\begin{split}\label{dfomdintro}
    \nabs_3\omd&=\frac{1}{2}\si,\qquad\quad\,\omd|_{H_{u_\infty}}=0,\\
    \nabs_4\ombd&=\frac{1}{2}\si,\qquad\quad\;\;\;\ombd|_{\Hb_0}=0.
\end{split}
\end{align}
We also define the mass aspect functions as follows:
\begin{equation}\label{murenorintro}
    \mu:=\K-\frac{1}{|u|^2}-\sdivs\eta,\qquad\quad\mub:=\K-\frac{1}{|u|^2}-\sdivs\etab.
\end{equation}
Then, the transport equations for renormalized quantities $X\in\{\vkp,\vkpb,\mu,\mub\}$ take the following schematic form:\footnote{See Proposition \ref{null} for more details.}
\begin{align}\label{renornullintro}
    \D X= \W+\D\T+\D(\Ga\c\Ga).
\end{align}
Applying \eqref{renornullintro} as a transport equation, we obtain the top order estimate for $X$, which has the same regularity as $\W$. Finally, we apply the $2D$--elliptic estimates to the system \eqref{renorvkpintro}, \eqref{murenorintro}, Codazzi equations and torsion equations to deduce the top-order estimates for $\eta,\etab,\om,\omb,\hch$, and $\hchb$. Combining these with the null structure identities \eqref{nullidentities} for the double null foliation concludes the derivation of all Ricci coefficient estimates.
\subsubsection{Trapping mechanism and charging process}
In this section, we briefly explain the trapping mechanism and charging process stated in Theorem \ref{Maintheoremintro}; see Section \ref{sectrapping} for more details. The key equations involved here are the Raychauduri's equation and one of the Maxwell equations\footnote{See Sections \ref{sec-nullstr} and \ref{sec-Maxwell} for more details of the equations.}
\begin{align*}
    \nabs_4\trch+\frac{1}{2}(\trch)^2&=-|\hch|^2-2\om\trch-\T_{44},\\
    \nabs_4\rhoF+\trch\rhoF&=\sdivs\bF+(\ze+\etab)\c\bF+2\ef\Im\left(\psi \Psi_4^\dag\right).
\end{align*}
The initial data given on $H_{u_\infty}$ satisfies
\begin{equation}\label{shortpulseintro}
    |\a,\hch,\nabs_4\bF,\bF,e_4\psi,\psi|\simeq \frac{\af}{|u_\infty|},
\end{equation}
together with the following lower bounds:
\begin{align}
        |\hch|^2+|\bF|^2+|\Psi_4|^2&\geq\frac{a}{|u_\infty|^2},\label{trappingintro}\\
        \int_{H_{u_\infty}^{(0,1)}}\Im(\psi\Om\Psi_4^\dag)&\geq a.\label{chargingintro}
    \end{align}
We also assume Minkowski initial data on $\Hb_0$. Thus, we infer on $H_{u_\infty}$
\begin{align*}
    |e_4\trch|\les |\trch|^2+|\hch|^2+|\bF|^2+|e_4\psi|^2\les\frac{a}{|u_\infty|^2},
\end{align*}
which implies the existence of a universal constant $C > 0$ such that
\begin{align*}
    |\trch(u_\infty,\ub)-\trch(u_\infty,0)|\leq \frac{Ca}{|u_\infty|^2}.
\end{align*}
Hence, for $|u_\infty| \geq Ca$, we deduce
\begin{align*}
    \trch(u_\infty, \ub) \geq \frac{2}{|u_\infty|} - \frac{C a}{|u_\infty|^2} \geq \frac{1}{|u_\infty|} > 0,
\end{align*}
showing that the initial foliation is free of trapped surfaces. Next, applying the estimates obtained in the whole region $V_*:=V(-\frac{a}{4},1)$, we easily deduce
\begin{align*}
    \nabs_3\left(|u|^2(|\hch|^2+|\bF|^2+|\Psi_4|^2)\right)\les \frac{a^\frac{3}{2}}{|u|^2}+\frac{\ef a^2}{|u|^2}.
\end{align*}
Integrating it along $e_3$ and applying \eqref{trappingintro}, we obtain for $\ef\ll 1$ and $a\gg 1$
\begin{align*}
    |\hch|^2+|\bF|^2+|\Psi_4|^2\geq\frac{3a}{4|u|^2}.
\end{align*}
Thus, we infer from Raychauduri's equation
\begin{align*}
    \nabs_4\trch\leq -|\hch|^2-|\bF|^2-|\Psi_4|^2+\lot
\end{align*}
Noticing that $\trch=\frac{8}{a}$ on $S_{-\frac{a}{4},0}$, we obtain
\begin{align*}
    \trch\left(-\frac{a}{4},1\right)\leq \trch\left(-\frac{a}{4},0\right)-\frac{a}{|u|^2}=\frac{8}{a}-\frac{12}{a}=-\frac{4}{a}<0.
\end{align*}
It is obviously that $\trchb<0$ on $S_{-\frac{a}{4},1}$. Thus, $S_{-\frac{a}{4},1}$ is a trapped surface.
\begin{rk}
The above argument shows that the initial foliation along $H_{u_\infty}^{(0,1)}$ and $\Hb_0$ is free of trapped surfaces, but a trapped surface $S_{-\frac{a}{4},1}$ forms dynamically in their future domain of dependence. Moreover, the smallness of the coupling constant $\ef$ is essential to ensure that the lower bound conditions persist throughout $V_*$.
\end{rk}

To analyze the charging process, observe
\begin{align*}
\Om\nabs_4\rhoF+\Om\trch\rhoF=\sdivs(\Om\bF)+2\ef\Im\left(\psi\Om\Psi_4^\dag\right),
\end{align*}
which implies from \eqref{chargingintro}
\begin{align*}
    \int_{S_{u_\infty,1}}\rhoF d\slg=\int_{H_{u_\infty}^{(0,1)}} 2\ef \Im\left(\psi\Om\Psi_4^\dag\right)\geq 2\ef a.
\end{align*}
Thus, the electric charge on the sphere $S_{u_\infty,1}$ satisfies
\begin{align}\label{chargeeq}
    Q(u_\infty,1)\geq \frac{1}{4\pi}\int_{S_{u_\infty,1}}\rhoF d\slg \geq \frac{\ef a}{2\pi}.
\end{align}
\begin{rk}\label{rkcharge}
The inequality \eqref{chargeeq} describes the charging process. Initially, $Q(u_\infty,0) = 0$ since $\rhoF = 0$ on $S_{u_\infty,0}$. However, \eqref{chargeeq} shows that the charge on the final sphere $S_{u_\infty,1}$ is of size $\ef a$. This implies that the spacetime becomes charged after the short pulse. Thus, the initial data in Theorem \ref{Maintheoremintro} capture both the classical trapping phenomenon in $V_*$ and a nontrivial charging process along $H_{u_\infty}$. The latter requires solving a large-data problem not only for $\hch$ and $\a$, but also for $e_4\psi$, $\psi$, and $\rhoF$.
\end{rk}
\begin{rk}
We also compute that the Hawking mass of the sphere $S_{u_\infty,1}$ is of size $a$; see Proposition \ref{massincreasing}. Combining this with Remark \ref{rkcharge}, we deduce that the charge-to-mass ratio of the final sphere $S_{u_\infty,1}$ on the initial null hypersurface is of size $\ef$, with a universal upper bound $e_0 \ll 1$.
\end{rk}
\subsection{Discussion on Extensions to Other Matter Field Models}
In this work, the gravitational and matter components are treated separately throughout the analysis; see the treatment of general matter field contributions in the Bianchi equations for the Weyl components in Section~\ref{sec-Bianchi}. The main difficulty in this paper arises from the abnormal behavior of the Maxwell–charged scalar fields near past null infinity, caused by a nontrivial charging process. Nevertheless, the semi-global existence result and the formation of a trapped surface remain valid.\\ \\
For other Einstein–matter systems, as long as the behavior of the matter fields is not worse than in the model considered here, the same framework applies. In particular, the gravitational part follows the same structure as in the Einstein–Maxwell–charged scalar field case, while the matter part can be estimated independently based on its own structure and suitable bootstrap assumptions on the Ricci coefficients.
\subsection{Structure of the paper}
\begin{itemize}
    \item In Section \ref{secpre}, we recall the fundamental notions of double null foliation and state the null structure equations and the Bianchi equations. Moreover, we deduce the pairing Bianchi equations for the components of Maxwell fields $\F$ and the gauge covariant derivatives of the complex scalar field $\psi$.
    \item In Section \ref{secmain}, we define all the norms used in this paper and present the main theorem. We then state intermediate results and prove the main theorem. The rest of the paper focuses on the proof of these intermediary results.
    \item In Section \ref{secboot}, we make bootstrap assumptions and prove the first consequences. These consequences will be used frequently in the remainder of the paper.
    \item In Section \ref{secBianchi}, we develop the $|u|^p$--weighted estimates for the general Bianchi equations. More precisely, we first prove Proposition \ref{keyintegral}, which provides a unified treatment of all linear estimates for the Bianchi pairs. We also prove Theorem \ref{Junkmanthm}, which provides a unified treatment of all nonlinear terms appearing in energy estimates and transport estimates.
    \item In Section \ref{secF}, we apply $|u|^p$--weighted estimates to the pairing Bianchi equations of the Maxwell fields and the scalar field to control the $L^2$--flux of them. We also estimate their lower order $L^2(S)$--norms by considering the Bianchi equations as transport equations.
    \item In Section \ref{secR}, we apply $|u|^p$--weighted estimates to the pairing Bianchi equations of Weyl components to control their $L^2$--flux. We also estimate their lower order $L^2(S)$--norms by considering the Bianchi equations as transport equations.
    \item In Section \ref{secO}, we first estimate the $L^2(S)$--norms of the Ricci coefficients by integrating the null structure equations along the incoming and outgoing null cones. By introducing the renormalized quantities, we estimate their top order $L^2$--flux.
    \item In Section \ref{sectrapping}, we prove the formation of a trapped surface under an additional lower bound assumption for the size of initial data. Moreover, the charge on the last sphere of the initial data have been computed under an additional condition of the lower bound of the electric current.
    \item In Section \ref{secscaling}, by exploring a scaling argument to Theorem \ref{Maintheoremintro}, we obtain Theorem \ref{scalingintro}, which constitutes a criterion for the formation of trapped surfaces of the EMCSF system in the region close to the center. This also provides a connection to the results in \cite{AnLuk,Chr}.
\end{itemize}
\subsection{Acknowledgments}
The authors would like to thank Xinliang An for many helpful discussions.
\section{Geometric setup}\label{secpre}
\subsection{Double null foliation}\label{doublenullfoliation}
We first introduce the geometric setup. We denote $(\MM,\g)$ a spacetime $\MM$ with the Lorentzian metric $\g$ and $\D$ its Levi-Civita connection. Let $u$ and $\unu$ be two optical functions on $\MM$, that is
\begin{equation*}
    \g(\grad u,\grad u)=\g(\grad\ub,\grad\ub)=0.
\end{equation*}
The spacetime $\M$ is foliated by the level sets of $u$ and $\ub$ respectively, and the functions $u,\ub$ are required to increase towards the future. We use $H_u$ to denote the outgoing null hypersurfaces which are the level sets of $u$ and use $\Hb_\ub$ to denote the incoming null hypersurfaces which are the level sets of $\ub$. We denote
\begin{equation}
    S_{u,\ub}:=H_u\cap\Hb_\ub,
\end{equation}
which are spacelike $2$--spheres. We introduce the vectorfields $L$ and $\Lb$ by
\begin{equation*}
    L:=-\grad u,\qquad\quad\Lb:=-\grad\ub.
\end{equation*}
We define a positive function $\Om$ by the formula
\begin{equation}\label{deflapse}
    \g(L,\Lb)=-2\Om^{-2},
\end{equation}
where $\Om$ is called the lapse function. We then define the normalized null pair $(e_3,e_4)$ by
\begin{equation*}
   e_3=\Om\Lb,\qquad\quad e_4=\Om L, 
\end{equation*}
and define another null pair by
\begin{equation*}
    \underline{N}=\Om e_3, \qquad\quad N=\Om e_4.
\end{equation*}
On a given $2$--sphere $S(u,\ub)$, we choose a local frame $(e_1,e_2)$, we call $(e_1,e_2,e_3,e_4)$ a null frame. Generally, throughout the paper, we use capital Latin letters $A,B,C,...$ to denote an index from $1$ to $2$, lower case Latin letters $i,j,k,...$ to denote an index from $1$ to $3$ and Greek letters $\a,\b,\ga,...$ to denote an index from $1$ to $4$, e.g., $e_A$ denotes either $e_1$ or $e_2$.\\ \\
The spacetime metric $\g$ induces a Riemannian metric $\slg$ on $S(u,\unu)$. We use $\nabs$ to denote the Levi-Civita connection of $\slg$ on $S(u,\unu)$. Using $(u,\unu)$, we introduce a local coordinate system $(u,\ub,x^A)$ on $\M$ with $e_4(x^A)=0$. In that coordinates system, the metric $\g$ takes the form:
\begin{equation}\label{metricg}
\g=-2\Om^2 (d\ub\otimes du+du\otimes d\ub)+\slg_{AB}(dx^A-\bbb^Adu)\otimes (dx^B-\bbb^Bdu),
\end{equation}
and we have 
\begin{equation*}
   \underline{N}=\pa_u+\bbb,\qquad N=\pa_{\ub},\qquad \bbb:=\bbb^A \pr_{x^A}.
\end{equation*}
\subsubsection{Ricci coefficients and Weyl components}
We recall the null decomposition of the Ricci coefficients and Weyl components of the null frame $(e_1,e_2,e_3,e_4)$ as follows:
\begin{align}
\begin{split}\label{defga}
\chib_{AB}&=\g(\D_A e_3, e_B),\qquad\quad \chi_{AB}=\g(\D_A e_4, e_B),\\
\xib_A&=\frac 1 2 \g(\D_3 e_3,e_A),\qquad\quad\,\,  \xi_A=\frac 1 2 \g(\D_4 e_4, e_A),\\
\omb&=\frac 1 4 \g(\D_3e_3 ,e_4),\qquad\quad \,\,\,\,\, \om=\frac 1 4 \g(\D_4 e_4, e_3), \\
\etab_A&=\frac 1 2 \g(\D_4 e_3, e_A),\qquad\quad\,  \eta_A=\frac 1 2 \g(\D_3 e_4, e_A),\\
\ze_A&=\frac 1 2 \g(\D_{e_A}e_4, e_3),
\end{split}
\end{align}
and
\begin{align}
\begin{split}\label{defr}
\a_{AB} &=\W(e_A, e_4, e_B, e_4) , \qquad \,\,\;\aa_{AB} =\W(e_A, e_3, e_B, e_3), \\
\b_{A} &=\frac 1 2 \W(e_A, e_4, e_3, e_4),\qquad\,\;\;\bb_{A}=\frac 1 2 \W(e_A, e_3, e_3, e_4),\\
\rho&= \frac 1 4 \W(e_3, e_4, e_3, e_4),\qquad\,\;\;\,\,\;\, \si =\frac{1}{4}{^*\W}( e_3, e_4, e_3, e_4),
\end{split}
\end{align}
where $^*\W$ denotes the Hodge dual of the Weyl tensor $\W$. The null second fundamental forms $\chi, \chib$ are further decomposed in their traces $\trch$ and $\trchb$, and traceless parts $\hch$ and $\hchb$:
\begin{align*}
\trch&:=\de^{AB}\chi_{AB},\qquad\quad \,\hch_{AB}:=\chi_{AB}-\frac{1}{2}\de_{AB}\trch,\\
\trchb&:=\de^{AB}\chib_{AB},\qquad\quad \, \hchb_{AB}:=\chib_{AB}-\frac{1}{2}\de_{AB}\trchb.
\end{align*}
We define the horizontal covariant operator $\nabs$ as follows:
\begin{equation*}
\nabs_X Y:=\D_X Y-\frac{1}{2}\chib(X,Y)e_4-\frac{1}{2}\chi(X,Y)e_3.
\end{equation*}
We also define $\nabs_4 X$ and $\nabs_3 X$ to be the horizontal projections:
\begin{align*}
\nabs_4 X&:=\D_4 X-\frac{1}{2} \g(X,\D_4e_3)e_4-\frac{1}{2} \g(X,\D_4e_4)e_3,\\
\nabs_3 X&:=\D_3 X-\frac{1}{2} \g(X,\D_3e_3)e_3-\frac{1}{2} \g(X,\D_3e_4)e_4.
\end{align*}
A tensor field $\phi$ defined on $\MM$ is called tangent to $S$ if it is a priori defined on the spacetime $\M$ and all possible contractions of $\phi$ with either $e_3$ or $e_4$ are zero. We use $\nabs_3 \phi$ and $\nabs_4 \phi$ to denote the projection to $S_{u,\ub}$ of the usual derivatives $\D_3\phi$ and $\D_4\phi$. As a direct consequence of \eqref{defga}, we have the Ricci formulae:
\begin{align}
\begin{split}\label{ricciformulas}
    \D_A e_B&=\nabs_A e_B+\frac{1}{2}\chi_{AB} e_3+\frac{1}{2}\chib_{AB}e_4,\\
    \D_A e_3&=\chib_{AB}e_B+\ze_A e_3,\\
    \D_A e_4&=\chi_{AB}e_B-\ze_A e_4,\\
    \D_3 e_A&=\nabs_3 e_A+\eta_A e_3+\xib_A e_4,\\
    \D_4 e_A&=\nabs_4 e_A+\etab_A e_4+\xi_A e_3,\\
    \D_3 e_3&=-2\omb e_3+2\xib_B e_B,\\ 
    \D_3 e_4&=2\omb e_4+2\eta_B e_B,\\
    \D_4 e_4&=-2\om e_4+2\xi_B e_B,\\
    \D_4 e_3&=2\om e_3+2\etab_B e_B.
\end{split}
\end{align}
The following identities hold for a double null foliation:
\begin{align}
\begin{split}\label{nullidentities}
    \nabs\log\Om&=\frac{1}{2}(\eta+\etab),\qquad\;\;\;\,\om=-\frac{1}{2}\nabs_4(\log\Om), \qquad\;\;\;\omb=-\frac{1}{2}\nabs_3(\log\Om),\\ 
    \eta&=\ze+\nabs\log\Om,\qquad \etab=-\zeta+\nabs\log\Om,\qquad\quad \xi=\xib=0.
\end{split}
\end{align}
see for example (6) in \cite{kr}.
\subsubsection{Null components of matter fields}\label{secmatterfields}
Let $\F$ be the \emph{Faraday tensor} in \eqref{EMCSF}. We define the following null decomposition of $\F$:
\begin{align}
\begin{split}\label{defF}
 \bF_{A}&:=\F(e_A,e_4),\qquad\quad\;\bbF_{A}:=\F(e_A,e_3),\\
 \rhoF&:=\frac{1}{2}\F(e_3,e_4),\qquad\quad\siF:=\frac{1}{2}\dual\F(e_3,e_4).
\end{split}
\end{align}
As a consequence, one has $\F(e_A,e_B)=-\slep_{AB}\siF$. Given the Faraday tensor $\F$, there exists an electromagnetic potential $\A$ which satisfies:
\begin{align*}
    \D_\mu\A_\nu-\D_\nu\A_\mu=\F_{\mu\nu}.
\end{align*}
We denote the components of electromagnetic potential $\A$ as follows:
\begin{equation}\label{dfUUbAsl}
    U:=\A_4=\A(e_4),\qquad\Ub:=\A_3=\A(e_3),\qquad\Asl_B:=\A_B=\A(e_B).
\end{equation}
Moreover, we define $\Psi$ as the gauge covariant derivative $D\psi$ and the follow null components of $\Psi$:
\begin{equation}
    \Psi_4=\Psi(e_4),\qquad \Psi_3=\Psi(e_3),\qquad \Psisl_B=\Psi(e_B).
\end{equation}
More explicitly, they are given by
\begin{equation}
    \Psi_3=e_3\psi+i\ef\Ub\psi,\qquad  \Psi_4=e_4\psi+i\ef U\psi,\qquad \Psisl_B=\nabs_B\psi+i\ef\Asl_B\psi.
\end{equation}
We also define the following \emph{Schouten tensor} $\S$ and \emph{Cotton tensor} $J$:
\begin{align}
\begin{split}\label{dfSchouten}
    \S_{\mu\nu}:&=\Ric_{\mu\nu}-\frac{1}{6}\g_{\mu\nu}\R,\\
    J_{\la\mu\nu}:&=\frac{1}{2}(\D_\mu\S_{\la\nu}-\D_\nu\S_{\la\mu}).
\end{split}
\end{align}
\subsection{Integral formulae}\label{ssecave}
We define the $S$-average of scalar functions.
\begin{df}\label{average}
Given any scalar function $f$, we denote its average and its average free part by
\begin{equation*}
    \overline{f}:=\frac{1}{|S|}\int_{S}f\,d\slg,\qquad \fc:=f-\overline{f}.
\end{equation*}
\end{df}
\begin{lem}\label{chav}
For any two scalar functions $u$ and $v$, we have
\begin{equation*}
    \overline{uv}=\overline{u}\,\overline{v}+\overline{\widecheck{u}\widecheck{v}},
\end{equation*}
and
\begin{equation*}
uv-\overline{uv}=\widecheck{u}\overline{v}+\overline{u}\widecheck{v}+\left(\widecheck{u}\widecheck{v}-\overline{\widecheck{u}\widecheck{v}}\right).
\end{equation*}
\end{lem}
\begin{proof}
    Straightforward verifications.
\end{proof}
We recall the following integral formulae, which will be used repeatedly in this paper.
\begin{lem}\label{dint}
For any scalar function $f$, the following identity hold:
\begin{align*}
    \Om e_4\left(\int_{S_{u,\ub}}fd\slg\right)&=\int_{S_{u,\ub}}\left(\Om e_4(f)+\Om\trch f\right) d\slg ,\\
    \Om e_3\left(\int_{S_{u,\ub}}fd\slg\right)&=\int_{S_{u,\ub}}\left(\Om e_3(f)+ \Om\trchb f\right) d\slg .
\end{align*}
Taking $f=1$, we obtain
\begin{equation*}
    e_4(r)=\frac{\ov{\Om\trch}}{2\Om}r,\qquad e_3(r)=\frac{\ov{\Om\trchb}}{2\Om}r,\label{e3e4r}
\end{equation*}
where $r$ is the \emph{areal radius} defined by
\begin{equation*}
    r(u,\ub):=\sqrt{\frac{|S_{u,\ub}|}{4\pi}}.
\end{equation*}
\end{lem}
\begin{proof}
See Lemma 2.26 in \cite{kl-ro} and Lemma 3.1.3 in \cite{kn}.
\end{proof}
\subsection{Hodge systems}\label{ssec7.2}
\begin{df}\label{tensorfields}
For tensor fields defined on a $2$--sphere $S$, we denote by $\sk_0:=\sk_0(S)$ the set of pairs of scalar functions, $\sk_1:=\sk_1(S)$ the set of $1$--forms and $\sk_2:=\sk_2(S)$ the set of symmetric traceless $2$--tensors.
\end{df}
\begin{df}\label{def7.2}
Given $\xi\in\sk_1$, we define its Hodge dual
\begin{equation*}
    {^*\xi}_A := \ins_{AB}\xi^B.
\end{equation*}
Given $U \in \sk_2$, we define its Hodge dual
\begin{equation*}
{^*U}_{AB}:= \ins_{AC} {U^C}_B.
\end{equation*}
\end{df}
\begin{df}
    Given $\xi,\eta\in\sk_1$ and $U,V\in\sk_2$, we denote
\begin{align*}
    \xi\c\eta&:= \de^{AB}\xi_A \eta_B,\\
    \xi\wedge\eta&:= \ins^{AB} \xi_A \eta_B,\\
    (\xi\hot\eta)_{AB}&:=\xi_A \eta_B +\xi_B \eta_A -\de_{AB}\xi\c \eta,\\
    (\xi \c U)_A&:= \de^{BC} \xi_B U_{AC},\\
    U\wedge V&:=\ins^{AB}U_{AC}V_{CB}.
\end{align*}
\end{df}
\begin{df}
For a given $\xi\in\sk_1$, we define the following differential operators:
    \begin{align*}
    \sdivs \xi&:= \de^{AB} \nabs_A\xi_B,\\
    \curls \xi&:= \ins^{AB} \nabs_A \xi_B,\\
    (\nabs\hot\xi)_{AB}&:=\nabs_A \xi_B+\nabs_B\xi_A-\de_{AB}(\sdivs\xi).
\end{align*}
\end{df}
\begin{df}\label{hodgeop}
    We define the following Hodge type operators:
    \begin{itemize}
        \item $\sld_1$ takes $\sk_1$ into $\sk_0$ and is given by:
        \begin{equation*}
            \sld_1 \xi :=(\sdivs\xi,\curls \xi),
        \end{equation*}
        \item $\sld_2$ takes $\sk_2$ into $\sk_1$ and is given by:
        \begin{equation*}
            (\sld_2 U)_A := \nabs^{B} U_{AB}, 
        \end{equation*}
        \item $\sld_1^*$ takes $\sk_0$ into $\sk_1$ and is given by:
        \begin{align*}
            \sld_1^*(f,f_*)_{A}:=-\nabs_A f +{\ins_A}^{B}\nabs_B f_*,
        \end{align*}
        \item $\sld_2^*$ takes $\sk_1$ into $\sk_2$ and is given by:
        \begin{align*}
            \sld_2^* \xi := -\frac{1}{2} \nabs \hot \xi.
        \end{align*}
    \end{itemize}
\end{df}
We have the following identities:
\begin{align}
    \begin{split}\label{dddd}
        \sld_1^*\sld_1&=-\Des_1+\mathbf{K},\qquad\qquad\sld_1\sld_1^*=-\Des_{{0}},\\
        \sld_2^*\sld_2&=-\frac{1}{2}\Des_2+\mathbf{K},\qquad\quad\sld_2\sld_2^*=-\frac{1}{2}(\Des_1+\mathbf{K}).
    \end{split}
\end{align}
where $\mathbf{K}$ denotes the Gauss curvature on $S$; see (2.2.2) in \cite{Ch-Kl}. 
\begin{df}\label{dfdkb}
We define the weighted angular derivatives $\dkb$ as follows:
\begin{align*}
    \dkb U &:= \af\sld_2 U,\qquad \forall U\in \sk_2,\\
    \dkb \xi&:=\af\sld_1 \xi,\qquad\,\,\, \forall \xi\in \sk_1,\\
    \dkb f&:= \af\sld_1^* f,\qquad \,\,\forall f\in \sk_0.
\end{align*}
We denote for any tensor $h\in\sk_k$, $k=0,1,2$,\footnote{We also use the convention that $h^{(i)}=0$ for $i<0$.}
\begin{equation*}
    h^{(0)}:=h,\qquad\quad h^{(i)}:=(h,\dkb h,...,\dkb^i h).
\end{equation*}
\end{df}
\begin{df}\label{Lpnorms}
Let $f$ be a tensor field on a $2$--sphere $S$, we denote its $L^p(S)$--norm:
\begin{equation}
    \|f\|_{L^p(S)}:= \left(\int_S |f|^p \right)^\frac{1}{p}.
\end{equation}
\end{df}
\begin{prop}\label{ellipticLp}
Assume that $S$ is an arbitrary compact 2-surface with positive bounded Gauss curvature. Then the following statements hold for all $p\in(1,+\infty)$:
\begin{enumerate}
\item Let $\phi\in\sk_0$ be a solution of $\Des\phi=f$. Then we have
\begin{align*}
    \|\nabs^2\phi\|_{L^p(S)(S)}+|u|^{-1}\|\nabs\phi\|_{L^p(S)}+|u|^{-2}\|\phi-\ov{\phi}\|_{L^p(S)}\les \|f\|_{L^p(S)}.
\end{align*}
\item Let $\xi\in\sk_1$ be a solution of $\sld_1\xi=(f,f_*)$. Then we have
\begin{align*}
    \|\nabs\xi\|_{L^p(S)}+|u|^{-1}\|\xi\|_{L^p(S)}\les\|(f,f_*)\|_{L^p(S)}.
\end{align*}
\item Let $U\in\sk_2$ be a solution of $\sld_2 U=f$. Then we have
\begin{align*}
\|\nabs U\|_{L^p(S)}+|u|^{-1}\|U\|_{L^p(S)}\les\|f\|_{L^p(S)}.
\end{align*}
\end{enumerate}
\end{prop}
\begin{proof}
See Corollary 2.3.1.1 in \cite{Ch-Kl}.
\end{proof}
\subsection{Commutation identities}
We also have the following version of the commutation identities.
\begin{lem}\label{comm}
Let $U_{A_1...A_k}$ be an $S$-tangent $k$-covariant tensor on $(\M,\g)$. Then
\begin{align*}
    [\Om\nabs_4,\nabs_B]U_{A_1...A_k}&=-\Om\chi_{BC}\nabs_CU_{A_1...A_k}+\sum_{i=1}^k \Om(\chi_{A_iB}\,\etab_C-\chi_{BC}\,\etab_{A_i}+\ins_{A_iC}{^*\beta}_B)U_{A_1...C...A_k},\\
    [\Om\nabs_3,\nabs_B]U_{A_1...A_k}&=-\Om\chib_{BC}\nabs_C U_{A_1...A_k}+\sum_{i=1}^k\Om(\chib_{A_iB}\,\eta_C-\chib_{BC}\,\eta_{A_i}+\ins_{A_iC}{^*\bb}_B)U_{A_1...C...A_k},\\
    [\Om\nabs_3,\Om\nabs_4]U_{A_1...A_k}&=4\Om^2\ze_B\nabs_B U_{A_1...A_k}+2\Om^2\sum_{i=1}^k(\etab_{A_i}\,\eta_C-\etab_{A_i}\,\eta_C+\ins_{A_iC}\si)U_{A_1...C...A_k}.
\end{align*}
\end{lem}
\begin{proof}
See Lemma 7.3.3 in \cite{Ch-Kl} and Lemma 2.23 in \cite{Shen22}.
\end{proof}
\subsection{Null structure equations}\label{sec-nullstr}
\begin{prop}\label{nulles}
We have the following null structure equations:
\begin{align}
\begin{split}\label{basicnull}
\nabs_4\eta&=-\chi\c(\eta-\etab)-\b-\frac{1}{2}\Ssl_4,\\
\nabs_3\etab&=-\chib\c(\etab-\eta)+\bb-\frac{1}{2}\Ssl_3,\\
\nabs_4\hch+(\trch)\hch&=-2\om\hch-\a,\\
\nabs_4\trch+\frac{1}{2}(\trch)^2&=-|\hch|^2-2\om\trch-\S_{44},\\
\nabs_3\hchb+(\trchb)\hchb&=-2\omb\hchb-\aa,\\
\nabs_3\trchb+\frac{1}{2}(\trchb)^2&=-|\hchb|^2-2\omb\trchb-\S_{33},\\
\nabs_4\hchb+\frac{1}{2}(\trch)\hchb&=\nabs\hot\etab+2\om\hchb-\frac{1}{2}\trchb\,\hch+\etab\hot\etab+\Sslh,\\
\nabs_3\hch+\frac{1}{2}(\trchb)\hch&=\nabs\hot\eta+2\omb\hch-\frac{1}{2}\trch\,\hchb+\eta\hot\eta+\Sslh,\\
\nabs_4\trchb+\frac{1}{2}(\trch)\trchb&=2\om\trchb+2\rho-\hch\c\hchb+2\sdivs\etab+2|\etab|^2+\Tr\S,\\
\nabs_3\trch+\frac{1}{2}(\trchb)\trch&=2\omb\trch+2\rho-\hch\c\hchb+2\sdivs\eta+2|\eta|^2+\Tr\S,
\end{split}
\end{align}
where we denoted
\begin{align*}
    (\Ssl_3)_A:=\S_{3A},\qquad (\Ssl_4)_A:=\S_{4A},\qquad \Ssl_{AB}:=\S_{AB},\qquad \Tr\S:=\g^{\mu\nu}\S_{\mu\nu}.
\end{align*}
and $\Sslh$ denotes the traceless part of $\Ssl$. We also have the Codazzi equations:
\begin{align}
\begin{split}\label{codazzi}
\sdivs\hch&=\frac{1}{2}\nabs\trch-\ze\c\left(\hch-\frac{1}{2}\trch\right)-\b+\frac{1}{2}\Ssl_4,\\ 
\sdivs\hchb&=\frac{1}{2}\nabs\trchb+\ze\c\left(\hchb-\frac{1}{2}\trchb\right)+\bb+\frac{1}{2}\Ssl_3,
\end{split}
\end{align}
the torsion equation:
\begin{equation}\label{torsion}
\curls\eta=-\curls\etab=\si-\frac{1}{2}\hch\wedge\hchb,
\end{equation}
and the Gauss equation:
\begin{equation}\label{gauss}
    \K=-\frac{1}{4}\trch\trchb+\frac{1}{2}\hch\c\hchb-\rho+\frac{1}{2}\tr\Ssl.
\end{equation}
Moreover, we have
\begin{align*}
\nabs_4\omb&=2\om\omb+\frac{3}{4}|\eta-\etab|^2-\frac{1}{4}(\eta-\etab)\c(\eta+\etab)-\frac{1}{8}|\eta+\etab|^2+\frac{1}{2}\rho+\frac{1}{4}(\tr\Ssl-\Tr\S),\\
\nabs_3\om&=2\om\omb+\frac{3}{4}|\eta-\etab|^2+\frac{1}{4}(\eta-\etab)\c(\eta+\etab)-\frac{1}{8}|\eta+\etab|^2+\frac{1}{2}\rho+\frac{1}{4}(\tr\Ssl-\Tr\S).
\end{align*}
\end{prop}
\begin{proof}
    See (2.5)--(2.7) in \cite{An24} and (2.17)--(2.34) in \cite{AMY}.
\end{proof}
\begin{prop}\label{nullrenor}
We define the following renormalized quantities:
\begin{align}
\begin{split}\label{renorvkp}
\vkp&:=\sld_1^*(-\om,\omd)-\frac{1}{2}\b,\qquad\quad\vkpb:=-\sld_1^*(\omb,\ombd)+\frac{1}{2}\bb.
\end{split}
\end{align}
where $\omd$ and $\ombd$ are defined by the solution of
\begin{align}
\begin{split}\label{dfomd}
    \nabs_3\omd&=\frac{1}{2}\si,\qquad\quad\;\omd|_{H_{u_\infty}}=0,\\
    \nabs_4\ombd&=\frac{1}{2}\si,\qquad\quad\;\;\;\ombd|_{\Hb_0}=0.
\end{split}
\end{align}
We also define the mass aspect functions as follows:
\begin{equation}\label{murenor}
    \mu:=\K-\frac{1}{|u|^2}-\sdivs\eta,\qquad\quad\mub:=\K-\frac{1}{|u|^2}-\sdivs\etab.
\end{equation}
Then, we have the following equations:
\begin{align*}
\nabs_4\mu+\trch\,\mu&=-\frac{1}{|u|^2}\trch-2(\eta+\etab)\c\b-\trch\,\sdivs\etab+\sdivs\left(\chi\c(\eta-\etab)\right)\\
&+\frac{1}{2}\hch\c\nabs\hot(\eta+\etab)+\frac{1}{2}\hch\c\etab\hot\etab+\sdivs\Ssl_4+\frac{1}{2}(3\etab+2\ze)\c\Ssl_4\\
&-\trch|\etab|^2+\frac{1}{2}\trch(\eta-\etab)\c\ze+\etab\c\hch\c(2\eta-\etab)+2\ze\c\hch\c\ze,\\
\nabs_3\mub+\trchb\,\mub&=-\frac{1}{|u|^2}\trchbt+2(\eta+\etab)\c\bb-\trchb \,\sdivs\eta+\sdivs(\chib\c(\etab-\eta))\\
&+\frac{1}{2}\hchb\c\nabs\hot(\eta+\etab)-\trchb|\eta|^2+\frac{1}{2}\hchb\c\eta\hot\eta+\sdivs\Ssl_3+\frac{1}{2}(3\eta-2\ze)\c\Ssl_3\\
&+\frac{1}{2}\trchb(\eta-\etab)\c\ze-\eta\c\hchb\c\eta+2\eta\c\hchb\c\etab+\ze \c\hchb\c \eta -\ze \c\hchb\c \etab,\\
\nabs_4\vkpb+\frac{1}{2}\trch\,\vkpb&=-\frac{1}{4}\trch\,\bb+\om\bb+\hchb\c\b-\frac{3}{2} (\etab\rho-{^*\etab}\si)+\frac{1}{2}\rho(\ze+\etab)\\
&-\frac{1}{2}\si(\dual\ze+\dual\etab)+\hch\c\sld_1^*(\omb,\ombd) +F(\ze+\etab)+\nabs F -\frac{1}{2}J_{4\bu3},\\
\nabs_3\vkp+\frac{1}{2}\trchb\,\vkp&=\frac{1}{4}\trchb\,\b-\omb\b-\hch\c\b-\frac{3}{2} (\eta\rho+{^*\eta}\si)+\frac{1}{2}\rho(-\ze+\eta)\\
&-\frac{1}{2}\si (-\dual\ze+\dual\eta)-\hchb\c \sld_1^*(-\om,\omd) +\Fb(-\ze+\eta)+\nabs \Fb -\frac{1}{2}J_{3\bu4}.
\end{align*}
where we denoted
\begin{align*}
    F:&=2\om\omb+\frac{3}{4}|\eta-\etab|^2-\frac{1}{4}(\eta-\etab)\c(\eta+\etab)-\frac{1}{8}|\eta+\etab|^2+\frac{1}{4}(\tr\Ssl-\Tr\S),\\
    \Fb:&=2\om\omb+\frac{3}{4}|\eta-\etab|^2+\frac{1}{4}(\eta-\etab)\c(\eta+\etab)-\frac{1}{8}|\eta+\etab|^2+\frac{1}{4}(\tr\Ssl-\Tr\S).
\end{align*}
\end{prop}
\begin{proof}
See (125) and (127) in \cite{kr} and Section 2.2 in \cite{LL}.
\end{proof}
\subsection{Bianchi equations}\label{sec-Bianchi}
\begin{prop}\label{bianchiequations}
We have the following Bianchi equations:
\begin{align*}
\nabs_3\a+\frac{1}{2}\trchb\,\a&=\nabs\hot\b+4\omb\a-3(\hch\rho+{^*\hch}\si)+(\ze+4\eta)\hot\b+\slJ_4,\\
\nabs_4\b+2\trch\,\b&=\sdivs\a-2\om\b+(2\ze+\etab)\c\a-J_{4\bu4},\\
\nabs_3\b+\trchb\,\b&=\nabs\rho+{^*\nabs}\si+2\omb\b+2\hch\c\bb+3(\eta\rho+{^*\eta}\si)+J_{3\bu4},\\
\nabs_4\rho+\frac{3}{2}\trch\,\rho&=\sdivs\b-\frac{1}{2}\hchb\c\a+\ze\c\b+2\etab\c\b-\frac{1}{2}J_{434},\\
\nabs_4\si+\frac{3}{2}\trch\,\si&=-\curls\b+\frac{1}{2}\hchb\c{^*\a}-\ze\c{^*\b}-2\etab\c{^*\b}-\frac{1}{2}{}^*J_{434},\\
\nabs_3\rho+\frac{3}{2}\trchb\,\rho&=-\sdivs\bb-\frac{1}{2}\hch\c\aa+\ze\c\bb-2\eta\c\bb-\frac{1}{2}J_{343},\\
\nabs_3\si+\frac{3}{2}\trchb\,\si&=-\curls\bb+\frac{1}{2}\hch\c{^*\aa}-\ze\c{^*\bb}-2\eta\c{^*\bb}+\frac{1}{2}{}^*J_{343},\\
\nabs_4\bb+\trch\,\bb&=-\nabs\rho+{^*\nabs\si}+2\om\bb+2\hchb\c\b-3(\etab\rho-{^*\etab}\si)-J_{4\bu3},\\
\nabs_3\bb+2\trchb\,\bb&=-\sdivs\aa-2\omb\bb+(2\ze-\eta)\c\aa+J_{3\bu3},\\
\nabs_4\aa+\frac{1}{2}\trch\,\aa&=-\nabs\hot\bb+4\om\aa-3(\hchb\rho-{^*\hchb}\si)+(\ze-4\etab)\hot\bb+\slJ_3,
\end{align*}
where we denoted
\begin{align*}
(\slJ_4)_{AB}&:=J_{AB4}+J_{BA4}-\frac{1}{2}J_{434}\,\slg_{AB},\qquad\;\; (\slJ_3)_{AB}:=J_{AB3}+J_{BA3}-\frac{1}{2}J_{343}\,\slg_{AB},\\
\dual J_{434}&:=(J_{AB4}-J_{BA4})\ins^{AB},\qquad\qquad\qquad\,\,\dual J_{343}:=(J_{AB3}-J_{BA3})\ins^{AB},\\
(J_{\la\bu\nu})_A&:=J_{\la A\nu},
\end{align*}
with $J_{\la\mu\nu}$ the Cotton tensor defined in \eqref{dfSchouten}. 
\end{prop}
\begin{proof}
See Section 2.3.4 in \cite{Giorgi}.
\end{proof}
\begin{lem}\label{Jformula}
The Cotton tensor $J$ has the following null components:
\begin{align*}
J_{434}&=\frac{1}{2}\nabs_3\S_{44}-2\omb\S_{44}-\frac{1}{2}\nabs_4\S_{34}-(2\eta_B+\etab_B)\S_{B4}, \\   
J_{343}&=\frac{1}{2}\nabs_4\S_{33}-2\om\S_{33}-\frac{1}{2}\nabs_3\S_{34}-(2\etab_B+\eta_B)\S_{B3}, \\ 
J_{4A4}&=\frac{1}{2}\nabs_A\S_{44}+\left(\ze_A+\frac{1}{2}\etab_A\right)  
\S_{44}-\frac{1}{2} \nabs_4\S_{4A}-\chi_{AB}\S_{B4}+\om\S_{4A},\\ 
J_{3A3}&=\frac{1}{2}\nabs_A\S_{33}+\left(-\ze_A+\frac{1}{2}\eta_A\right)  
\S_{33}-\frac{1}{2} \nabs_3\S_{3A}-\chib_{AB}\S_{B3}+\omb\S_{3A},\\ 
J_{3A4}&=\frac{1}{2}\nabs_A\S_{34}-\frac{1}{2}\etab_A\S_{34}-\frac{1}{2}\nabs_4\S_{3A}-\frac{1}{2}\chi_{AB}\S_{3B}+\om\S_{3A}-\frac{1}{2}\chib_{AB}\S_{4B}+\etab_B\S_{AB},\\
J_{4A3}&=\frac{1}{2}\nabs_A\S_{34}-\frac{1}{2}\etab_A\S_{34}-\frac{1}{2}\nabs_3\S_{4A}-\frac{1}{2}\chib_{AB}\S_{4B}+\omb\S_{4A}-\frac{1}{2}\chi_{AB}\S_{3B}+\eta_B\S_{AB},\\
J_{AB4}&=\frac{1}{2}\nabs_B\S_{A4} +\frac{1}{2}(\ze_B+\etab_B)\S_{A4}+\frac{1}{2}\etab_A\S_{B4}-\frac{1}{2}\nabs_4\S_{AB}-\frac{1}{2}\chi_{BC}\S_{AC}-\frac{1}{4}\chi_{AB}\S_{34}-\frac{1}{4}\chib_{AB}\S_{44},\\
J_{AB3}&=\frac{1}{2}\nabs_B\S_{A3} +\frac{1}{2}(\eta_B-\ze_B)\S_{A3}+\frac{1}{2}\eta_A\S_{B3}-\frac{1}{2}\nabs_3\S_{AB}-\frac{1}{2}\chib_{BC}\S_{AC}-\frac{1}{4}\chib_{AB}\S_{34}-\frac{1}{4}\chi_{AB}\S_{33},\\
{}^*J_{434}&=-\curls\,\Ssl_{4}-\ze\wedge\Ssl_{4}+\hch\wedge\Sslh,\\
{}^*J_{343}&=-\curls\,\Ssl_{3}+\ze\wedge\Ssl_{3}+\hchb\wedge\Sslh.
\end{align*}
\end{lem}
\begin{proof}
    It follows directly from \eqref{ricciformulas} and \eqref{dfSchouten}.
\end{proof}
\begin{lem}\label{Jrelations}
    We have the following identities:
    \begin{align*}
        J_{34C}+J_{43C}=2\slg^{AB}J_{ABC},\qquad J_{343}=2\slg^{AB}J_{AB3},\qquad J_{434}=2\slg^{AB}J_{AB4}.
    \end{align*}
\end{lem}
\begin{proof}
    Notice that, by the contracted second Bianchi identity, we have
    \begin{align*}
        \g^{\la\mu}J_{\la\mu\nu}&=\frac{1}{2}\g^{\la\mu}\left(\D_{\mu}\S_{\la\nu}-\D_{\nu}\S_{\la\mu}\right)\\
        &=\frac{1}{2}\left(\D^\la\left(\Ric_{\la\nu}-\frac{1}{6}\R\g_{\la\nu}\right)-\D_\nu\left(\R-\frac{2}{3}\R\right)\right)\\
        &=\frac{1}{2}\left(\D^\la\Ric_{\la\nu}-\frac{1}{6}\D_\nu\R-\frac{1}{3}\D_\nu\R\right)\\
        &=\frac{1}{2}\D^\la\left(\Ric_{\la\nu}-\frac{1}{2}\R\g_{\la\nu}\right)\\
        &=0.
    \end{align*}
    Taking $\nu=A,3,4$, we obtain
    \begin{align*}
        \slg^{AB}J_{ABC}+\g^{34}J_{34C}+\g^{34}J_{43C}&=0,\\
        \slg^{AB}J_{AB3}+\g^{34}J_{343}+\g^{34}J_{433}&=0,\\
        \slg^{AB}J_{AB4}+\g^{34}J_{344}+\g^{34}J_{434}&=0.
    \end{align*}
    This concludes the proof of Lemma \ref{Jrelations}.
\end{proof}
\begin{cor}\label{J343J434}
    We have the following expressions:
        \begin{align*}
        J_{434}&=\sdivs\Ssl_{4}+(\ze+2\etab)\c\Ssl_4-\nabs_4(\tr\Ssl)-\hch\c\Ssl-\frac{1}{2}\trch(\tr\Ssl+\S_{34})-\frac{1}{2}\trchb\S_{44},\\
        J_{343}&=\sdivs\Ssl_{3}+(2\eta-\ze)\c\Ssl_3-\nabs_3(\tr\Ssl)-\hchb\c\Ssl-\frac{1}{2}\trchb(\tr\Ssl+\S_{34})-\frac{1}{4}\trch\S_{33}.
    \end{align*}
\end{cor}
\begin{proof}
It follows directly from Lemmas \ref{Jformula} and \ref{Jrelations}.
\end{proof}
\begin{prop}\label{BianchiKcsic}
We introduce the following renormalized quantities\footnote{These renormalized quantities are used in \cite{An24,AnAth,AMY,LL}.}
\begin{align}
\begin{split}\label{dfKcsic}
    \bt_A&:=\b_A-\frac{1}{2}\S_{4A},\qquad\qquad\bbt_A:=\bb_A+\frac{1}{2}\S_{3A},\\ \Kc&:=\K-\frac{1}{|u|^2},\qquad\qquad\quad\;\,\sic:=\si-\frac{1}{2}\hch\wedge\hchb,\\
    \trcht&:=\trch-\frac{2}{|u|},\qquad\qquad\, \trchbt:=\trchb+\frac{2}{\Om|u|}.
\end{split}
\end{align}
    We have the following Bianchi equations:
\begin{align*}
\nabs_3\a+\frac{1}{2}\trchb\,\a&=\nabs\hot\bt+4\omb\a-3(\hch\rho+{^*\hch}\si)+(\ze+4\eta)\hot\bt\\
&+\slJ_4+\frac{1}{2}\nabs\hot\Ssl_4+\frac{1}{2}(\ze+4\eta)\hot\Ssl_4,\\
\nabs_4\bt+2\trch\,\bt&=\sdivs\a-2\om\bt+(2\ze+\etab)\c\a\\
&-\frac{1}{2}\nabs\S_{44}-\frac{1}{2}\left(2\ze+\etab\right)\S_{44}-\hch\c\Ssl_4,\\
\nabs_3\b+\trchb\,\b&=\sld_1^*(\Kc,\sic)-\frac{1}{4}\nabs(\trch\trchb)+\frac{1}{2}\nabs(\hch\c\hchb)+\frac{1}{2}{^*\nabs}(\hch\wedge\hchb)+2\omb\b+2\hch\c\bb\\
&+3(\eta\rho+{^*\eta}\si)+J_{3\bu4},\\
\nabs_4\K+\trch\,\K&=-\sdivs\b-(2\etab+\ze)\c\b-\frac{1}{2}\trch\sdivs\etab-\frac{1}{2}\trch|\etab|^2\\
&+\frac{1}{2}\hch\c\nabs\hot\etab+\frac{1}{2}\hch\c\etab\hot\etab+\frac{1}{2}\sdivs\Ssl_4+\frac{1}{2}(2\etab+\ze)\c\Ssl_4,\\
\nabs_4\sic+\frac{3}{2}\trch\sic&=-\curls\b-\frac{1}{2}\hch\wedge\nabs\hot\etab-(\ze+2\etab)\wedge\b-\hch\wedge\etab\hot\etab\\
&+\frac{1}{2}\curls\Ssl_4+\frac{1}{2}\ze\c\Ssl_4-\hch\wedge\Sslh,\\
\nabs_3\Kc+\frac{3}{2}\trchb\Kc&=\sdivs\bb+(2\eta-\ze)\c\bb-\frac{1}{2}\trchb|\eta|^2-\trchbt\frac{1}{|u|^2}+\frac{1}{2}\trchb\mu\\
&+\frac{1}{2}\hchb\c\nabs\hot\eta+\frac{1}{2}\hchb\c\eta\hot\eta+\frac{1}{2}\sdivs\Ssl_3+\frac{1}{2}(2\eta-\ze)\c\Ssl_3,\\
\nabs_3\sic+\frac{3}{2}\trchb\,\sic&=-\curls\bb-\frac{1}{2}\hchb\wedge\nabs\hot\etab-(2\eta-\ze)\wedge\bb-\hchb\wedge\eta\hot\eta\\
&-\frac{1}{2}\curl\Ssl_3+\frac{1}{2}\ze\wedge\Ssl_3+\hchb\wedge\Sslh,\\
\nabs_4\bb+\trch\,\bb&=-\sld_1^*(\Kc,-\sic)+\frac{1}{4}\nabs(\trch\trchb)-\frac{1}{2}\nabs(\hch\c\hchb)+\frac{1}{2}\nabs(\hch\wedge\hchb)+2\om\bbt+2\hchb\c\bt\\
&-3(\etab\rho-{^*\etab}\si)-J_{4\bu3}.
\end{align*}
\end{prop}
\begin{proof}
    See (2.14)--(2.19) in \cite{An24} and Section 2.2 in \cite{LL}.
\end{proof}
\begin{rk}\label{noaark}
By performing renormalized Bianchi equations, the estimate for $\a$ completely decouples. At the same time, we also do not need the estimate for $\aa$, which is hard to control based on the presence of $\S_{33}$ in the R.H.S. of the equations of $\nabs_3\bb$. See also \cite{AnLuk} for this idea of eliminating $\aa$ from the system.
\end{rk}
\subsection{Maxwell equations}\label{sec-Maxwell}
\begin{prop}\label{Maxwellequations}
The Maxwell equations take the following form:
\begin{align*}
\nabs_3\bF+\frac{1}{2}\trchb\bF&=\nabs\rhoF-\dual\nabs\siF+2\omb\bF+2\rhoF\eta-2\siF\dual\eta+\hch\c\bbF -2\ef\Im\left(\psi\Psisl^\dag\right),\\
\nabs_4\rhoF+\trch\rhoF&=\sdivs\bF+(\ze+\etab)\c\bF+2\ef\Im\left(\psi \Psi_4^\dag\right),\\
\nabs_4\siF+\trch\siF&=\curls\bF+(\ze+\etab)\wedge\bF,\\
\nabs_3\rhoF+\trchb\rhoF&=-\sdivs\bbF-(-\ze+\eta)\c\bbF-2\ef\Im\left(\psi\Psit_3^\dag\right),\\
\nabs_3\siF+\trchb\siF&=\curls\bbF+(-\ze+\eta)\wedge\bbF,\\
\nabs_4\bbF+\frac{1}{2}\trch\bbF&=-\nabs\rhoF-\dual\nabs\siF+2\om\bbF-2\rhoF\etab-2\siF\dual\etab+\hchb\c\bF -2\ef\Im\left(\psi\Psisl^\dag\right).
\end{align*}
\end{prop}
\begin{proof}
    See Appendix \ref{AppendixM}.
\end{proof}
\begin{cor}\label{corteukolsky}
    We have the following Teukolsky equation for $\bF$:
    \begin{align*}
        \left(\nabs_3+\frac{1}{2}\trchb\right)(\nabs_4\bF)&=\nabs(\sdivs\bF)-\dual\nabs(\curls\bF)\\
        &+\chi\c\nabs\rhoF+\nabs\left(-\trch\rhoF+(\ze+\etab)\c\bF\right)\\
        &-\chi\c\dual\nabs\siF-\dual\nabs\left(-\trch\siF+(\ze+\etab)\wedge\bF\right)\\
        &+\nabs_4\left(2\omb \bF +2\rhoF \eta -2\siF\dual\eta +\hch\c\bbF)\right)\\
        &-2\om \nabs_3\bF +2\omb \nabs_4\bF+4(\ze\c\nabs)\bF\\
        &+2(\etab\c\bF)\eta -2(\eta\c\bF)\etab +2\si \dual\bF\\
        &-2\ef\nabs_4(\Im(\psi\Psisl^\dag))+2\ef\nabs(\Im(\psi \Psi_4^\dag)),\\
        \left(\nabs_4+\frac{1}{2} \trch\right)\sdivs\bF&=\sdivs(\nabs_4\bF)-\hch\c\nabs\bF+\trch(\etab\c\bF)-\etab\c\chi\c\bF+\b\c\bF,\\
        \left(\nabs_4+\frac{1}{2}\trch\right)\curls\bF&=\curls(\nabs_4\bF)-\hch\c\nabs\dual\bF -\dual\etab\c\hch\c \bF-\dual\b\c\bF
    \end{align*}
\end{cor}
\begin{proof}
    This follows from Propositions \ref{comm} and \ref{Maxwellequations}.
\end{proof}
\subsection{Electromagnetic Potential equations}
\begin{prop}\label{lem:FdAeq}
Let $U$, $\Ub$ and $\slA$ defined in \eqref{dfUUbAsl}. Then, the null components of $\F$ can be expressed in terms of the null components of $\A$ as follows:
\begin{align*}
    \rhoF&=\frac{1}{2}\nabs_3U-\frac{1}{2}\nabs_4\Ub-\omb U+\om\Ub +(\etab-\eta)\c \Asl,\\
    \siF&=-\curls\Asl,\\
    \bF&=\nabs U-\nabs_4\Asl-\chi\c\Asl+U(\ze+\etab),\\
    \bbF&=\nabs \Ub-\nabs_3\Asl-\chib\c\Asl+\Ub(-\ze+\eta).
\end{align*}
\end{prop}
\begin{proof}
See Appendix \ref{AppendixA}
\end{proof}
\begin{cor}\label{cordA=F}
In the case of $U=0$, we have
\begin{align*}
    \rhoF&=-\frac{1}{2}\nabs_4\Ub+\om\Ub+(\etab-\eta)\c\Asl,\\
    \siF&=-\curls\Asl,\\
    \bF&=-\nabs_4\Asl-\frac{1}{2}\trch\Asl-\hch\c\Asl,\\
    \bbF&=\nabs \Ub-\nabs_3\Asl-\frac{1}{2}\trchb\Asl-\hchb\c\Asl+\Ub(-\ze+\eta).
\end{align*}
\end{cor}
\begin{proof}
    It follows directly from Proposition \ref{lem:FdAeq}.
\end{proof}
\subsection{Wave equation for complex scalar field}
\begin{prop}\label{waveequation}
    We have the following equations for the complex scalar field:
    \begin{align*}
        \nabs_3(\Psi_4)+\frac{1}{2}\trchb\Psi_4-2\omb\Psi_4+i\ef \Ub\Psi_4&=\sdivs\Psisl+2\eta\c\Psisl+i\ef \Asl\c\Psisl-\frac{1}{2}\trch\Psi_3+i\ef\rhoF\psi,\\
        \nabs_4\Psisl+\frac{1}{2}\trch \Psisl+\hch\c\Psisl+i\ef U\Psisl&=\nabs(\Psi_4)+\frac{\eta+\etab}{2}\Psi_4+i\ef\Asl\Psi_4-i\ef \bF \psi,\\
        \nabs_3\Psisl+\frac{1}{2}\trchb \Psisl+\hchb\c\Psisl+i\ef \Ub\Psisl&=\nabs(\Psi_3)+\frac{\eta+\etab}{2}\Psi_3+i\ef \Asl \Psi_3-i\ef \bbF \psi,\\
        \nabs_4(\Psi_3)+\frac{1}{2}\trch\Psi_3-2\om\Psi_3+i\ef U\Psi_3&=\sdivs\Psisl+2\etab\c\Psisl+i\ef\Asl\c\Psisl-\frac{1}{2}\trchb\Psi_4-i\ef\rhoF\psi.
    \end{align*}
    We also have
    \begin{align*}
        \nabs_3\Psisl+\trchb \Psisl+\hchb\c\Psisl+i\ef \Ub\Psisl&=\nabs(\Psit_3)+\frac{\eta+\etab}{2}\Psit_3+\frac{1}{2}\trchbt\Psisl+i\ef \Asl \Psit_3-i\ef \bbF \psi,\\
        \nabs_4(\Psit_3)+\frac{1}{2}\trch\Psit_3-2\om\Psit_3+i\ef U\Psit_3&=\sdivs\Psisl+2\etab\c\Psisl+i\ef\Asl\c\Psisl-\frac{1}{2}\trchbt\Psi_4-\frac{\trch}{2\Om|u|}\psi-i\ef\rhoF\psi.
    \end{align*}
    where we denoted
    \begin{equation}\label{dfPsit3}
    \Psit_3:=|u|^{-1}e_3(|u|\psi)+i\ef\Ub\psi.
    \end{equation}
\end{prop}
\begin{proof}
See Appendix \ref{AppendixPsi}.
\end{proof}
\begin{rk}\label{AAGmethod}
The motivation for introducing $\Psit_3$ in \eqref{dfPsit3} is to eliminate the radial solution $|u|^{-1}$ of the wave equation. Specifically, $\Psit_3$ exhibits better decay in $|u|$ compared to $\Psi_3$. Similar techniques are used in \cite{AAG}, where the radial field $r^{-1}$ is eliminated using the vector field $re_4(r\c)$.
\end{rk}
\subsection{Hawking mass and electric charge}
We recall the following definitions of Hawking mass and electric charge.
\begin{df}\label{Hawkingmass}
Let $S$ be a $2$--sphere. The Hawking mass of $S$ is defined by
\begin{align}\label{Hawkingdf}
    \frac{2m(S)}{r}:=1+\frac{1}{16\pi}\int_{S}\trch\trchb\, d\slg.
\end{align}
\end{df}
\begin{df}\label{dfcharge}
Let $S$ be a $2$--sphere. The electric charge enclosed by $S$ is defined by
    \begin{align*}
        Q(S):=\frac{1}{4\pi}\int_{S} \rhoF\, d\slg.
    \end{align*}
\end{df}
\subsection{Signatures and scale invariant norms}\label{secsignatures}
As taken in \cite{Chr}, we prescribe $\hch$ such that
\begin{align*}
    |\hch|\simeq \frac{\af}{|u_\infty|}\qquad\quad \mbox{ along }\; H_{u_\infty}^{(0,1)}.
\end{align*}
Following the same procedures in Chapter 2 of \cite{Chr}, we obtain the following estimates on $H_{u_\infty}^{(0,1)}$:
\begin{align*}
    |\a|&\les\frac{\af}{|u_\infty|},\qquad |\b|\les\frac{\af}{|u_\infty|^2}, \qquad |\rho,\si|\les\frac{a}{|u_\infty|^3},\qquad |\bb|\les\frac{a}{|u_\infty|^4},\qquad |\aa|\les\frac{a^\frac{3}{2}}{|u_\infty|^5},\\
    |\om,\trch|&\les\frac{1}{|u_\infty|},\qquad\qquad\qquad\qquad\,\,|\trchbt,\omb|\les\frac{a}{|u_\infty|^3},\qquad\qquad\qquad\quad\;\,\,|\eta,\etab,\hchb|\les\frac{\af}{|u_\infty|^2}.
\end{align*}
We also assume that on $H_{u_\infty}^{(0,1)}$:
\begin{align*}
    |\nabs_4\bF,\bF,\Psi_4,\psi|\simeq\frac{\af}{|u_\infty|}.
\end{align*}
Applying Propositions \ref{Maxwellequations} and \ref{waveequation} and Corollary \ref{cordA=F}, proceeding as in Chapter 2 of \cite{Chr}, we obtain the following estimates on $H_{u_\infty}^{(0,1)}$:
\begin{align*}
    |\rhoF,\siF,\Psisl|&\les\frac{\af+a\ef}{|u_\infty|^2},\qquad\quad|\bbF,\Psit_3|\les\frac{\af+\ef a+\ef^2 a^\frac{3}{2}}{|u_\infty|^3},\\
    |\Asl|&\les\frac{\af}{|u_\infty|},\qquad\qquad\qquad\quad|\Ub|\les\frac{\af+a\ef}{|u_\infty|^2}.
\end{align*}
It will be difficult to treat these weights $|u|$ and $a$ term by term. We hope to design a \emph{scale-invariant} $L^\infty_{sc}(S_{u,\ub})$ with $|u|$ and $a$ weights built in, such that for most geometric quantities $\phi$, we have
\begin{align*}
    \|\phi\|_{L^\infty_{sc}(S_{u,\ub})}\les 1.
\end{align*}
To achieve this, we use the signature $s_2$ introduced by An in \cite{An}. First, we need to relax the above estimates for $\bb$ and $\aa$ as follows:
\begin{align*}
    |\bb|\les\frac{a}{|u_\infty|^4}\les\frac{a^\frac{3}{2}}{|u_\infty|^4},\qquad\qquad|\aa|\les\frac{a^\frac{3}{2}}{|u_\infty|^5}\les\frac{a^2}{|u_\infty|^5},
\end{align*}
and keeping the other estimates for now, we find a systematical way to define $L^\infty_{sc}(S_{u,\ub})$.
\begin{df}\label{signature}
We first introduce the signature for decay rates to $\phi$, we assign signatures $s_2(\phi)$ according to the rule:
\begin{align*}
    s_2(\phi)=\frac{1}{2} N_A(\phi)+N_3(\phi)-1,
\end{align*}
with $N_3(\phi)$ is the number of times $e_3$ appears in the definition of $\phi$. Similarly, we define $N_A(\phi)$ where $A=1,2$.
\end{df}
\begin{rk}
Following Definition \ref{signature}, we have the following signature table:
\begin{center}
\begin{tabular}{|c|c|c|c|c|c|c|c|c|c|c|c|c|c|c|c|}
\hline
{} & $\a$ & $\b$ & $\rho$ & $\si$ & $\bb$ & $\aa$ & $\chi$ & $\om$& $\Om$ & $\ze$ & $\eta$ & $\etab$ & $\chib$ & $\omb$ & $\slg$ \\ 
$s_2$ & $0$ & $0.5$ & $1$ & $1$ & $1.5$ & $2$ & $0$ & $0$ & $0$ & $0.5$ & $0.5$ & $0.5$ & $1$ & $1$ & $0$ \\
\hline
\end{tabular}
\end{center}
We have the following signature table for the quantities of matter fields:
\begin{center}
\begin{tabular}{|c|c|c|c|c|c|c|c|c|c|c|c|c|c|c|c|}
\hline
{} & $\bF$ & $\rhoF$ & $\siF$ & $\bbF$ & $\psi$ & $\Psi_4$& $\Psisl$ & $\Psi_3$ & U &$\Ub$ & $\Asl$ & $\ef$ \\ 
$s_2$ & $0$ & $0.5$ & $0.5$ & $1$ & $0$ & $0$ & $0.5$ & $1$ & $-0.5$ & $0$ & $0.5$ & $0.5$\\
\hline
\end{tabular}
\end{center}
We also have the following signature table for the renormalized quantities:
\begin{center}
\begin{tabular}{|c|c|c|c|c|c|c|c|c|c|c|c|c|c|c|c|}
\hline
{} & $\mu$ & $\mub$ & $\vkp$ & $\vkpb$ & $\K$ & $\sic$\\ 
$s_2$ & $1$ & $1$ & $0.5$ & $1.5$ & $1$ & $1$\\
\hline
\end{tabular}
\end{center}
Moreover, we have
\begin{equation}
    s_2(\nabs_4\phi)=s_2(\phi),\qquad s_2(\nabs\phi)=s_2(\phi)+\frac{1}{2},\qquad s_2(\nabs_3\phi)=s_2(\phi)+1.
\end{equation}
\end{rk}
For any horizontal tensorfield $\phi$ with signature $s_2(\phi)$, we define the scale invariant norms:
\begin{align}
    \begin{split}\label{dfsc}
        \|\phi\|_{L^\infty_{sc}(S_{u,\ub})}&:=a^{-s_2(\phi)}|u|^{2s_2(\phi)+1}\|\phi\|_{L^\infty(S_{u,\ub})},\\
        \|\phi\|_{L^2_{sc}(S_{u,\ub})}&:=a^{-s_2(\phi)}|u|^{2s_2(\phi)}\|\phi\|_{L^2(S_{u,\ub})},\\
        \|\phi\|_{L^1_{sc}(S_{u,\ub})}&:=a^{-s_2(\phi)}|u|^{2s_2(\phi)-1}\|\phi\|_{L^1(S_{u,\ub})}.
    \end{split}
\end{align}
For convenience, we also define the scale invariant norms along null hypersurfaces
\begin{align}
\begin{split}\label{phifluxsc}
        \|\phi\|_{L^2_{sc}(\cuv)}^2&:=\int_0^\ub\|\phi\|_{L^2_{sc}(S_{u,\ub'})}^2d\ub',\\
        \|\phi\|_{L^2_{sc}(\ucuv)}^2&:=\int_{u_\infty}^u\frac{a}{|u'|^2}\|\phi\|_{L^2_{sc}(S_{u',\ub})}^2du'.
\end{split}
\end{align}
As an immediate consequence of \eqref{dfsc} and H\"older inequality, we have the following proposition.
\begin{prop}\label{Holder}
    We have the following inequalities:
    \begin{align*}
        \|\phi_1\c\phi_2\|_{L^2_{sc}(S_{u,\ub})}&\leq\frac{1}{|u|}\|\phi_1\|_{L^\infty_{sc}(S_{u,\ub})}\|\phi_2\|_{L^2_{sc}(S_{u,\ub})},\\
        \|\phi_1\c\phi_2\|_{L^1_{sc}(S_{u,\ub})}&\leq\frac{1}{|u|}\|\phi_1\|_{L^\infty_{sc}(S_{u,\ub})}\|\phi_2\|_{L^1_{sc}(S_{u,\ub})},\\
        \|\phi_1\c\phi_2\|_{L^1_{sc}(S_{u,\ub})}&\leq\frac{1}{|u|}\|\phi_1\|_{L^2_{sc}(S_{u,\ub})}\|\phi_2\|_{L^2_{sc}(S_{u,\ub})}.
    \end{align*}
\end{prop}
\begin{rk}
   Note that in the study region we have $\frac{1}{|u|}\leq\frac{4}{a}\ll 1$. This means if all terms are normal, the nonlinear terms could be treated as lower order terms.
\end{rk}
\section{Main theorem}\label{secmain}
In the sequel, we always denote
\begin{align}
    V:=V(u,\ub):=\{(u',\ub')\in\left[u_\infty,u\right]\times[0,\ub]\},\qquad V_*:=V\left(-\frac{a}{4},1\right).\label{Vdf}
\end{align}
\subsection{Key parameters}\label{smallconstant}
Before starting our main theorem, we first summarize the properties of the following constants, which will be used throughout this paper.
\begin{itemize}
    \item The coupling constant $\ef$ satisfies $0<\ef<e_0\ll 1$.
    \item The size of the short pulse $a$ satisfies $a\gg 1$.
    \item An auxiliary constant $b$ for the purpose of bootstrap argument satisfies \begin{align}\label{bootsize}
        b\gg 1,\qquad\quad b\ll a^\frac{1}{6},\qquad\quad b^3 e_0 \ll 1.
    \end{align}
\end{itemize}
\subsection{Fundamental norms}\label{secnorms}
For any quantity $X$, we define\footnote{Here, we use the convention that $X^{(q)}=0$ for $q<0$.}
\begin{align}\label{dfXq}
    X^{(q)}:=\bigcup_{i=0}^q\left\{(\af\nabs)^iX\right\},\qquad \forall\; q\in\mathbb{Z}.
\end{align}
\subsubsection{\texorpdfstring{$\TT$}{}--norms (matter fields)}
We define the following $L^2$--flux norms:
\begin{align*}
    \mf_i(u,\ub)&:=\afd\|\bF^{(i)}\|_{L^2_{sc}(\cuv)}+\afd\|(\nabs_4\bF)^{(i-1)}\|_{L^2_{sc}(\cuv)},\\
    &+b\afd\|(\rhoF,\siF)^{(i)}\|_{L^2_{sc}(\cuv)},\\
    \uf_i(u,\ub)&:=\afd\|(\bF,\rhoF,\siF)^{(i)}\|_{L^2_{sc}(\ucuv)}+ b\afd\|\bbF^{(i)}\|_{L^2_{sc}(\ucuv)},\\
    \pf_i(u,\ub)&:=\afd\|\Psi_4^{(i)}\|_{L^2_{sc}(\cuv)}+b\afd\|\Psisl^{(i)}\|_{L^2_{sc}(\cuv)},\\
    \pfb_i(u,\ub)&:=\afd\|\Psisl^{(i)}\|_{L^2_{sc}(\ucuv)}+b\afd\|\Psit_3^{(i)}\|_{L^2_{sc}(\ucuv)},\\
    \AAb_i(u,\ub)&:=\afd\|\Asl^{(i)}\|_{L^2_{sc}(\ucuv)}+\afd\|\Ub^{(i)}\|_{L^2_{sc}(\ucuv)}.
\end{align*}
We also define the following $L^2_{sc}(S_{u,\ub})$--norms:
\begin{align*}
    \mo_i[T](u,\ub)&:=\afd\|(\bF,\Psi_4,\Asl,\Ub)^{(i)}\|_{L^2_{sc}(S_{u,\ub})}+b\afd\|(\rhoF,\siF,\bbF,\Psisl,\Psit_3)^{(i)}\|_{L^2_{sc}(S_{u,\ub})}\\
    &+\afd\|(\nabs_4\bF)^{(i-1)}\|_{L^2_{sc}(S_{u,\ub})}+\afd\|\psi\|_{L^2_{sc}(S_{u,\ub})}.
\end{align*}
Finally, we denote
\begin{align*}
\TT&:=\sup_{V_*}\sum_{i=0}^4\big(\mf_i(u,\ub)+\uf_i(u,\ub)+\pf_i(u,\ub)+\pfb_i(u,\ub)+\AAb_i(u,\ub)\big)+\sum_{i=0}^3\OO_i[T](u,\ub).
\end{align*}
\subsubsection{\texorpdfstring{$\mw$}{}--norms (Weyl components)}\label{secRnorms}
We define
\begin{align*}
    \mw_i(u,\ub)&:=\afd\|\a^{(i)}\|_{L^2_{sc}(\cuv)}+\|(\b,\rho,\si)^{(i)}\|_{L^2_{sc}(\cuv)},\\
    \uw_i(u,\ub)&:=\afd\|\b^{(i)}\|_{L^2_{sc}(\ucuv)}+\|(\rho,\si,\bb)^{(i)}\|_{L^2_{sc}(\ucuv)}.
\end{align*}
We also define
\begin{align*}
    \mo_i[W](u,\ub):=\afd\|\a^{(i)}\|_{L^2_{sc}(S_{u,\ub})}+\|(\b,\rho,\si,\bb)^{(i)}\|_{L^2_{sc}(S_{u,\ub})}.
\end{align*}
Finally, we denote
\begin{align*}
\mw:=\sup_{V_*}\left(\sum_{i=0}^3\left(\mw_i(u,\ub)+\uw_i(u,\ub)\right)+\sum_{i=0}^2\OO_i[W](u,\ub)\right).
\end{align*}
\subsubsection{\texorpdfstring{$\mg$}{}--norms (Ricci coefficients)}\label{secOnorms}
We define
\begin{align*}
    \mg_i(u,\ub)&:=\afd\|(\hch,\etab,\om,\omd)^{(i)}\|_{L^2_{sc}(\cuv)}+\|(\mub,\vkp)^{(i-1)}\|_{L^2_{sc}(\cuv)},\\
    \ug_i(u,\ub)&:=\frac{\af}{|u|}\|\hchb^{(i)}\|_{L^2_{sc}(\ucuv)}+\afd\|(\eta,\omb,\ombd)^{(i)}\|_{L^2_{sc}(\ucuv)}\\
    &+\|\vkpb^{(i-1)}\|_{L^2_{sc}(\ucuv)}+\frac{|u|}{a}\|\muc^{(i-1)}\|_{L^2_{sc}(\ucuv)}.
\end{align*}
We also define\footnote{Notice that the top-order estimates for $\trcht$ and $\trchbt$ are worse than their lower order estimates.}
\begin{align*}
    \mo_i[G](u,\ub)&:=\afd\|\hch^{(i-1)}\|_{L^2_{sc}(S_{u,\ub})}+\frac{\af}{|u|}\|\hchb^{(i-1)}\|_{L^2_{sc}(S_{u,\ub})}\\
    &+\|(\eta,\etab,\om,\omd,\omb,\ombd,\trchbt)^{(i-1)}\|_{L^2_{sc}(S_{u,\ub})}+\frac{|u|}{a}\|\trcht^{(i-1)}\|_{L^2_{sc}(S_{u,\ub})}\\
    &+\afd\|\trcht^{(i)}\|_{L^2_{sc}(S_{u,\ub})}+|u|^{-\frac{1}{2}}\|\trchbt^{(i)}\|_{L^2_{sc}(S_{u,\ub})}.
\end{align*}
Finally, we denote
\begin{align*}
    \mg:=\sup_{V_*}\sum_{i=0}^4\left(\mg_i(u,\ub)+\ug_i(u,\ub)+\mo_i[G](u,\ub)\right).
\end{align*}
\subsubsection{\texorpdfstring{$\II_{(0)}$}{}--norms (Initial data)}\label{initialO0}
We introduce the following norms on $H_{u_\infty}$:
\begin{align*}
\mo_{(0)}&:=\sup_{H_{u_\infty}}\sum_{i=0}^4\left(\mo_i[G](u_\infty,\ub)+\mo_i[G](u_\infty,\ub)+\mo_i[G](u_\infty,\ub)\right),\\
\mr_{(0)}&:=\sup_{H_{u_\infty}}\sum_{i=0}^4\left(\mf_i(u,\ub)+\PP_i(u,\ub)+\mw_i(u,\ub)+\mg_i(u,\ub)\right).
\end{align*}
Then, we denote
\begin{align*}
    \II_{(0)}:=\mo_{(0)}+\mr_{(0)},
\end{align*}
which controls all the quantities on $H_{u_\infty}^{(0,1)}$.
\subsection{Statement of the main theorem}\label{secmainstate}
\begin{thm}\label{maintheorem}
There exists a sufficiently large constant $a_0>0$ and a sufficiently small constant $e_0>0$ such that the following holds. For $a>a_0$ and $0<\ef<e_0$, with an initial data that satisfies: \begin{itemize}
        \item the following estimate hold along $u=u_\infty$:
        $$
        \sum_{i\leq 10,k\leq 4}\afd|u_\infty|\left\|\nabs_4^k(|u_\infty|\nabs)^i\left(\hch,\bF,\psi\right)\right\|_{L^\infty(S_{u_\infty,\ub})}\leq 1,
        $$
        \item Minkowskian initial data along $\ub=0$.
    \end{itemize}
    Einstein--Maxwell--charged scalar field system \eqref{EMCSF} admits a unique smooth solution in $V_*$ which satisfies:\footnote{Here, $A\ll B$ means that $CA<B$ where $C$ is the largest universal constant among all the constants involved in the proof via $\les$.}
    \begin{equation}\label{finalestimates}
        \mg\les 1,\qquad\quad \mw\les 1,\qquad\quad \TT\les 1.
    \end{equation}
    The estimate \eqref{finalestimates} implies, in particular, on every sphere $S_{u,\ub}\subseteq V_*$:
    \begin{align*}
    |\hch,\a|&\les\frac{a^\frac{1}{2}}{|u|},\qquad |\b|\les\frac{a^\frac{1}{2}}{|u|^2}, \qquad |\rho,\si|\les\frac{a}{|u|^3},\qquad |\bb|\les\frac{a^\frac{3}{2}}{|u|^4},\qquad |\aa|\les\frac{a^2}{|u|^5},\\
    |\om,\trch,\log\Om|&\les\frac{1}{|u|},\qquad\qquad\qquad\quad\;\;|\trchbt,\omb|\les\frac{a}{|u|^3},\qquad\quad\;|\eta,\etab,\ze,\hchb,\Psi_3,\bbb|\les\frac{a^\frac{1}{2}}{|u|^2},\\
    |\bF,\Psi_4,\psi,\Asl|&\les\frac{\af}{|u|},\qquad\qquad|\rhoF,\siF,\Psisl,\Ub|\les\frac{a}{|u|^2},\qquad\qquad\quad\;\;\;\,|\bbF,\Psit_3|\les\frac{a^\frac{3}{2}}{|u|^3}.
\end{align*}
\end{thm}
The proof of Theorem \ref{maintheorem} is given in Section \ref{proofmain}. It hinges on three basic theorems stated in Section \ref{intermediate}, concerning estimates for the norms $\TT$, $\mw$ and $\mg$.
\subsection{Main intermediate results}\label{intermediate}
\begin{thm}\label{M1}
Under the assumptions of Theorem \ref{maintheorem} and let $V$ be the region defined in \eqref{Vdf}. Moreover, we assume that the norms in $V$ satisfy the following bounds:
\begin{equation}
   \II_{(0)}\les 1,\qquad\TT\leq\bo,\qquad \mw\leq\bo,\qquad \mg\leq\bo.
\end{equation}
Then, we have
\begin{equation}
    \TT\les 1.
\end{equation}
\end{thm}
Theorem \ref{M1} is proved in Section \ref{secF}. The $L^2$--flux estimates are made by applying the $|u|^p$--weighted estimates to the pairing Bianchi equations of $\F$ and $\Psi$ in Propositions \ref{Maxwellequations} and \ref{waveequation}. Then, the electromagnetic potentials $\A$ are estimated by considering the equations in Corollary \ref{cordA=F} as transport equations. Moreover, the lower order $L^2(S_{u,\ub})$--norms of matter fields are controlled by considering the Bianchi equations as transport equations.
\begin{thm}\label{M2}
Under the assumptions of Theorem \ref{maintheorem} and let $V$ be the region defined in \eqref{Vdf}. Moreover, we assume that the norms in $V$ satisfy the following bounds:
\begin{equation}
\II_{(0)}\les 1,\qquad\TT\les 1,\qquad \mw\leq\bo,\qquad \mg\leq\bo.
\end{equation}
Then, we have
\begin{equation}
    \mw\les 1.
\end{equation}
\end{thm}
Theorem \ref{M2} is proved in Section \ref{secR}. The proof is done by applying the $|u|^p$--weighted estimates to the Bianchi equations in Proposition \ref{bianchiequations}. As in Theorem \ref{M1}, the lower order $L^2(S_{u,\ub})$--norms of Weyl components are deduced by considering the Bianchi equations as transport equations. Similar arguments are also used in \cite{An,AnLuk,Shen24}.
\begin{thm}\label{M3}
Under the assumptions of Theorem \ref{maintheorem} and let $V$ be the region defined in \eqref{Vdf}. Moreover, we assume that the norms in $V$ satisfy the following bounds:
\begin{equation}
    \II_{(0)}\les 1,\qquad\TT\les 1,\qquad \mw\les 1,\qquad \mg\leq\bo.
\end{equation}
Then, we have
\begin{equation}
    \mg\les 1.
\end{equation}
\end{thm}
Theorem \ref{M3} is proved in Section \ref{secO}. The proof is done by first integrating the transport equations for the renormalized quantities in Propositions \ref{nulles} and \ref{nullrenor} along the incoming and outgoing null cones. Then, applying the $2D$--elliptic equations for the Ricci coefficients, we obtain the top-order $L^2$--flux estimates for the Ricci coefficients. The lower order $L^2(S_{u,\ub})$--norms of Weyl components are deduced by integrating the null structure equations along the incoming and outgoing null cones.
\subsection{Proof of the main theorem}\label{proofmain}
We now use Theorems \ref{M1}--\ref{M3} to prove Theorem \ref{maintheorem}.
\begin{df}\label{bootstrap}
For any $u_\infty\leq u_*\leq -\frac{a}{4}$, let $\aleph(u_*)$ be the set of spacetimes $V(u_*,1)$ associated with a double null foliation $(u,\ub)$ in which we have the following bounds:
\begin{align}
    \TT\leq\bo,\qquad\quad\mw\leq\bo,\qquad\quad\mg\leq\bo.\label{B2}
\end{align}
\end{df}
\begin{df}\label{defboot}
We denote by $\mathcal{U}$ the set of values $u_*$ such that $\aleph(u_*)\ne\emptyset$.
\end{df}
Combining the initial assumption, \eqref{bootsize} and the results in Section \ref{secsignatures}, we obtain
\begin{equation*}
    \II_{(0)}\leq 1.
\end{equation*}
Combining with the local existence theorem, we deduce that \eqref{B2} holds if $u_*$ is sufficiently close to $u_\infty$. So, we have $\mathcal{U}\ne\emptyset$.\\ \\
Define $u_*$ as the supremum of the set $\mathcal{U}$. We assume by contradiction that $u_*<-\frac{a}{4}$. In particular, we may assume $u_*\in\mathcal{U}$. We consider the region $V(u_*,1)$. Applying Theorem \ref{M1}, we obtain
\begin{equation*}
    \TT\les1.
\end{equation*}
Then, we apply Theorem \ref{M2} to obtain
\begin{equation*}
    \mw\les1.
\end{equation*}
Next, we apply Theorem \ref{M3} to infer
\begin{align*}
    \mg\les 1.
\end{align*}
Applying the characteristic local existence result,\footnote{See Theorem 1 in \cite{luk}.} we can extend $V(u_*,1)$ to $V(u_*+\nu,1)$ for $\nu>0$ sufficiently small. We denote $\widetilde{\TT}$, $\widetilde{\mw}$ and $\widetilde{\mg}$ the norms in the extended region $V(u_*+\nu,1)$. Then, we have the following estimates:
\begin{equation*}
    \widetilde{\TT}\les1,\qquad\quad\widetilde{\mw}\les 1,\qquad\quad\widetilde{\mg}\les 1,
\end{equation*}
as a consequence of continuity. Thus, we deduce that $V(u_*+\nu,1)$ satisfies all the properties of Definition \ref{bootstrap}, and so $\aleph(u_*+\nu)\ne\emptyset$, which is a contradiction. Thus, we have $u_*=-\frac{a}{4}$. Moreover, we have
\begin{equation*}
    \TT\les 1,\qquad\quad \mw\les 1,\qquad\quad\mg\les 1\qquad\mbox{ in }\; V_*.
\end{equation*}
Combining with the definitions of the norms in Section \ref{secnorms}, this concludes the proof of Theorem \ref{maintheorem}.
\section{Bootstrap assumptions and first consequences}\label{secboot}
In the rest of the paper, we always make the following bootstrap assumptions:
\begin{align}
    \TT\leq\bo,\quad\qquad\mw\leq\bo,\qquad\quad \mg\leq\bo.\label{B1}
\end{align}
\subsection{Schematic notation \texorpdfstring{$\Gag$}{} and \texorpdfstring{$\Gab$}{}}
We introduce the following schematic notations.
\begin{df}\label{gammag}
We divide the quantities into two parts:
\begin{align*}
\Gag&:=\left\{\trch,\,\frac{|u|}{a}\trcht,\,\frac{a}{|u|}\trchbt,\,\eta,\,\etab,\,\ze,\,\om,\,\omd,\,\omb,\,\ombd\right\},\\
\Gab&:=\left\{\hch,\,\frac{a}{|u|}\hchb,\,\bF,\,\psi,\,\Psi_4,\,\frac{a}{|u|}\Psi_3,\,\Asl\right\}\bigcup b\left\{\rhoF,\,\siF,\,\bbF,\,\Psisl,\,\Psit_3,\,\Ub\right\}.
\end{align*}
We also denote:
\begin{align*}
    \Gag^{(1)}&:=(a^\frac{1}{2}\nabs)^{\leq 1}\Gag,\\
    \Gab^{(1)}&:=(a^\frac{1}{2}\nabs)^{\leq 1}\Gab\bigcup\left\{\a,\nabs_4\bF\right\}\bigcup\af\left\{\b,\,\rho,\,\si,\,\bb,
    \,\Kc,\,\mu,\,\frac{|u|}{a}\muc,\,\mub,\,\vkp,\,\vkpb\right\}.
\end{align*}
Moreover, we define for $i\geq 1$
\begin{align*}
    \Gag^{(i+1)}:=(\af\nabs)^{\leq 1}\Gag^{(i)},\qquad\qquad \Gab^{(i+1)}:=(\af\nabs)^{\leq 1}\Gab^{(i)}.
\end{align*}
\end{df}
The following consequences of \eqref{B1} will be used frequently throughout this paper.
\begin{lem}\label{decayGagGab}
    Under the bootstrap assumption \eqref{B1}, we have
    \begin{align}
        \|\Gag^{(1)}\|_{L^\infty_{sc}(S_{u,\ub})}&\les\bo,\qquad \qquad\, \|\Gag^{(3)}\|_{L^2_{sc}(S_{u,\ub})}\les\bo,\label{Gagapriori}\\
        \|\Gab^{(1)}\|_{L^\infty_{sc}(S_{u,\ub})}&\les\bo\af,\qquad \quad \|\Gab^{(3)}\|_{L^2_{sc}(S_{u,\ub})}\les\bo\af.\label{Gabapriori}
    \end{align}
    Moreover, for all $\Up\in\Gag\cup\Gab$, we can write
    \begin{align*}
        \Up=\Up_1+\Up_2,
    \end{align*}
    with $\Up_1$ and $\Up_2$ satisfy:
    \begin{align*}
        \left\|\Up_1^{(4)}\right\|_{L^2_{sc}(\cuv)}\les\bo\af,\qquad\quad \left\|\Up_2^{(4)}\right\|_{L^2_{sc}(\ucuv)}\les\bo\af.
    \end{align*}
\end{lem}
\begin{proof}
    It follows directly from \eqref{B1} and the definitions of the norms in Section \ref{secnorms}.
\end{proof}
\begin{rk}
    Lemma \ref{decayGagGab} implies that the $L^2_{sc}(S_{u,\ub})$--norms of the elements in $\Gag$ behave $\afd$ better than the elements in $\Gab$. However, they have the same estimate for the $L^2$--flux at the top-order.
\end{rk}
\begin{rk}\label{GagGabrk}
In the sequel, we choose the following conventions:
\begin{itemize}
    \item For a quantity $h$ satisfying the same or even better decay and regularity as $\Ga_{i}$, for $i=g,b$, we write
    \begin{equation*}
        h\in\Ga_i,\qquad i=g,b.
    \end{equation*}
    \item For a sum of schematic notations, we ignore the terms which have same or even better decay and regularity. For example, we write\footnote{These schematic identities follow directly from \eqref{B1} and Proposition \ref{Holder}.}
    \begin{equation*}
        \Gag+\Gab=\Gab,\qquad\qquad (\Gab\c\Gab)^{(i)}=\Gab\c\Gab^{(i)}.
    \end{equation*}
    \item For a quantity $h$ satisfying the same or even better decay and regularity than $X^{(q)}$, we write
    \begin{equation*}
        h=X^{(q)},
    \end{equation*}
    where $X^{(q)}$ is defined in \eqref{dfXq}.
\end{itemize}
\end{rk}
\begin{prop}\label{computationGagGab}
    We have the following properties:
    \begin{align*}
        \Gab\c\Gab&=\frac{\af\bo}{|u|}\Gab,\qquad\qquad \trchb\,\Gab=\frac{|u|}{a}\Gab,\qquad\qquad \Om\Gab=\Gab,\\
        \ef\Gab\c\Gab&=b^{-1}\Gab,\qquad\qquad\quad\;\;\;\nabs\Gab=\afd\Gab^{(1)}.
    \end{align*}
\end{prop}
\begin{proof}
    We have from Proposition \ref{Holder}, \eqref{dfsc} and \eqref{B1} that for $i\leq 4$
    \begin{align*}
        \|(\Gab\c\Gab)^{(i)}\|_{L^2_{sc}(S_{u,\ub})}\les \frac{1}{|u|}\|\Gab\|_{L^\infty_{sc}(S_{u,\ub})}\|\Gab^{(i)}\|_{L^2_{sc}(S_{u,\ub})}\les \frac{\af\bo}{|u|}\|\Gab^{(i)}\|_{L^2_{sc}(S_{u,\ub})}.
    \end{align*}
    We also have for $i\leq 4$
    \begin{align*}
        \|(\trchb\Gab)^{(i)}\|_{L^2_{sc}(S_{u,\ub})}&\les\|(\dkb\trchb)^{(i-1)}\Gab\|_{L^2_{sc}(S_{u,\ub})}+\|\trchb\Gab^{(i)}\|_{L^2_{sc}(S_{u,\ub})}\\
        &\les\frac{1}{\af|u|}\|\trchbc^{(i)}\|_{L^2_{sc}(S_{u,\ub})}\|\Gab\|_{L^\infty_{sc}(S_{u,\ub})}+\frac{1}{|u|}\|\trchb\|_{L^\infty_{sc}(S_{u,\ub})}\|\Gab^{(i)}\|_{L^2_{sc}(S_{u,\ub})}\\
        &\les\|\Gag^{(i)}\|_{L^2_{sc}(S_{u,\ub})}+\frac{1}{|u|}\frac{|u|^3}{a}|\trchb|_{L^\infty(S_{u,\ub})}\|\Gab^{(i)}\|_{L^2_{sc}(S_{u,\ub})}\\
        &\les\frac{|u|}{a}\|\Gab^{(i)}\|_{L^2_{sc}(S_{u,\ub})}.
    \end{align*}
    Similarly, we have for $i\leq 4$
    \begin{align*}
        \|(\Om\Gab)^{(i)}\|_{L^2_{sc}(S_{u,\ub})}&\les\|(\dkb\Om)^{(i-1)}\Gab\|_{L^2_{sc}(S_{u,\ub})}+\|\Om\Gab^{(i)}\|_{L^2_{sc}(S_{u,\ub})}\\
        &\les\frac{1}{|u|}\|(\log\Om)^{(i)}\|_{L^2_{sc}(S_{u,\ub})}\|\Gab\|_{L^2_{sc}(S_{u,\ub})}+\frac{1}{|u|}\|\Om\|_{L^\infty_{sc}(S_{u,\ub})}\|\Gab^{(i)}\|_{L^2_{sc}(S_{u,\ub})}\\
        &\les\frac{\af\bo}{|u|}\|\Gag^{(i)}\|_{L^2_{sc}(S_{u,\ub})}+\|\Gab^{(i)}\|_{L^2_{sc}(S_{u,\ub})}\\
        &\les\|\Gab^{(i)}\|_{L^2_{sc}(S_{u,\ub})}.
    \end{align*}
    Next, we have from $s_2(\ef)=0.5$ and \eqref{dfsc} for $i\leq 4$
    \begin{align*}
        \|\ef(\Gab\c\Gab)^{(i)}\|_{L^2_{sc}(S_{u,\ub})}&\les\frac{1}{|u|}\|\ef\|_{L^2_{sc}(S_{u,\ub})}\|\Gab\|_{L^\infty_{sc}(S_{u,\ub})}\|\Gab^{(i)}\|_{L^2_{sc}(S_{u,\ub})}\\
        &\les\frac{1}{|u|^2}\frac{|u|^2}{\af}\ef\af\bo\|\Gab^{(i)}\|_{L^2_{sc}(S_{u,\ub})}\\
        &\les b^{-1}\|\Gab^{(i)}\|_{L^2_{sc}(S_{u,\ub})}.
    \end{align*}
    Finally, we have from \eqref{dfsc}
    \begin{align*}
        \nabs\Gab=\afd\dkb\Gab=\afd\Gab^{(1)}.
    \end{align*}
    Combining with Remark \ref{GagGabrk}, this concludes the proof of Proposition \ref{computationGagGab}.
\end{proof}
Remark \ref{GagGabrk} and Proposition \ref{computationGagGab} will be used frequently in the rest of the paper without explicit mention them.
\subsection{First consequences of bootstrap assumptions}\label{secfirstconsequence}
\begin{lem}\label{evolution}
For any $S$--tangent tensor field $\phi$, we have
\begin{align*}
    \|\phi\|_{L^2_{sc}(S_{u,\ub})}\les\|\phi\|_{L^2_{sc}(S_{u,0})}+\int_0^\ub \|\nabs_4\phi\|_{L^2_{sc}(S_{u,\ub'})}d\ub'.
\end{align*}
\end{lem}
\begin{proof}
See Proposition 3.6 in \cite{An}.
\end{proof}
\begin{lem}\label{3evolution}
    Let $\phi$ and $F$ be $S$--tangent tensor fields satisfying the following transport equation:
    \begin{align*}
        \nabs_3\phi+\la_0\trchb\phi=F.
    \end{align*}
    Denoting $\la_1=2\la_0-1$, we have
    \begin{align*}
        |u|^{\la_1-2s_2(\phi)}\|\phi\|_{L^2_{sc}(S_{u,\ub})}\les |u_\infty|^{\la_1-2s_2(\phi)}\|\phi\|_{L^2_{sc}(S_{u_\infty,\ub})}+a\int_{u_\infty}^u|u'|^{\la_1-2s_2(F)}\|F\|_{L^2_{sc}(S_{u',\ub})}du'.
    \end{align*}
\end{lem}
\begin{proof}
    See Proposition 3.7 in \cite{An}.
\end{proof}
\begin{prop}\label{sobolev}
    We have the following Sobolev inequality:
    \begin{align}
        \|\phi\|_{L^\infty_{sc}(S_{u,\ub})}\les\|\phi^{(2)}\|_{L^2_{sc}(S_{u,\ub})}.
    \end{align}
\end{prop}
\begin{proof}
    See Proposition 3.11 in \cite{An}.
\end{proof}
\begin{prop}\label{commutation}
We have the following schematic commutator formulae:
\begin{align*}
    [\Om\nabs_4,\dkb]&=\Gab\c\dkb+\afd\Gab^{(1)},\\
    [\Om\nabs_3,\dkb]&=\frac{|u|}{a}\Gab\c\dkb+\afd\Gab^{(1)}.
\end{align*}
\end{prop}
\begin{proof}
    It follows directly from Lemma \ref{comm} and Definition \ref{gammag}.
\end{proof}
\subsection{Schematic Maxwell equations and wave equations}
\begin{prop}\label{Maxwell}
We have the following schematic Maxwell equations:
\begin{align*}
    \nabs_3\bF+\frac{1}{2}\trchb\bF&=-\sld_1^*(\rhoF,\siF)+\Gab\c\Gab+\ef\Gab\c\Gab,\\
    \nabs_4(\rhoF,\siF)&=\sld_1\bF+\Gab\c\Gab+\ef\Gab\c\Gab,\\
    \nabs_3(\rhoF,-\siF)+\trchb (\rhoF,-\siF)&=-\sld_1\bbF+\Gab\c\Gab+\ef\Gab\c\Gab,\\
    \nabs_4\bbF&=\sld_1^*(\rhoF,-\siF)+\frac{|u|}{a}\Gab\c\Gab+\ef\Gab\c\Gab.
\end{align*}
We also have the following schematic Teukolsky equation for $\bF$:
\begin{align*}
    \nabs_3(\nabs_4\bF)+\frac{1}{2}\trchb(\nabs_4\bF)&=-\sld_1^*(\sld_1\bF)+(\Gab\c\Gab)^{(1)}+\ef(\Gab\c\Gab)^{(1)},\\
    \nabs_4(\sld_1\bF)&=\sld_1(\nabs_4\bF)+(\Gab\c\Gab)^{(1)}.
\end{align*}
\end{prop}
\begin{proof}
    It follows directly from Proposition \ref{Maxwellequations}, Corollary \ref{corteukolsky} and Definition \ref{gammag}.
\end{proof}
\begin{prop}\label{wave}
We have the following schematic wave equations:
\begin{align*}
\nabs_3(\Psi_4,0)+\frac{1}{2}\trchb(\Psi_4,0)&=\sld_1\Psisl+\Gab\c\Gab+\ef\Gab\c\Gab,\\
\nabs_4\Psisl&=-\sld_1^*(\Psi_4,0)+\Gab\c\Gab+\ef\Gab\c\Gab,\\
\nabs_3\Psisl+\frac{1}{2}\trchb\Psisl&=-\sld_1^*(\Psit_3,0)+\frac{|u|}{a}\Gab\c\Gab+\ef\Gab\c\Gab,\\
\nabs_4(\Psit_3,0)&=\sld_1\Psisl+\frac{|u|}{a}\Gab\c\Gab+\ef\Gab\c\Gab.
\end{align*}
\end{prop}
\begin{proof}
It follows directly from Proposition \ref{waveequation} and Definition \ref{gammag}.
\end{proof}
\subsection{Ricci curvature and Schouten tensor}
\begin{lem}\label{Ricmunu}
The null components of Ricci curvature take the following form:
    \begin{align}\label{EMCSFRic}
    \Ric_{\mu\nu}=2\Re\left(D_\mu\psi (D_\nu\psi)^\dag\right)+2\left(\F_{\mu\a}{\F_{\nu}}^\a-\frac{1}{4}\g_{\mu\nu}\F_{\a\b}\F^{\a\b}\right).
\end{align}
\end{lem}
\begin{proof}
    Recall from \eqref{EMCSF} that the Einstein--Maxwell-complex scalar field (EMCSF) system is given by
\begin{align}
\begin{split}\label{EMCSFeq}
    \Ric_{\mu\nu}-\frac{1}{2}\R\g_{\mu\nu}&=\T^{EM}_{\mu\nu}+\T^{CSF}_{\mu\nu},
\end{split}
\end{align}
where
\begin{align*}
    \T^{EM}_{\mu\nu}&=2\left(\F_{\mu\a}{\F_{\nu}}^\a-\frac{1}{4}\g_{\mu\nu}\F_{\a\b}\F^{\a\b}\right),\\
    \T^{CSF}_{\mu\nu}&= 2 \left(\Re\left(D_\mu\psi (D_\nu\psi)^\dag\right)-\frac{1}{2}\g_{\mu\nu} D_\a\psi (D^\a\psi)^\dag\right).
\end{align*}
Contracting \eqref{EMCSFeq} with $\g^{\mu\nu}$, we obtain
\begin{align*}
    \R=-\g^{\mu\nu}T_{\mu\nu}^{CSF}=2\Re\left(D_\a\psi(D^\a\psi)^\dag\right).
\end{align*}
Injecting it into \eqref{EMCSFeq}, we obtain \eqref{EMCSFRic}. This concludes the proof of Lemma \ref{Ricmunu}.
\end{proof}
\begin{lem}\label{Ricciexpression}
We have the following expressions for the components of Ricci curvature:
\begin{align*}
\Ric_{AB}&=2\Re(\Psisl_A\Psisl_B^\dag)-2\left(\bF\hot\bbF\right)_{AB}+\slg_{AB}\left(\rhoF^2+\siF^2\right),\\
\Ric_{A3}&=2\Re(\Psisl_A\Psi_3^\dag)-2\rhoF\bbF_A+2\siF{}^*\bbF_A,\\
\Ric_{A4}&=2\Re(\Psisl_A\Psi_4^\dag)+2\rhoF\bF_A+2\siF{}^*\bF_A,\\
\Ric_{33}&=2|\Psi_3|^2+2|\bbF|^2,\\
\Ric_{34}&=2\Re(\Psi_3\Psi_4^\dag)+2\rhoF^2+2\siF^2,\\
\Ric_{44}&=2|\Psi_4|^2+2|\bF|^2.
\end{align*}
We also have the following expression for scalar curvature:
\begin{align*}
    \R=-2\Re(\Psi_3\Psi_4^\dag)+2|\Psisl|^2.
\end{align*}
\end{lem}
\begin{proof}
We recall from \eqref{EMCSFRic}
\begin{align*}
\Ric_{\mu\nu}=2\Re\left(D_\mu\psi(D_\nu\psi)^\dag\right)+2\left(\F_{\mu\a}{\F_{\nu}}^\a-\frac{1}{4}\g_{\mu\nu}\F_{\a\b}\F^{\a\b}\right).
\end{align*}
Combining with \eqref{defF}, this concludes the proof of Lemma \ref{Ricciexpression}.
\end{proof}
\begin{cor}\label{RicGagGab}
We have the following schematic expressions:
\begin{align}
\begin{split}\label{SGagGab}
\S_{44}&=\Gab\c\Gab,\qquad\quad\quad\;\S_{34}=\frac{|u|}{a}\Gab\c\Gab,\\
\Ssl,\Ssl_{4}&=b^{-1}\Gab\c\Gab,\qquad\quad\,\Ssl_{3}=\frac{|u|}{ab}\Gab\c\Gab.
\end{split}
\end{align}
We also have
\begin{align}\label{J4GagGab}
    \Jh_4=(\Gab\c\Gab)^{(1)},\qquad\quad J_{3A4},J_{4A3}=\frac{|u|}{ab}(\Gab\c\Gab)^{(1)}.
\end{align}
\end{cor}
\begin{proof}
Notice that \eqref{SGagGab} follows directly from \eqref{dfSchouten}, Lemma \ref{Ricciexpression} and Definition \ref{gammag}. Next, we have from Lemma \ref{Jformula}
    \begin{align}
    \begin{split}\label{JAB4computation}
        J_{AB4}&=\frac{1}{2}\nabs_B\S_{A4} +\frac{1}{2}(\ze_B+\etab_B)\S_{A4}+\frac{1}{2}\etab_A\S_{B4}-\frac{1}{2}\nabs_4\S_{AB}\\
        &-\frac{1}{2}\chi_{BC}\S_{AC}-\frac{1}{4}\chi_{AB}\S_{34}-\frac{1}{4}\chib_{AB}\S_{44}.
    \end{split}
    \end{align}
    Notice that we have from Lemma \ref{Ricciexpression} and \eqref{dfSchouten}
    \begin{align*}
\nabs_4\S_{AB}&=\nabs_4\left(2\Re(\Psisl_A\Psisl_B^\dag)-2(\bF\hot\bbF)_{AB}+\de_{AB}(\rhoF^2+\siF^2)\right)\\
&=2\nabs_4\left(\Re(\Psisl_A\Psisl_B^\dag)-(\bF\hot\bbF)_{AB}\right)+\de_{AB}\nabs_4\left(\rhoF^2+\siF^2\right).
    \end{align*}
    Applying Propositions \ref{Maxwell} and \ref{wave}, we have from Remark \ref{GagGabrk}
    \begin{align*}
        \nabs_4\Sslh=\Gab\c\Gab^{(1)}+\ef\Gab\c\Gab\c\Gab=\Gab\c\Gab^{(1)}.
    \end{align*}
    Injecting it into \eqref{JAB4computation} and combining with Definition \ref{gammag} and \eqref{RicGagGab}, we obtain
    \begin{align*}
        \Jh_4=\Gab\c\Gab^{(1)}-\frac{1}{4}\S_{44}\,\hchb=\Gab\c\Gab+\frac{|u|}{a}\Gab\c\Gab\c\Gab=\Gab\c\Gab^{(1)}.
    \end{align*}
    Finally, we have from Lemma \ref{Ricciexpression} and Propositions \ref{Maxwell} and \ref{wave}
    \begin{align*}
        \nabs_4\S_{3A}&=2\nabs_4\left(\Re(\Psisl_A\Psi_3^\dag-\rhoF\bbF+\siF\bbF)\right)\\
        &=\frac{|u|}{a^\frac{3}{2}}\Gab\c\Gab^{(1)}+\frac{\ef|u|}{a}\Gab\c\Gab\c\Gab\\
        &=\frac{|u|}{ab}\Gab\c\Gab^{(1)}.
    \end{align*}
    Hence, we deduce from Lemma \ref{Jformula} and Proposition \ref{computationGagGab}
    \begin{align*}
        J_{3A4}&=\frac{1}{2}\nabs_A\S_{34}-\frac{1}{2}\etab_A\S_{34}-\frac{1}{2}\nabs_4\S_{3A}-\frac{1}{2}\chi_{AB}\S_{3B}+\om\S_{3A}-\frac{1}{2}\chib_{AB}\S_{4B}+\etab_B\S_{AB}\\
        &=\frac{|u|}{a^\frac{3}{2}}\Gab\c\Gab^{(1)}+\frac{|u|}{ab}\Gab\c\Gab^{(1)}+b^{-1}\trchb\,\Gab\c\Gab+\frac{|u|}{a}\Gab\c\Gab\c\Gab\\
        &=\frac{|u|}{ab}(\Gab\c\Gab)^{(1)}.
    \end{align*}
    Similarly, we also have
    \begin{align*}
        J_{4A3}=\frac{|u|}{ab}(\Gab\c\Gab)^{(1)}.
    \end{align*}
    This concludes the proof of Corollary \ref{RicGagGab}.
\end{proof}
\subsection{Schematic Bianchi equations}
We now deduce the following schematic Bianchi equations.
\begin{prop}\label{bianchi}
We have the following schematic Bianchi equations:
\begin{align*}
\nabs_3\a+\frac{1}{2}\trchb\,\a&=\nabs\hot\bt+(\Gab\c\Gab)^{(1)},\\
\nabs_4\bt&=\sdivs\a+(\Gab\c\Gab)^{(1)},\\
\nabs_3\b+\trchb\,\b&=\sld_1^*(\K,\sic)+\frac{|u|}{ab}(\Gab\c\Gab)^{(1)},\\
\nabs_4(\K,\sic)&=-\sld_1\b-\frac{1}{2}\trchb|\hch|^2+\bo\afd(\Gab\c\Gab)^{(1)},\\
\nabs_3(\Kc,-\sic)+\frac{3}{2}\trchb(\Kc,-\sic)&=\sld_1\bb+\frac{1}{2}\trchb\,\mu+\frac{|u|}{ab}(\Gab\c\Gab)^{(1)},\\
\nabs_4\bb&=-\sld_1^*(\Kc,-\sic)+\frac{|u|}{ab}(\Gab\c\Gab)^{(1)}.
\end{align*}
\end{prop}
\begin{proof}
The first two equations follow directly from Proposition \ref{bianchiequations}, Definition \ref{gammag} and Corollary \ref{RicGagGab}. Next, note that the Bianchi equations for $\b$, $(\Kc,\sic)$ and $\bb$ decouple with $\a$ and $\aa$. Hence, the worst quadratic terms among them are given by $\hchb\c \{\b,\rho,\si,\bb\}$, which behaves like $|u|a^{-\frac{3}{2}}\Gab\c\Gab^{(1)}$. Moreover, the worst cubic terms among them have the form $\hchb\c\Gab\c\Gab$, which behaves like $|u|a^{-1}\Gab\c\Gab\c\Gab=\afd\bo\Gab\c\Gab$. Combining with the schematic form of $\S$ and $J$ in Corollary \ref{RicGagGab}, this concludes the schematic equations in Proposition \ref{bianchi}.
\end{proof}
\subsection{Schematic null structure equations}
\begin{prop}\label{null}
We have the following schematic null structure equations:
\begin{align*}
\nabs_4\eta&=-\b+\Gab\c\Gab,\\
\nabs_3\etab+\frac{1}{2}\trchb\,\etab&=\bb+\frac{1}{2}\trchb\,\eta+\frac{|u|}{a}\Gab\c\Gab,\\
\nabs_4\hch&=-\a+\Gab\c\Gab,\\
\nabs_4\trcht&=\Gab\c\Gab,\\
\nabs_3\hchb+\trchb\,\hchb&=-\aa+\frac{|u|}{a}\Gab\c\Gab,\\
\nabs_3\trchbt+\trchb\,\trchbt&=\frac{|u|^2}{a^2}\Gab\c\Gab,\\
\nabs_4\hchb&=\nabs\hot\etab-\frac{1}{2}\trchb\,\hch+\frac{|u|}{a}\Gab\c\Gab,\\
\nabs_3\hch+\frac{1}{2}\trchb\,\hch&=\nabs\hot\eta+\frac{|u|}{a}\Gab\c\Gab,\\
\nabs_3\om&=-\frac{1}{2}\K-\frac{1}{8}\trch\trchb+\frac{|u|}{a}\Gab\c\Gab,\\
\nabs_4\omb&=-\frac{1}{2}\K-\frac{1}{8}\trch\trchb+\frac{|u|}{a}\Gab\c\Gab,\\
\sdivs\hch&=\frac{1}{2}\nabs\trch-\b+\Gab\c\Gab,\\ 
\sdivs\hchb&=\frac{1}{2}\nabs\trchb+\bb-\frac{1}{2}\trchb\,\ze+\frac{|u|}{a}\Gab\c\Gab.
\end{align*}
We also have the following schematic equations for the renormalized quantities:
\begin{align*}
  \nabs_4\vkpb&=\frac{|u|}{ab}(\Gab\c\Gab)^{(1)},\\
  \nabs_4\mu&=-\frac{1}{|u|^2}\trch+\afd(\Gab\c\Gab)^{(1)},\\
  \nabs_3\vkp+\frac{1}{2}\trchb\vkp&=\frac{1}{4}\trchb\,\b+\frac{|u|}{ab}(\Gab\c\Gab)^{(1)},\\
  \nabs_3\mub+\trchb\mub&=-\frac{1}{|u|^2}\trchbt-\trchb\,\sdivs\eta+\frac{|u|}{a^\frac{3}{2}}(\Gab\c\Gab)^{(1)}.
\end{align*}
\end{prop}
\begin{proof}
It follows directly from Proposition \ref{nullrenor}, Corollary \ref{RicGagGab} and Definition \ref{gammag}.
\end{proof}
\section{General Bianchi equations}\label{secBianchi}
\subsection{General Bianchi pairs}
The following lemma provides the general structure of Bianchi pairs.
\begin{lem}\label{keypoint}
Let $k=1,2$. Then, we have the following properties:
\begin{enumerate}
    \item If $\psi_{(1)},h_{(1)}\in\sk_k$ and $\psi_{(2)},h_{(2)}\in\sk_{k-1}$ satisfying
    \begin{align}
    \begin{split}\label{bianchi1}
    \nabs_3(\psi_{(1)})+\left(\frac{1}{2}+s_2(\psi_{(1)})\right)\trchb\,\psi_{(1)}&=-k\sld_k^*(\psi_{(2)})+h_{(1)},\\
    \nabs_4(\psi_{(2)})&=\sld_k(\psi_{(1)})+h_{(2)}.
    \end{split}
    \end{align}
Then, the pair $(\psi_{(1)},\psi_{(2)})$ satisfies for any real number $p$
\begin{align}
\begin{split}\label{div}
       &\bdiv(|u|^p|\psi_{(1)}|^2e_3)+k\bdiv(|u|^p|\psi_{(2)}|^2e_4)+\left(p-4s_2(\psi_{(1)})\right)|u|^{p-1}|\psi_{(1)}|^2\\
       =&2k|u|^p\sdivs\left[\Re\left(\psi_{(1)}^\dag\c\psi_{(2)}\right)\right]
       +2|u|^p\Re\left(\psi_{(1)}^\dag\cdot h_{(1)}\right)+2k|u|^p\Re\left(\psi_{(2)}^\dag\c h_{(2)}\right)\\
       +&a^{-1}|u|^{p+1}\Gag|\psi_{(1)}|^2+|u|^p\Gag|\psi_{(2)}|^2.
\end{split}
\end{align}
    \item If $\psi_{(1)},h_{(1)}\in\sk_{k-1}$ and $\psi_{(2)},h_{(2)}\in\sk_k$ satisfying
\begin{align}
\begin{split}\label{bianchi2}
\nabs_3(\psi_{(1)})+\left(\frac{1}{2}+s_2(\psi_1)\right)\trchb\,\psi_{(1)}&=\sld_k(\psi_{(2)})+h_{(1)},\\
\nabs_4(\psi_{(2)})&=-k\sld_k^*(\psi_{(1)})+h_{(2)}.
\end{split}
\end{align}
Then, the pair $(\psi_{(1)},\psi_{(2)})$ satisfies for any real number $p$
\begin{align}
\begin{split}\label{div2}
    &k\bdiv(|u|^p|\psi_{(1)}|^2e_3)+\bdiv(|u|^p|\psi_{(2)}|^2e_4)+k(p-4s_2(\psi_{(1)}))|u|^{p-1}|\psi_{(1)}|^2\\
    =&2|u|^p\sdivs\left[\Re\left(\psi_{(1)}^\dag\c\psi_{(2)}\right)\right]
    +2k|u|^p\Re\left(\psi_{(1)}^\dag\c h_{(1)}\right)+2|u|^p\Re\left(\psi_{(2)}^\dag\c h_{(2)}\right)\\
    +&a^{-1}|u|^{p+1}\Gag|\psi_{(1)}|^2+|u|^p\Gag|\psi_{(2)}|^2.
\end{split}
\end{align}
\end{enumerate}
\end{lem}
\begin{rk}\label{Bianchitype}
    Note that the Bianchi equations can be written as systems of equations of type \eqref{bianchi1} and \eqref{bianchi2}. In particular
    \begin{itemize}
        \item the Bianchi pair $(\a,\bt)$ satisfies \eqref{bianchi1} with $k=2$,
        \item the Bianchi pair $(\b,(\K,-\sic))$ satisfies \eqref{bianchi1} with $k=1$,
        \item the Bianchi pair $((\Kc,\sic),\bb)$ satisfies \eqref{bianchi2} with $k=1$,
        \item the Bianchi pair $(\bb,\aa)$ satisfies \eqref{bianchi2} with $k=2$,
        \item the Bianchi pair $(\bF,(\rhoF,\siF))$ satisfies \eqref{bianchi1} with $k=1$,
        \item the Bianchi pair $((\rhoF,-\siF),\bbF))$ satisfies \eqref{bianchi2} with $k=1$,
        \item the Bianchi pair $((\Psi_4,0),\Psisl)$ satisfies \eqref{bianchi1} with $k=1$,
        \item the Bianchi pair $(\Psisl,(\Psit_3,0))$ satisfies \eqref{bianchi2} with $k=1$.
    \end{itemize}
\end{rk}
\begin{proof}[Proof of Lemma \ref{keypoint}]
By a direct computation, we have
\begin{align*}
    \bdiv e_4&=-2\om+\trch\in\Gag,\qquad\qquad\bdiv e_3=-2\omb+\trchb=\trchb+\Gag.
\end{align*}
For $\psi_{(1)}$ and $\psi_{(2)}$ satisfying \eqref{bianchi1}, we compute
\begin{align*}
&\bdiv(|u|^p|\psi_{(1)}|^2 e_3)+k\bdiv(|u|^p |\psi_{(2)}|^2 e_4)\\
=&\nabs_3(|u|^p|\psi_{(1)}|^2)+ \bdiv(e_3)|u|^p|\psi_{(1)}|^2+k\nabs_4(|u|^p |\psi_{(2)}|^2)+k\bdiv(e_4)|u|^p|\psi_{(2)}|^2\\
=&2|u|^p\Re\left(\psi_{(1)}^\dag\c\nabs_3\psi_{(1)}\right)+p|u|^{p-1}e_3(|u|)|\psi_{(1)}|^2+\trchb|u|^p|\psi_{(1)}|^2\\
&+2k|u|^p\Re\left(\psi_{(2)}^\dag\c\nabs_4\psi_{(2)}\right)+|u|^p\Gag(|\psi_{(1)}|^2+|\psi_{(2)}|^2)\\
=&2|u|^p\Re\left[\psi_{(1)}^\dag\c\left(-\left(\frac{1}{2}+s_2(\psi_{(1)})\right)\trchb\psi_{(1)}-k\sld_k^*(\psi_{(2)})+h_{(1)}\right)\right]-(p+2)|u|^{p-1}|\psi_{(1)}|^2\\
&+2k|u|^p\Re\left[\psi_{(2)}^\dag\c\left(\sld_k(\psi_{(1)})+h_{(2)}\right)\right]+a^{-1}|u|^{p+1}\Gag|\psi_{(1)}|^2+|u|^p\Gag|\psi_{(2)}|^2\\
=&(4s_2(\psi_{(1)})-p)|u|^{p-1}|\psi_{(1)}|^2+2k|u|^p\sdivs\left[\Re\left(\psi_{(1)}^\dag\c\psi_{(2)}\right)\right]\\
&+2|u|^p\Re\left(\psi_{(1)}^\dag\c h_{(1)}\right)+2k|u|^p\Re\left(\psi_{(2)}^\dag\c h_{(2)}\right)+a^{-1}|u|^{p+1}\Gag|\psi_{(1)}|^2+|u|^p\Gag|\psi_{(2)}|^2,
\end{align*}
which implies \eqref{div}. The proof of \eqref{div2} is similar and is left to the reader. This concludes the proof of Lemma \ref{keypoint}.
\end{proof}
\begin{prop}\label{keyintegral}
Let $(\psi_{(1)},\psi_{(2)})$ be a Bianchi pair that satisfies \eqref{bianchi1} or \eqref{bianchi2}. Then, we have
\begin{align*}
&\;\|\psi_{(1)}\|_{L^2_{sc}(\cuv)}^2+\|\psi_{(2)}\|^2_{L^2_{sc}(\ucuv)}\\
\les&\;\|\psi_{(1)}\|_{L^2_{sc}(H_{u_\infty}^{(0,\ub)})}^2+\|\psi_{(2)}\|^2_{L^2_{sc}(\Hb_{0}^{(u_\infty,u)})}\\
+&\int_0^\ub\int_{u_\infty}^u\frac{a}{|u'|}\left\|\psi_{(1)}\c\left(h_{(1)}+|u'|a^{-1}\Gag\c\psi_{(1)}\right)\right\|_{L^1_{sc}(S_{u',\ub'})}du'd\ub'\\
+&\int_0^\ub\int_{u_\infty}^u\frac{a}{|u'|}\left\|\psi_{(2)}\c (h_{(2)}+\Gag\c\psi_{(2)})\right\|_{L^1_{sc}(S_{u',\ub'})}du'd\ub'.
\end{align*}
\end{prop}
\begin{proof}
Integrating \eqref{div} or \eqref{div2} in $V$ and taking $p=4s_2(\psi_1)$, we obtain
\begin{align*}
&\int_\cuv |u|^{4s_2(\psi_1)}|\psi_1|^2+\int_\ucuv |u|^{4s_2(\psi_1)}|\psi_2|^2\\
\les&\int_{H_{u_\infty}^{(0,\ub)}}|u_\infty|^{4s_2(\psi_1)}|\psi_1|^2+\int_{\Hb_0^{(u_\infty,u)}}|u|^{4s_2(\psi_1)}|\psi_2|^2\\
+&\int_V |u|^{4s_2(\psi_1)}\left(\psi_{(1)}\c h_{(1)}+\psi_{(2)}\c h_{(2)}+|u|a^{-1}\Gag|\psi_{(1)}|^2+\Gag|\psi_{(2)}|^2\right).
\end{align*}
Notice that we have from \eqref{bianchi1} or \eqref{bianchi2}
\begin{align*}
    s_2(\psi_2)=s_2(\psi_1)+\frac{1}{2},\qquad s_2(h_{(1)})=s_2(\psi_1)+1,\qquad s_2(h_{(2)})=s_2(\psi_1)+\frac{1}{2}.
\end{align*}
Combining with \eqref{dfsc}, we infer from Proposition \ref{Holder}
\begin{align*}
&\int_0^\ub\|\psi_1\|_{L^2_{sc}(S_{u,\ub'})}^2d\ub'+\int_{u_\infty}^u\frac{a}{|u'|^2}\|\psi_2\|^2_{L^2_{sc}(S_{u',\ub})}du'\\
\les&\int_{H_{u_\infty}^{(0,\ub)}}\|\psi_1\|_{L^2_{sc}(S_{u_\infty,u'})}^2d\ub'+\int_{\Hb_0^{(u_\infty,u)}}\frac{a}{|u'|^2}\|\psi_2\|^2_{L^2_{sc}(S_{u',0})}du'\\
+&\int_0^\ub\int_{u_\infty}^u\frac{a}{|u'|}\left\|\psi_{(1)}\c\left(h_{(1)}+|u'|a^{-1}\Gag\c\psi_{(1)}\right)\right\|_{L^1_{sc}(S_{u',\ub'})}du'd\ub'\\
+&\int_0^\ub\int_{u_\infty}^u\frac{a}{|u'|}\left\|\psi_{(2)}\c (h_{(2)}+\Gag\c\psi_{(2)})\right\|_{L^1_{sc}(S_{u',\ub'})}du'd\ub'.
\end{align*}
This concludes the proof of Proposition \ref{keyintegral}.
\end{proof}
\subsection{Estimates for nonlinear terms}
The following theorem provides a unified treatment of all nonlinear error terms in the estimates for matter fields, curvature components, and Ricci coefficients.
\begin{thm}\label{Junkmanthm}
We have the following estimates:
\begin{align}
\int_0^\ub\int_{u_\infty}^u\left\|\Gab^{(4)}\c(\Gab\c\Gab)^{(4)}\right\|_{L_{sc}^1(S_{u',\ub'})}du'd\ub'&\les \af b^\frac{3}{4},\label{GabGabGab}\\
\int_0^\ub\int_{u_\infty}^u\left\|(\Gab\c\Gab)^{(4)}\right\|_{L_{sc}^2(S_{u',\ub'})}^2du'd\ub'&\les ab.\label{GabGab}
\end{align}
\end{thm}
\begin{proof}
We have from \eqref{B1} and Proposition \ref{Holder}
\begin{align*}
    \left\|(\Gab\c\Gab)^{(4)}\right\|_{L^2_{sc}(S_{u,\ub})}^2&\les\left\|\Gab\c\Gab^{(4)}\right\|_{L^2_{sc}(S_{u,\ub})}\left\|\Gab\c\Gab^{(4)}\right\|_{L^2_{sc}(S_{u,\ub})}\\
    &\les\frac{b^\frac{1}{2}a}{|u|^2}\left\|\Gab^{(4)}\right\|_{L^2_{sc}(S_{u,\ub})}\left\|\Gab^{(4)}\right\|_{L^2_{sc}(S_{u,\ub})},\\
    \left\|\Gab^{(4)}\c(\Gab\c\Gab)^{(4)}\right\|_{L_{sc}^1(S_{u,\ub})}&\les \frac{1}{|u|}\left\|\Gab^{(4)}\right\|_{L_{sc}^2(S_{u,\ub})}\left\|\Gab^{(4)}\c\Gab\right\|_{L_{sc}^2(S_{u,\ub})}\\
    &\les\frac{\bo\af}{|u|^2}\left\|\Gab^{(4)}\right\|_{L_{sc}^2(S_{u,\ub})}\left\|\Gab^{(4)}\right\|_{L_{sc}^2(S_{u,\ub})}.
\end{align*}
Hence, we obtain
\begin{align}
\begin{split}\label{bulkest}
\int_0^\ub\int_{u_\infty}^u\left\|\Gab^{(4)}\c(\Gab\c\Gab)^{(4)}\right\|_{L_{sc}^1(S_{u',\ub'})}du'd\ub'&\les\bo\int_0^\ub\int_{u_\infty}^u\frac{\af}{|u'|^2}\left\|\Gab^{(4)}\right\|_{L_{sc}^2(S_{u',\ub'})}^2du'd\ub',\\
\int_0^\ub\int_{u_\infty}^u\left\|(\Gab\c\Gab)^{(4)}\right\|_{L_{sc}^2(S_{u',\ub'})}^2du'd\ub'&\les b^\frac{1}{2}\int_0^\ub\int_{u_\infty}^u\frac{a}{|u'|^2}\left\|\Gab^{(4)}\right\|_{L_{sc}^2(S_{u',\ub'})}^2du'd\ub'.
\end{split}
\end{align}
We have from Lemma \ref{decayGagGab} that $\Gab=\Ga_1+\Ga_2$ with
\begin{align}
    \|\Ga_1^{(4)}\|_{L^2_{sc}(\cuv)}&\les\bo\af,\label{Gagleft}\\
    \|\Ga_2^{(4)}\|_{L^2_{sc}(\ucuv)}&\les\bo\af.\label{Gagright}
\end{align}
Notice that we have
\begin{align*}
    \int_0^\ub\int_{u_\infty}^u\frac{a}{|u'|^2}\left\|\Ga_1^{(4)}\right\|_{L_{sc}^2(S_{u',\ub'})}^2du'd\ub'\les \int_{u_\infty}^u\frac{a^2b^\frac{1}{2}}{|u'|^2}du'\les\frac{a^2b^\frac{1}{2}}{|u|}\les ab^\frac{1}{2}.
\end{align*}
We also have
\begin{align*}
    \int_0^\ub\int_{u_\infty}^u\frac{a}{|u'|^2}\left\|\Ga_2^{(4)}\right\|_{L_{sc}^2(S_{u',\ub'})}^2du'd\ub'\les \int_{0}^\ub ab^\frac{1}{2}d\ub'\les ab^\frac{1}{2}.
\end{align*}
Injecting them into \eqref{bulkest}, this concludes the proof of Theorem \ref{Junkmanthm}.
\end{proof}
\section{Estimates for matter fields}\label{secF}
In this section, we prove Theorem \ref{M1} using the $|u|^p$--weighted estimate method developed in Section \ref{secBianchi} and the evolution lemmas stated in Section \ref{secfirstconsequence}.
\subsection{Estimates for Maxwell field and scalar field}
\subsubsection{Estimates for the Bianchi pairs \texorpdfstring{$(\bF,(\rhoF,\siF))$}{} and \texorpdfstring{$((\Psi_4,0),\Psisl)$}{}}
\begin{prop}\label{estbrF}
Let
\begin{equation*}
(\psi_1,\psi_2)\in\left\{(\bF,(\rhoF,\siF)), ((\Psi_4,0),\Psisl)\right\}.
\end{equation*}
We have the following estimate:
\begin{align*}
    \left\|\psi_1^{(4)}\right\|_{L^2_{sc}(\cuv)}+\left\|\psi_2^{(4)}\right\|_{L^2_{sc}(\ucuv)}\les\af.
\end{align*}
\end{prop}
\begin{proof}
    We have from Propositions \ref{Maxwell} and \ref{wave}
    \begin{align*}
    \nabs_3\psi_1+\left(\frac{1}{2}+s_2(\psi_1)\right)\trchb\,\psi_1&=-\slD^*\psi_2+\Gab\c\Gab+\ef\Gab\c\Gab,\\
    \nabs_4\psi_2&=\slD\psi_1+\Gab\c\Gab+\ef\Gab\c\Gab,
    \end{align*}
    where $\slD$ denotes a Hodge operator which depends on the type of $\psi_1$ and $\psi_2$.\footnote{See Lemma \ref{keypoint} and Remark \ref{Bianchitype} for the explicit form of $\slD$ and $\slD^*$ for $(\psi_1,\psi_2)$.} Differentiating it by $\dkb^i$ and applying Proposition \ref{commutation}, we infer
    \begin{align*}
    \nabs_3(\dkb^i\psi_1)+\left(\frac{1+i}{2}+s_2(\psi_1)\right)\trchb(\dkb^i\psi_1)&=-\slD^*(\dkb^i\psi_2)+\frac{|u|}{a}(\Gab\c\Gab)^{(i)}+\ef(\Gab\c\Gab)^{(i)},\\
    \nabs_4(\dkb^i\psi_2)&=\slD(\dkb^i\psi_1)+(\Gab\c\Gab)^{(i)}+\ef(\Gab\c\Gab)^{(i)},
    \end{align*}
where we used the fact that $\psi_1,\psi_2\in\Gab$. Applying Proposition \ref{keyintegral}, we deduce for $i\leq 4$
\begin{align*}
\|\dkb^i\psi_1\|_{L^2_{sc}(\cuv)}^2+\|\dkb^i\psi_2\|^2_{L^2_{sc}(\ucuv)}&\les\|\dkb^i\psi_1\|_{L^2_{sc}(H_{u_\infty}^{(0,\ub)})}^2+\|\dkb^i\psi_2\|^2_{L^2_{sc}(\Hb_{0}^{(u_\infty,u)})}\\
&+\int_0^\ub\int_{u_\infty}^u\left\|\Gab^{(4)}\c(\Gab\c\Gab)^{(4)}\right\|_{L^1_{sc}(S_{u',\ub'})}du'd\ub'\\
&+\int_0^\ub\int_{u_\infty}^u\frac{a}{|u'|}\left\|\ef\Gab^{(4)}\c(\Gab\c\Gab)^{(4)}\right\|_{L^1_{sc}(S_{u',\ub'})}du'd\ub'.
\end{align*}
We have from the initial assumption that
\begin{align*}
     \|\psi_1^{(4)}\|_{L^2_{sc}(H_{u_\infty}^{(0,\ub)})}^2+\|\psi_2^{(4)}\|^2_{L^2_{sc}(\Hb_{0}^{(u_\infty,u)})}\les a.
\end{align*}
Applying \eqref{GabGabGab} in Theorem \ref{Junkmanthm}, we deduce from \eqref{bootsize}
\begin{align*}
    \int_0^\ub\int_{u_\infty}^u\left\|\Gab^{(4)}\c(\Gab\c\Gab)^{(4)}\right\|_{L^1_{sc}(S_{u',\ub'})}du'd\ub'\les b^\frac{3}{4}\af\les a.
\end{align*}
Recalling that $s_2(\ef)=0.5$, we have from Lemma \ref{Holder} and \eqref{dfsc}
\begin{align*}
&\int_0^\ub\int_{u_\infty}^u\frac{a}{|u'|}\left\|\ef\Gab^{(4)}\c(\Gab\c\Gab)^{(4)}\right\|_{L^1_{sc}(S_{u',\ub'})}du'd\ub'\\
\les\;&\int_0^\ub\int_{u_\infty}^u\frac{a}{|u'|^2}\|\ef\|_{L^\infty_{sc}}\left\|\Gab^{(4)}\c(\Gab\c\Gab)^{(4)}\right\|_{L^1_{sc}(S_{u',\ub'})}du'd\ub'\\
\les\;&\af\ef\int_0^\ub\int_{u_\infty}^u\left\|\Gab^{(4)}\c(\Gab\c\Gab)^{(4)}\right\|_{L^1_{sc}(S_{u',\ub'})}du'd\ub'\\
\les\;&b^\frac{3}{4}\ef a\les a,
\end{align*}
where we used \eqref{GabGabGab} in Theorem \ref{Junkmanthm} and \eqref{bootsize} at the last step. Combining the above estimates and applying Proposition \ref{ellipticLp}, this concludes the proof of Proposition \ref{estbrF}.
\end{proof}
\subsubsection{Estimates for the Bianchi pairs \texorpdfstring{$((\rhoF,-\siF),\bbF)$}{} and \texorpdfstring{$(\Psisl,(\Psit_3,0))$}{}}
\begin{prop}\label{estrbbF}
Let
\begin{equation*}
(\psi_1,\psi_2)\in\left\{(\rhoF,-\siF),\bbF), (\Psisl,(\Psit_3,0))\right\}.
\end{equation*}
We have the following estimate:
\begin{align*}
    \left\|\psi_1^{(4)}\right\|_{L^2_{sc}(\cuv)}+\left\|\psi_2^{(4)}\right\|_{L^2_{sc}(\ucuv)}\les b^{-1}\af.
\end{align*}
\end{prop}
\begin{proof}
    We have from Propositions \ref{Maxwell} and \ref{wave}
    \begin{align*}
    \nabs_3\psi_1+\left(\frac{1}{2}+s_2(\psi_1)\right)\trchb\,\psi_1&=-\slD^*\psi_2+\frac{|u|}{a}\Gab\c\Gab+\ef\Gab\c\Gab,\\
    \nabs_4\psi_2&=\slD\psi_1+\frac{|u|}{a}\Gab\c\Gab+\ef\Gab\c\Gab.
    \end{align*}
    Differentiating it by $\dkb^i$ and applying Proposition \ref{commutation}, we infer
    \begin{align*}
    \nabs_3(\dkb^i\psi_1)+\left(\frac{1+i}{2}+s_2(\psi_1)\right)\trchb(\dkb^i\psi_1)&=-\slD^*(\dkb^i\psi_2)+\frac{|u|}{a}(\Gab\c\Gab)^{(i)}+\ef(\Gab\c\Gab)^{(i)},\\
    \nabs_4(\dkb^i\psi_2)&=\slD(\dkb^i\psi_1)+\frac{|u|}{a}(\Gab\c\Gab)^{(i)}+\ef(\Gab\c\Gab)^{(i)},
    \end{align*}
where we used the fact that $\psi_1,\psi_2\in\Gab$. Applying Proposition \ref{keyintegral}, we deduce for $i\leq 4$
\begin{align*}
\|\dkb^i\psi_1\|_{L^2_{sc}(\cuv)}^2+\|\dkb^i\psi_2\|^2_{L^2_{sc}(\ucuv)}&\les\|\dkb^i\psi_1\|_{L^2_{sc}(H_{u_\infty}^{(0,\ub)})}^2+\|\dkb^i\psi_2\|^2_{L^2_{sc}(\Hb_{0}^{(u_\infty,u)})}\\
&+\int_0^\ub\int_{u_\infty}^u\left\|\Gab^{(4)}\c(\Gab\c\Gab)^{(4)}\right\|_{L^1_{sc}(S_{u',\ub'})}du'd\ub'\\
&+\int_0^\ub\int_{u_\infty}^u\frac{a}{|u'|}\left\|\ef\Gab^{(4)}\c(\Gab\c\Gab)^{(4)}\right\|_{L^1_{sc}(S_{u',\ub'})}du'd\ub'.
\end{align*}
We have from the initial assumption that
\begin{align*}
    \|\psi_1^{(4)}\|_{L^2_{sc}(H_{u_\infty}^{(0,\ub)})}^2+\|\psi_2^{(4)}\|^2_{L^2_{sc}(\Hb_{0}^{(u_\infty,u)})}\les b^{-2}a.
\end{align*}
Applying \eqref{GabGabGab} in Theorem \ref{Junkmanthm}, we deduce from \eqref{bootsize}
\begin{align*}
    \int_0^\ub\int_{u_\infty}^u\left\|\Gab^{(4)}\c(\Gab\c\Gab)^{(4)}\right\|_{L^1_{sc}(S_{u',\ub'})}du'd\ub'\les b^\frac{3}{4}\af\les b^{-2}a.
\end{align*}
Recalling that $s_2(\ef)=0.5$, we have from Lemma \ref{Holder} and \eqref{dfsc}
\begin{align*}
&\int_0^\ub\int_{u_\infty}^u\frac{a}{|u'|}\left\|\ef\Gab^{(4)}\c(\Gab\c\Gab)^{(4)}\right\|_{L^1_{sc}(S_{u',\ub'})}du'd\ub'\\
\les\;&\int_0^\ub\int_{u_\infty}^u\frac{a}{|u'|^2}\|\ef\|_{L^\infty_{sc}}\left\|\Gab^{(4)}\c(\Gab\c\Gab)^{(4)}\right\|_{L^1_{sc}(S_{u',\ub'})}du'd\ub'\\
\les\;&\af\ef\int_0^\ub\int_{u_\infty}^u\left\|\Gab^{(4)}\c(\Gab\c\Gab)^{(4)}\right\|_{L^1_{sc}(S_{u',\ub'})}du'd\ub'\\
\les\;&b^\frac{3}{4}\ef a\les b^{-2}a,
\end{align*}
where we used \eqref{GabGabGab} in Theorem \ref{Junkmanthm} and \eqref{bootsize} at the last step. Combining the above estimates and applying Proposition \ref{ellipticLp}, this concludes the proof of Proposition \ref{estrbbF}.
\end{proof}
\subsubsection{Estimate for the Bianchi pair \texorpdfstring{$(\nabs_4\bF,\sld_1\bF)$}{}}
\begin{prop}\label{estbF4}
We have the following estimate:
\begin{align*}
    \left\|(\nabs_4\bF)^{(3)}\right\|_{L^2_{sc}(\cuv)}+\left\|(\sld_1\bF)^{(3)}\right\|_{L^2_{sc}(\ucuv)}\les\af.
\end{align*}
\end{prop}
\begin{proof}
    We have from Proposition \ref{Maxwell}
    \begin{align*}
    \nabs_3(\nabs_4\bF)+\frac{1}{2}\trchb\,(\nabs_4\bF)&=-\sld_1^*\psi_2+(\Gab\c\Gab)^{(1)}+\ef(\Gab\c\Gab)^{(1)},\\
    \nabs_4\psi_2&=\sld_1(\nabs_4\bF)+(\Gab\c\Gab)^{(1)}+\ef(\Gab\c\Gab)^{(1)},
    \end{align*}
    Differentiating it by $\dkb^i$ and applying Proposition \ref{commutation}, we infer
    \begin{align*}
    \nabs_3(\dkb^i(\nabs_4\bF))+\left(\frac{1+i}{2}\right)\trchb(\dkb^i(\nabs_4\bF))&=-\slD^*(\dkb^i\sld_1\bF)+\frac{|u|}{a}(\Gab\c\Gab)^{(i+1)}\\
    &+\ef(\Gab\c\Gab)^{(i+1)},\\
    \nabs_4(\dkb^i\sld_1\bF)&=\slD(\dkb^i(\nabs_4\bF))+(\Gab\c\Gab)^{(i+1)}\\
    &+\ef(\Gab\c\Gab)^{(i+1)},
    \end{align*}
where we used the fact that $\nabs_4\bF,\sld_1\bF\in\Gab$. Applying Proposition \ref{keyintegral}, we deduce for $i\leq 3$
\begin{align*}
&\|\dkb^i(\nabs_4\bF)\|_{L^2_{sc}(\cuv)}^2+\|\dkb^i\sld_1\bF\|^2_{L^2_{sc}(\ucuv)}\\
\les\;&\|\dkb^i(\nabs_4\bF)\|_{L^2_{sc}(H_{u_\infty}^{(0,\ub)})}^2+\|\dkb^i(\sld_1\bF)\|^2_{L^2_{sc}(\Hb_{0}^{(u_\infty,u)})}\\
+&\int_0^\ub\int_{u_\infty}^u\left\|\Gab^{(4)}\c(\Gab\c\Gab)^{(4)}\right\|_{L^1_{sc}(S_{u',\ub'})}+\frac{a}{|u'|}\left\|\ef\Gab^{(4)}\c(\Gab\c\Gab)^{(4)}\right\|_{L^1_{sc}(S_{u',\ub'})}du'd\ub'.
\end{align*}
We have from the initial assumption that
\begin{align*}
    \|(\nabs_4\bF)^{(3)}\|_{L^2_{sc}(H_{u_\infty}^{(0,\ub)})}^2+\|(\sld_1\bF)^{(3)}\|^2_{L^2_{sc}(\Hb_{0}^{(u_\infty,u)})}\les a.
\end{align*}
Next, proceeding as in Proposition \ref{estbrF}, we have
\begin{align*}
\int_0^\ub\int_{u_\infty}^u\left\|\Gab^{(4)}\c(\Gab\c\Gab)^{(4)}\right\|_{L^1_{sc}(S_{u',\ub'})}du'd\ub'&\les a,\\
\int_0^\ub\int_{u_\infty}^u\frac{a}{|u'|}\left\|\ef\Gab^{(4)}\c(\Gab\c\Gab)^{(4)}\right\|_{L^1_{sc}(S_{u',\ub'})}du'd\ub'&\les a.
\end{align*}
Combining the above estimates and applying Proposition \ref{ellipticLp}, this concludes the proof of Proposition \ref{estbF4}.
\end{proof}
\subsection{Estimates for electromagnetic potential \texorpdfstring{$\A$}{}}
\begin{prop}\label{propestA}
We have the following estimate:
\begin{align*}
\left\|\slA^{(4)}\right\|_{L^2_{sc}(\ucuv)}+\left\|\Ub^{(4)}\right\|_{L^2_{sc}(\ucuv)}\les\af.
\end{align*}
\end{prop}
\begin{proof}
We have from Corollary \ref{cordA=F}
\begin{align*}
\nabs_4\slA&=-\bF+\Gab\c\Gab,\\
\nabs_4\Ub&=-2\rhoF+\Gag\c\Gab.
\end{align*}
Differentiating them by $\dkb^i$ and applying Proposition \ref{commutation}, we deduce
\begin{align*}
    \nabs_4(\dkb^i\slA)&=(\bF)^{(i)}+(\Gab\c\Gab)^{(i)},\\
    \nabs_4(\dkb^i\Ub)&=(\rhoF)^{(i)}+(\Gab\c\Gab)^{(i)}.
\end{align*}
Applying Lemma \ref{evolution} and Proposition \ref{ellipticLp}, we obtain
\begin{align*}
    \|\Ub^{(4)}\|_{L^2_{sc}(S_{u,\ub})}&\les\int_0^\ub \|\rhoF^{(4)}\|_{L^2_{sc}(S_{u,\ub'})}+\|(\Gab\c\Gab)^{(4)}\|_{L^2_{sc}(S_{u,\ub'})}d\ub'.
\end{align*}
Multiplying it by $\frac{a}{|u|^2}$ and integrating it from $u_\infty$ to $u$, we infer
\begin{align*}
    \int_{u_\infty}^u\frac{a}{|u'|^2}\|\Ub^{(4)}\|_{L^2_{sc}(S_{u',\ub})}^2du'&\les\int_{u_\infty}^u\frac{a}{|u'|^2}du'\int_0^\ub \|\rhoF^{(4)}\|_{L^2_{sc}(S_{u',\ub'})}^2d\ub'\\
    &+\int_{u_\infty}^u\frac{a}{|u'|^2}du'\int_0^\ub\|(\Gab\c\Gab)^{(4)}\|_{L^2_{sc}(S_{u',\ub'})}^2d\ub'.
\end{align*}
Combining with \eqref{phifluxsc},
\begin{align*}
\|\Ub^{(4)}\|_{L^2_{sc}(\ucuv)}^2&\les\|\rhoF^{(4)}\|_{L^2_{sc}(\ucuv)}^2+a^{-1}\int_0^\ub\int_{u_\infty}^u\|(\Gab\c\Gab)^{(4)}\|_{L^2_{sc}(S_{u',\ub'})}^2du'd\ub'\\
&\les a+b\les a,
\end{align*}
where we used Propositions \ref{estbrF}, \ref{estbF4} and \eqref{GabGab} in Theorem \ref{Junkmanthm} at the second step. Next, we have from Lemma \ref{evolution} and Proposition \ref{ellipticLp}
\begin{align*}
    \|\Asl^{(4)}\|_{L^2_{sc}(S_{u,\ub})}&\les\int_0^\ub\|\bF^{(4)}\|_{L^2_{sc}(S_{u,\ub'})}+\|(\Gab\c\Gab)^{(4)}\|_{L^2_{sc}(S_{u,\ub'})}d\ub'\\
    &\les\left(\int_0^\ub\|\bF^{(4)}\|^2_{L^2_{sc}(S_{u,\ub'})}d\ub'\right)^\frac{1}{2}+\int_0^\ub\|(\Gab\c\Gab)^{(4)}\|_{L^2_{sc}(S_{u,\ub'})}d\ub'\\
    &\les\af+\int_0^\ub\|(\Gab\c\Gab)^{(4)}\|_{L^2_{sc}(S_{u,\ub'})}d\ub',
\end{align*}
where we used Proposition \ref{estbrF} at the last step. Multiplying it by $\frac{a}{|u|^2}$ and integrating it from $u_\infty$ to $u$, we infer
\begin{align*}
    \|\Asl^{(4)}\|^2_{L^2_{sc}(\ucuv)}&\les \int_{u_\infty}^u \frac{a^2}{|u'|^2}du'+\int_0^\ub\int_{u_\infty}^u\frac{a}{|u'|^2}\|(\Gab\c\Gab)^{(4)}\|^2_{L^2_{sc}(S_{u',\ub'})}du'd\ub'\\
    &\les \frac{a^2}{|u|}+a^{-1}\int_0^\ub\int_{u_\infty}^u\|(\Gab\c\Gab)^{(4)}\|^2_{L^2_{sc}(S_{u',\ub'})}du'd\ub'\\
    &\les a+b\les a,
\end{align*}
where we used \eqref{GabGab} in Theorem \ref{Junkmanthm} at the third step. This concludes the proof of Proposition \ref{propestA}.
\end{proof}
\subsection{Estimate for \texorpdfstring{$\OO_i[T](u,\ub)$}{}--norms}
\begin{prop}\label{est4FPsi}
    We have the following estimate:
    \begin{align*}
        \left\|(\rhoF,\siF,\bbF,\Psisl,\Psit_3)^{(3)}\right\|_{L^2_{sc}(S_{u,\ub})}\les b^{-1}\af.
    \end{align*}
\end{prop}
\begin{proof}
Given
\begin{align*}
    X\in\left\{\rhoF,\siF,\bbF,\Psisl,\Psit_3\right\},
\end{align*}
we have from Propositions \ref{Maxwell} and \ref{wave}
    \begin{align*}
        \nabs_4X=\afd(\bF,\rhoF,\siF,\Psi_4,\Psisl)^{(1)}+\Gab\c\Gab+\ef\Gab\c\Gab,
    \end{align*}
    which implies from Proposition \ref{commutation} that for $i\leq 3$
    \begin{align*}
        \nabs_4(\dkb^iX)=\afd(\bF,\rhoF,\siF,\Psi_4,\Psisl)^{(4)}+(\Gab\c\Gab)^{(3)}+\ef(\Gab\c\Gab)^{(3)}.
    \end{align*}
    Applying Propositions \ref{evolution}, \ref{Holder} and \ref{ellipticLp}, we infer\footnote{Here, we used the fact that $X=0$ on $\Hb_0$ and  $s_2(\ef)=0.5$.}
    \begin{align*}
        \|X^{(3)}\|_{L^2_{sc}(S_{u,\ub})}&\les \int_0^\ub \afd\|(\bF,\rhoF,\siF,\Psi_4,\Psisl)^{(4)}\|_{L^2_{sc}}d\ub'\\
        &+\int_0^\ub\|(\Gab\c\Gab)^{(3)}\|_{L^2_{sc}}+\|\ef(\Gab\c\Gab)^{(3)}\|_{L^2_{sc}}d\ub'\\
        &\les \afd\left(\int_0^\ub\|(\bF,\rhoF,\siF,\Psi_4,\Psisl)^{(4)}\|_{L^2_{sc}}^2d\ub'\right)^\frac{1}{2}\\
        &+\int_0^\ub \frac{1}{|u|}\|\Gab\|_{L^\infty_{sc}(S_{u,\ub'})}\|\Gab^{(3)}\|_{L^2_{sc}(S_{u,\ub'})}+\frac{\ef}{\af}\|\Gab\|_{L^\infty_{sc}(S_{u,\ub'})}\|\Gab^{(3)}\|_{L^2_{sc}(S_{u,\ub'})}d\ub'\\
        &\les \bo+\int_0^\ub \frac{\ef|u|}{\af}b^\frac{1}{2}a d\ub'\\
        &\les b^{-1}\af.
    \end{align*}
    This concludes the proof of Proposition \ref{est4FPsi}.
\end{proof}
\begin{prop}\label{est3FPsi}
    We have the following estimate:
    \begin{align*}
        \left\|(\Psi_4,\bF)^{(3)}\right\|_{L^2_{sc}(S_{u,\ub})}+\left\|\nabs_4\bF^{(2)}\right\|_{L^2_{sc}(S_{u,\ub})}+\left\|\psi\right\|_{L^2_{sc}(S_{u,\ub})}\les \af.
    \end{align*}
\end{prop}
\begin{proof}
Let $X\in\left\{\Psi_4,\bF\right\}$, we have
\begin{align*}
    \nabs_3X+\frac{1}{2}\trchb X=\afd(\rhoF,\siF,\Psisl)^{(1)}+\Gab\c\Gab+\ef\Gab\c\Gab\c\Gab.
\end{align*}
Differentiating it by $\dkb^i$ for $i\leq 3$ and applying Propositions \ref{commutation}, we deduce
\begin{align*}
    \nabs_3(\dkb^iX)+\frac{1+i}{2}\trchb(\dkb^iX)=\afd(\rhoF,\siF,\Psisl)^{(4)}+(\Gab\c\Gab)^{(3)}+\ef(\Gab\c\Gab)^{(3)}
\end{align*}
Applying Lemma \ref{3evolution}, we obtain
\begin{align*}
    \|X^{(3)}\|_{L^2(S_{u,\ub})}&\les \|X^{(3)}\|_{L^2(S_{u_\infty,\ub})}+\afd\int_{u_\infty}^u \frac{a}{|u'|^2}\|(\rhoF,\siF,\Psisl)^{(4)}\|_{L^2_{sc}(S_{u',\ub})}du'\\
    &+\int_{u_\infty}^u\frac{a}{|u'|^2}\|(\Gab\c\Gab)^{(3)}\|_{L^2_{sc}(S_{u',\ub})}+\frac{a}{|u'|^2}\|\ef(\Gab\c\Gab)^{(3)}\|_{L^2_{sc}(S_{u',\ub})}du'\\
    &\les\af+\afd\|(\rhoF,\siF,\Psisl)^{(4)}\|_{L^2_{sc}(\ucuv)}+\int_{u_\infty}^u\frac{a}{|u'|^4}\frac{\ef|u'|^2}{\af}ab^\frac{1}{2}du'\\
    &\les\af+\afd(\bo\af)+\frac{\ef a^\frac{3}{2}b^\frac{1}{2}}{|u|}\\
    &\les\af.
\end{align*}
The estimate for $\nabs_4\bF$ can be deduced in a similar way. Finally, we have
\begin{align*}
    \nabs_4\psi=\Psi_4.
\end{align*}
Applying Lemma \ref{evolution}, we infer
\begin{align*}
    \|\psi\|_{L^2_{sc}(S_{u,\ub})}\les \int_0^\ub\|\Psi_4\|_{L^2_{sc}(S_{u,\ub'})}d\ub'\les \af.
\end{align*}
This concludes the proof of Proposition \ref{est3FPsi}.
\end{proof}
\begin{prop}\label{estAS}
We have the following estimate:
\begin{align*}
    \left\|(\Asl,\Ub)^{(3)}\right\|_{L^2_{sc(S_{u,\ub})}}\les \af.
\end{align*}
\end{prop}
\begin{proof}
    We have from Corollary \ref{cordA=F}
    \begin{align*}
        \nabs_4\Asl&=-\bF+\Gab\c\Gab,\\
        \nabs_4\Ub&=-2\rhoF+\Gag\c\Gab.
    \end{align*}
    Differentiating it by $\dkb^i$ for $i\leq 3$ and applying Proposition \ref{commutation}, we infer
    \begin{align*}
        \nabs_4(\dkb^i\Asl)&=\bF^{(3)}+(\Gab\c\Gab)^{(3)},\\
        \nabs_4(\dkb^i\Ub)&=\rhoF^{(3)}+(\Gab\c\Gab)^{(3)}.
    \end{align*}
    Applying Lemma \ref{evolution} and Proposition \ref{ellipticLp}, we obtain
\begin{align*}
    \|(\Asl,\Ub)^{(3)}\|_{L^2_{sc}(S_{u,\ub})}&\les\int_0^\ub \|(\bF,\rhoF)^{(3)}\|_{L^2_{sc}(S_{u,\ub'})}+\|(\Gab\c\Gab)^{(3)}\|_{L^2_{sc}(S_{u,\ub'})}d\ub'\\
    &\les\int_0^\ub\af+\frac{1}{|u|}\|\Gab\|_{L^\infty_{sc}(S_{u,\ub'})}\|\Gab^{(3)}\|_{L^2_{sc}(S_{u,\ub'})}d\ub'\\
    &\les\af+\int_0^\ub\frac{1}{|u|}b^\frac{1}{2}a d\ub'\\
    &\les\af.
\end{align*}
This concludes the proof of Proposition \ref{estAS}.
\end{proof}
Combining Propositions \ref{estbrF}--\ref{estAS}, this concludes the proof of Theorem \ref{M1}.
\section{Curvature estimates}\label{secR}
In this section, we prove Theorem \ref{M1} using the $|u|^p$--weighted estimate method developed in Section \ref{secBianchi}.
\subsection{Estimate for the Bianchi pair \texorpdfstring{$(\a,\bt)$}{}}
\begin{prop}\label{estab}
We have the following estimate:
\begin{align*}
    \|\a^{(3)}\|_{L^2_{sc}(\cuv)}+\|\bt^{(3)}\|_{L^2_{sc}(\ucuv)}\les\af.
\end{align*}
\end{prop}
\begin{proof}
We have from Proposition \ref{BianchiKcsic}
\begin{align*}
\nabs_3\a+\frac{1}{2}\trchb\,\a&=-2\sld_2^*\bt+(\Gab\c\Gab)^{(1)},\\
\nabs_4\bt&=\sld_2\a+(\Gag\c\Gab)^{(1)}.
\end{align*}
Commuting it with $\dkb^i$ and applying Proposition \ref{commutation}, we infer
\begin{align*}
\nabs_3(\dkb^i\a)+\frac{1+i}{2}\trchb(\dkb^i\a)&=-2\sld_2^*(\dkb^i\b)+\frac{|u|}{a}(\Gab\c\Gab)^{(i+1)},\\
\nabs_4(\dkb^i\bt)&=\sld_2(\dkb^i\a)+(\Gab\c\Gab)^{(i+1)}.
\end{align*}
Applying Proposition \ref{keyintegral} with $\psi_{(1)}=\dkb^i\a$ and $\psi_{(2)}=\dkb^i\bt$, we deduce for $i\leq 3$
\begin{align*}
&\|\dkb^i\a\|_{L^2_{sc}(\cuv)}^2+\|\dkb^i\bt\|^2_{L^2_{sc}(\ucuv)}\\
\les&\|\dkb^i\a\|_{L^2_{sc}(H_{u_\infty}^{(0,\ub)})}^2+\int_0^\ub\int_{u_\infty}^u\left\|\Gab^{(4)}\c(\Gab\c\Gab)^{(4)}\right\|_{L^1_{sc}(S_{u',\ub'})}du'd\ub'.
\end{align*}
Notice that we have from Theorem \ref{Junkmanthm}
\begin{align*}
\int_0^\ub\int_{u_\infty}^u\left\|\Gab^{(4)}\c(\Gab\c\Gab)^{(4)}\right\|_{L^1_{sc}(S_{u',\ub'})}du'd\ub'\les b^3\af\les a.
\end{align*}
Moreover, we have from the initial assumption that for $i\leq 3$
\begin{align*}
\|\dkb^i\a\|_{L^2_{sc}(\cuv)}^2+\|\dkb^i\bt\|^2_{L^2_{sc}(\ucuv)}\les a.
\end{align*}
Combining with Proposition \ref{ellipticLp}, this concludes the proof of Proposition \ref{estab}.
\end{proof}
\subsection{Estimates for \texorpdfstring{$L^2_{sc}(S_{u,\ub})$}{}--norms of \texorpdfstring{$\a$}{}, \texorpdfstring{$\hch$}{} and \texorpdfstring{$\hchb$}{}}\label{secGabimproved}
\begin{prop}\label{esthchhchb}
We have the following estimates:
\begin{align}
    \|\hch^{(3)}\|_{L^2_{sc}(S_{u,\ub})}&\les\af,\label{esthch}\\
    \|\a^{(2)}\|_{L^2_{sc}(S_{u,\ub})}&\les\af,\label{esta}\\
    \frac{a}{|u|}\|\hchb^{(3)}\|_{L^2_{sc}(S_{u,\ub})}&\les\af.\label{esthchb}
\end{align}
\end{prop}
\begin{proof}[Proof of Proposition \ref{esthchhchb}]
    We have from Proposition \ref{null}
    \begin{align*}
        \nabs_4\hch=-\a+\Gag\c\Gab.
    \end{align*}
    Differentiating it by $\dkb^i$ and applying Proposition \ref{commutation}, we infer
    \begin{align*}
        \nabs_4(\dkb^i\hch)=\a^{(i)}+(\Gab\c\Gab)^{(i)}.
    \end{align*}
    Applying Lemma \ref{evolution} and Proposition \ref{estab}, we deduce for $i\leq 3$
    \begin{align*}
        \|\dkb^i\hch\|_{L^2_{sc}(S_{u,\ub})}&\les\int_0^\ub\|\a^{(3)}\|_{L^2_{sc}(S_{u,\ub'})}+\|(\Gab\c\Gab)^{(3)}\|_{L^2_{sc}(S_{u,\ub'})}d\ub'\\
        &\les\|\a^{(i)}\|_{L^2_{sc}(\cuv)}+\frac{1}{|u|}\int_0^\ub\|\Gab^{(3)}\|_{L^2_{sc}(S_{u,\ub'})}\|\Gab\|_{L^\infty_{sc}(S_{u,\ub'})}d\ub'\\
        &\les\af+\frac{b^\frac{1}{2}a}{|u|}\les\af.
    \end{align*}
    Combining with Proposition \ref{ellipticLp}, we obtain \eqref{esthch}. Next, we have from Proposition \ref{bianchiequations}
    \begin{align*}
        \nabs_3\a+\frac{1}{2}\trchb\,\a=\afd\dkb\bt+(\Gag\c\Gab)^{(1)}.
    \end{align*}
    Differentiating it by $\dkb^{i}$ and applying Proposition \ref{commutation}, we infer for $i\leq 2$
    \begin{align*}
        \nabs_3(\dkb^i\a)+\frac{1+i}{2}\trchb(\dkb^i\a)=\afd\dkb^{i+1}\bt+\frac{|u|}{a}(\Gab\c\Gab)^{(3)}.
    \end{align*}
    Applying Lemma \ref{3evolution} and Proposition \ref{Holder}, we deduce for $i\leq 2$
    \begin{align*}
        \|\dkb^i\a\|_{L^2_{sc}(S_{u,\ub})}&\les\|\dkb^i\a\|_{L^2_{sc}(S_{u_\infty,\ub})}+\int_{u_\infty}^u\frac{\af}{|u'|^2}\|\bt^{(3)}\|_{L^2_{sc}(S_{u',\ub})}du'\\
        &+\int_{u_\infty}^u\frac{1}{|u'|^2}\|\Gab^{(3)}\|_{L^2_{sc}(S_{u',\ub})}\|\Gab\|_{L^\infty_{sc}(S_{u',\ub})}du'\\
        &\les\af+\left(\int_{u_\infty}^u\frac{a}{|u'|^2}\|\bt^{(3)}\|^2_{L^2_{sc}(S_{u',\ub})}du'\right)^\frac{1}{2}\left(\int_{u_\infty}^u\frac{1}{|u'|^2}du'\right)^\frac{1}{2}+\int_{u_\infty}^u\frac{b^\frac{1}{2}a}{|u'|^2}du'\\
        &\les\af+\afd\|\bt^{(3)}\|_{L^2_{sc}(\ucuv)}+b^\frac{1}{2}\\
        &\les\af.
    \end{align*}
    Combining with Proposition \ref{ellipticLp}, we obtain \eqref{esta}. Next, we have from Proposition \ref{null}
    \begin{align*}
        \nabs_4\hchb=-\frac{1}{2}\trchb\,\hch+\nabs\hot\etab+\frac{|u|}{a}\Gab\c\Gab.
    \end{align*}
    Differentiating it by $\dkb^i$ and applying Proposition \ref{commutation}, we obtain
    \begin{align*}
        \nabs_4(\dkb^i\hchb)=\trchb\,\hch^{(i)}+\afd\etab^{(i+1)}+\frac{|u|}{a}(\Gab\c\Gab)^{(i)}.
    \end{align*}
    Applying Lemma \ref{evolution} and Proposition \ref{Holder}, we infer for $i\leq 3$\footnote{Recall that we have $\|\trchb\|_{L^\infty_{sc}(S_{u,\ub})}=\frac{|u|^3}{a}\|\trchb\|_{L^\infty(S_{u,\ub})}\les \frac{|u|^2}{a}$.}
    \begin{align*}
        \|\dkb^i\hchb\|_{L^2_{sc}(S_{u,\ub})}&\les\int_{0}^\ub\|\trchb\,\hch^{(3)}\|_{L^2_{sc}(S_{u,\ub'})}+\afd\|\etab^{(4)}\|_{L^2_{sc}(S_{u,\ub'})}+\frac{|u|}{a}\|(\Gab\c\Gab)^{(3)}\|_{L^2_{sc}(S_{u,\ub'})}d\ub'\\
        &\les\int_0^\ub \frac{1}{|u|}\|\trchb\|_{L^\infty_{sc}(S_{u,\ub'})}\|\hch^{(3)}\|_{L^2_{sc}(S_{u,\ub'})} d\ub'+\int_0^\ub\afd\af\bo d\ub'\\
        &+\int_0^\ub a^{-1}\|\Gab\|_{L^\infty_{sc}(S_{u,\ub'})}\|\Gab^{(3)}\|_{L^2_{sc}(S_{u,\ub'})}d\ub'\\
        &\les \frac{|u|}{a}\int_0^\ub\|\hch^{(3)}\|_{L^2_{sc}(S_{u,\ub'})}d\ub'+\bo+b^\frac{1}{2}\\
        &\les \frac{|u|}{\af},
    \end{align*}
    where we used \eqref{esthch} at the last step. Combining with Proposition \ref{ellipticLp}, we obtain \eqref{esthchb}. This concludes the proof of Proposition \ref{esthchhchb}.
\end{proof}
\subsection{Estimate for the Bianchi pair \texorpdfstring{$(\b,(-\rho,\si))$}{}}
\begin{prop}\label{estbr}
We have the following estimate:
\begin{align*}
    \|\b^{(3)}\|_{L^2_{sc}(\cuv)}+\|(K,\sic)^{(3)}\|_{L^2_{sc}(\ucuv)}\les 1.
\end{align*}
\end{prop}
\begin{proof}
We have from Proposition \ref{bianchi}
\begin{align*}
\nabs_3\b+\trchb\,\b&=\sld_1^*(\K,\sic)+\frac{|u|}{ab}(\Gab\c\Gab)^{(1)},\\
\nabs_4(\K,\sic)&=-\sld_1\b-\frac{1}{2}\trchb|\hch|^2+\bo\afd(\Gab\c\Gab)^{(1)}.
\end{align*}
Differentiating it by $\dkb^i$ and applying Proposition \ref{commutation}, we infer for $i\leq 3$
\begin{align*}
    \nabs_3(\dkb^i\b)+\frac{2+i}{2}\trchb(\dkb^i\psi_1)&=\slD^*(\dkb^i(K,\sic))+\frac{|u|}{ab}(\Gab\c\Gab)^{(4)},\\
    \nabs_4(\dkb^i(\K,\sic))&=-\slD(\dkb^i\b)+\trchb(\hch\c\hch)^{(3)}+\frac{\bo}{\af}(\Gab\c\Gab)^{(4)}.
\end{align*}
Applying Proposition \ref{keyintegral}, we deduce for $i\leq 3$
\begin{align*}
&\|\dkb^i\b\|_{L^2_{sc}(\cuv)}^2+\|\dkb^i(\K,\sic)\|^2_{L^2_{sc}(\ucuv)}\\
\les\;&\|\dkb^i\b\|_{L^2_{sc}(\cuvi)}^2+\int_0^\ub\int_{u_\infty}^u \afd b^{-1}\|\Gab^{(4)}\c(\Gab\c\Gab)^{(4)}\|_{L^1_{sc}(S_{u',\ub'})}\\
+&\int_0^\ub\int_{u_\infty}^u\frac{a}{|u|^2}\|(\K,\sic)^{(3)}\|_{L^2_{sc}(S_{u',\ub'})}\|\trchb(\hch\c\hch)^{(3)}\|_{L^2_{sc}(S_{u',\ub'})}du'd\ub'.
\end{align*}
We first have from Theorem \ref{Junkmanthm}
\begin{align*}
    \int_0^\ub\int_{u_\infty}^u\afd b^{-1}\|\Gab^{(4)}\c(\Gab\c\Gab)^{(4)}\|_{L^1_{sc}(S_{u',\ub'})}du'd\ub'\les \afd b^{-1}\af b^\frac{3}{4}\les 1.
\end{align*}
Next, we have from Proposition \ref{Holder} and \eqref{esthch}
\begin{align*}
    &\int_0^\ub\int_{u_\infty}^u\frac{a}{|u'|^2}\|(\K,\sic)^{(3)}\|_{L^2_{sc}(S_{u',\ub'})}\|\trchb(\hch\c\hch)^{(3)}\|_{L^2_{sc}(S_{u',\ub'})}du'd\ub'\\
    \les\;&\int_0^\ub\int_{u_\infty}^u\frac{a}{|u'|^4}\|(\K,\sic)^{(3)}\|_{L^2_{sc}(S_{u',\ub'})}\|\trchb\|_{L^\infty_{sc}(S_{u',\ub'})}\|\hch\|_{L^\infty_{sc}(S_{u',\ub'})}\|\hch^{(3)}\|_{L^2_{sc}(S_{u',\ub'})}du'd\ub'\\
    \les\;&\int_0^\ub\int_{u_\infty}^u\frac{a}{|u'|^2}\|(\K,\sic)^{(3)}\|_{L^2_{sc}(S_{u',\ub'})}du'd\ub'\\
    \les\;&\de_0\int_0^\ub\int_{u_\infty}^u\frac{a}{|u'|^2}\|(\K,\sic)^{(3)}\|^2_{L^2_{sc}(S_{u',\ub'})}du'd\ub'+\de_0^{-1}\int_0^\ub\int_{u_\infty}^u\frac{a}{|u'|^2}du'd\ub'\\
    \les\;&\de_0\|(\K,\sic)^{(3)}\|^2_{L^2_{sc}(\ucuv)}+\de_0^{-1},
\end{align*}
where $0<\de_0\ll 1$ is a constant independent of $a$, $b$ and $\ef$. Moreover, we have from the initial assumption that
\begin{align*}
    \int_0^\ub\|\b^{(3)}\|_{L^2_{sc}(S_{u_\infty,\ub'})}^2d\ub'\les 1.
\end{align*}
Combining the above estimates and Proposition \ref{ellipticLp}, we deduce
\begin{align*}
    \|\b^{(3)}\|_{L^2_{sc}(\cuv)}^2+\|(\K,\sic)^{(3)}\|^2_{L^2_{sc}(\ucuv)}\les 1+\de_0\|(\K,\sic)^{(3)}\|^2_{L^2_{sc}(\ucuv)}+\de_0^{-1}.
\end{align*}
Thus, we obtain for $\de_0$ small enough
\begin{align*}
    \|\b^{(3)}\|_{L^2_{sc}(\cuv)}^2+\|(\K,\sic)^{(3)}\|^2_{L^2_{sc}(\ucuv)}\les1+\de_0^{-1}\les 1.
\end{align*}
This concludes the proof of Proposition \ref{estbr}.
\end{proof}
\subsection{Estimate for the Bianchi pair \texorpdfstring{$((\Kc,\sic),\bb)$}{}}
\begin{prop}\label{estrb}
We have the following estimate:
\begin{align*}
    \|(\Kc,\sic)^{(3)}\|_{L^2_{sc}(\cuv)}+\|\bb^{(3)}\|_{L^2_{sc}(\ucuv)}\les 1.
\end{align*}
\end{prop}
\begin{proof}
We have from Proposition \ref{bianchi}
\begin{align}
\begin{split}\label{Ksibbstartest}
\nabs_3(\Kc,-\sic)+\frac{3}{2}\trchb(\Kc,-\sic)&=\sld_1\bb+\frac{1}{2}\trchb\,\mu+\frac{|u|}{ab}(\Gab\c\Gab)^{(1)},\\
\nabs_4\bb&=-\sld_1^*(\Kc,-\sic)+\frac{|u|}{ab}(\Gab\c\Gab)^{(1)}.
\end{split}
\end{align}
Notice that we have
\begin{align*}
    \trchb\,\mu=-\frac{2}{\Om|u|}\mu+\trchbt\,\mu=-\frac{2}{|u|}\mu+\Gag\c\Gag^{(1)}.
\end{align*}
Hence, we obtain for $1\leq i\leq 3$
\begin{align*}
    \dkb^i(\trchb\,\mu)=-\frac{2}{|u|}(\dkb^i\mu)+(\Gag\c\Gag)^{(i+1)}.
\end{align*}
Recalling from Definition \ref{gammag} that
\begin{align*}
    \mu\in \Gag^{(1)},\qquad\quad\dkb\mu\in \frac{a}{|u|}\Gag^{(2)},\qquad\quad\trchb\in \frac{|u|^2}{a^2}\Gag,\qquad \quad\dkb\trchb\in\Gag^{(1)},
\end{align*}
we obtain for $1\leq i\leq 3$
\begin{align*}
    \dkb^i(\trchb\,\mu)&=\dkb^{i-1}(\trchb(\dkb\mu)+(\dkb\trchb)\mu)\\
    &=\left(\frac{|u|^2}{a^2}\frac{a}{|u|}\Gag\c\Gag^{(2)}+\Gag^{(1)}\c\Gag^{(1)}\right)^{(2)}\\
    &=\frac{|u|}{a}(\Gag\c\Gag)^{(4)}\\
    &=\frac{|u|}{a^\frac{3}{2}}(\Gab\c\Gab)^{(4)}.
\end{align*}
Differentiating \eqref{Ksibbstartest} by $\dkb^i$ and applying Proposition \ref{commutation}, we infer for $1\leq i\leq 3$
\begin{align*}
    \nabs_3(\dkb^i(\Kc,-\sic))+\frac{3+i}{2}\trchb(\dkb^i(\Kc,-\sic))&=\slD(\dkb^i\bb)+\frac{|u|}{ab}(\Gab\c\Gab)^{(4)},\\
    \nabs_4(\dkb^i\bb)&=-\slD^*(\dkb^i(\Kc,-\sic))+\frac{|u|}{ab}(\Gab\c\Gab)^{(4)}.
\end{align*}
Applying Proposition \ref{keyintegral}, we obtain for $1\leq i\leq 3$
\begin{align*}
&\|\dkb^i(\Kc,\sic)\|_{L^2_{sc}(\cuv)}^2+\|\dkb^i\bb\|^2_{L^2_{sc}(\ucuv)}\\
\les\;&\|\dkb^i(\Kc,\sic)\|_{L^2_{sc}(\cuvi)}^2+\int_0^\ub\int_{u_\infty}^u\afd b^{-1}\|\Gab^{(4)}\c(\Gab\c\Gab)^{(4)}\|_{L^1_{sc}(S_{u',\ub'})}du'd\ub'.
\end{align*}
Applying Theorem \ref{Junkmanthm}, we have
\begin{align*}
    \int_0^\ub\int_{u_\infty}^u\afd b^{-1}\|\Gab^{(4)}\c(\Gab\c\Gab)^{(4)}\|_{L^1_{sc}(S_{u',\ub'})}du'd\ub'\les\afd b^{-1}\af b^\frac{3}{4}\les 1.
\end{align*}
We also have from the initial assumption that
\begin{align*}
    \|\dkb^i(\Kc,\sic)\|_{L^2_{sc}(\cuvi)}^2\les 1.
\end{align*}
Thus, we deduce for $1\leq i\leq 3$
\begin{align*}
    \|\dkb^i(\Kc,\sic)\|_{L^2_{sc}(\cuv)}^2+\|\dkb^i\bb\|^2_{L^2_{sc}(\ucuv)}\les 1.
\end{align*}
Moreover, we have from Gauss-Bonnet Theorem and \eqref{torsion}
\begin{align*}
    \left|\ov{\Kc}\right|=\left|\frac{1}{r^2}-\frac{1}{|u|^2}\right|\les\frac{1}{|u|^3},\qquad\quad \ov{\sic}=\ov{\curls\eta}=0.
\end{align*}
Combining with Proposition \ref{ellipticLp}, we infer
\begin{align*}
    \|(\Kc,\sic)^{(3)}\|_{L^2_{sc}(\cuv)}^2+\|\bb^{(3)}\|^2_{L^2_{sc}(\ucuv)}\les 1.
\end{align*}
This concludes the proof of Proposition \ref{estrb}.
\end{proof}
\subsection{Estimate for \texorpdfstring{$\OO_i[W](u,\ub)$}{}--norms}
\begin{prop}\label{estR}
    We have the following estimate:
    \begin{align*}
        \|(\b,\rho,\si,\bb,\bt,\Kc,\sic)^{(2)}\|_{L^2_{sc}(S_{u,\ub})}\les 1.
    \end{align*}
\end{prop}
\begin{proof}
    We have from Proposition \ref{bianchi}\footnote{Here, we used $\a\in\Gab^{(1)}$ and we ignore all the terms which have better decay properties.}
    \begin{align*}
        \nabs_4(\bt,\Kc,\sic,\bb)=-\frac{1}{|u|^2}\trch+\afd\a^{(1)}+\trchb(\hch\c\hch)^{(0)}+\frac{|u|}{ab}\Gab\c\Gab^{(1)}.
    \end{align*}
    Differentiating it by $\dkb^{i}$ and applying Proposition \ref{commutation}, we infer
    \begin{align*}
        \nabs_4(\dkb^{i}(\bt,\Kc,\sic,\bb))=\afd\a^{(i+1)}+\frac{1}{|u|^2}\dkb^i\trch+\trchb(\hch\c\hch)^{(i)}+\frac{|u|}{ab}(\Gab\c\Gab)^{(i+1)}.
    \end{align*}
    Applying Lemma \ref{evolution}, Propositions \ref{Holder} and \eqref{esthch}, we deduce for $1\leq i\leq 2$
    \begin{align*}
        \|\dkb^{i}(\bt,\Kc,\sic,\bb)\|_{L^2_{sc}(S_{u,\ub})}&\les\int_0^\ub\afd\|\a^{(3)}\|_{L^2_{sc}(S_{u,\ub'})}+a^{-1}\|\trcht^{(3)}\|_{L^2_{sc}(S_{u,\ub'})}d\ub'\\
        &+a^{-1}\int_0^\ub\|\hch\|_{L^\infty_{sc}(S_{u,\ub'})}\|\hch^{(3)}\|_{L^2_{sc}(S_{u,\ub'})}d\ub'\\
        &+a^{-1}b^{-1}\int_0^\ub\|\Gab^{(3)}\|_{L^2_{sc}(S_{u,\ub'})}\|\Gab\|_{L^\infty_{sc}(S_{u,\ub'})} d\ub'\\
        &\les\afd\left(\int_0^\ub\|\a^{(3)}\|^2_{L^2_{sc}(S_{u,\ub'})} d\ub'\right)^\frac{1}{2}+a^{-1}\bo+a^{-1}\af\af\\
        &\les 1,
    \end{align*}
    where we used Proposition \ref{estab} at the last step. Combining with Proposition \ref{ellipticLp} and the fact that $\ov{\Kc}\les|u|^{-3}$, we obtain
    \begin{align}\label{btKcbbL2S}
        \|(\bt,\Kc,\bb)^{(2)}\|_{L^2_{sc}(S_{u,\ub})}\les 1.
    \end{align}
    Next, we recall from \eqref{gauss} and \eqref{dfKcsic}
    \begin{align*}
        \rho=-\K-\frac{1}{4}\trch\,\trchb+\frac{1}{2}\hch\c\hchb+\frac{1}{2}\tr\Ssl,\qquad \si=\sic+\frac{1}{2}\hch\wedge\hchb,\qquad \b=\bt+\frac{1}{2}\Ssl_{4}.
    \end{align*}
    Combining with Theorem \ref{M1}, Proposition \ref{esthchhchb} and \eqref{btKcbbL2S}, we deduce
    \begin{align*}
        \|(\rho,\si,\b)^{(2)}\|_{L^2_{sc}(S_{u,\ub})}\les 1.
    \end{align*}
    This concludes the proof of Proposition \ref{estR}.
\end{proof}
Combining Propositions \ref{estab}--\ref{estR}, this concludes the proof of Theorem \ref{M2}.
\begin{rk}
    Combining Proposition \ref{esthchhchb} and Theorems \ref{M1} and \ref{M2}, we obtain the following estimates\footnote{Here, we also used Proposition \ref{sobolev}.}
    \begin{align}\label{improvedGab}
        \|\Gab^{(3)}\|_{L^2_{sc}(S_{u,\ub})}\les \af,\qquad \|\Gab^{(1)}\|_{L^\infty_{sc}(S_{u,\ub})}\les\af,
\end{align}
which improves the a priori estimate for $\Gab$ in \eqref{Gabapriori}. In the remainder of this paper, we will use the improved decay \eqref{improvedGab} for $\Gab$, which behaves $b^{-\frac{1}{4}}$ better than the a priori estimate in Proposition \ref{decayGagGab}, which follows directly from the bootstrap assumption \eqref{B1}.
\end{rk}
\section{Estimates for Ricci coefficients}\label{secO}
The goal of this section is to prove Theorem \ref{M3}.
\subsection{Estimate for \texorpdfstring{$\mo_i[G](u,\ub)$}{}}
\subsubsection{Estimate for \texorpdfstring{$(\om,\omd)$}{}}
\begin{prop}\label{estom}
    We have the following estimate:
    \begin{align*}
        \|(\om,\omd)^{(3)}\|_{L^2_{sc}(S_{u,\ub})}\les 1.
    \end{align*}
\end{prop}
\begin{proof}
    We have from Proposition \ref{null} and \eqref{dfomd}
    \begin{align*}
        \nabs_3(\om,-\omd)=-\frac{1}{2}(\K,\si)-\frac{1}{8}(\trch\trchb,0)+\frac{|u|}{a}\Gab\c\Gab.
    \end{align*}
    Differentiating it by $\dkb^i$ and applying Proposition \ref{commutation}, we infer
    \begin{align*}
        \nabs_3(\dkb^i(\om,-\omd))+\frac{i}{2}\trchb(\dkb^i(\om,-\omd))=(\K,\si)^{(i)}+(\trch\trchb)^{(i)}+\frac{|u|}{a}(\Gab\c\Gab)^{(i)}.
    \end{align*}
    Note that we have from Propositions \ref{Holder} and \ref{sobolev}
    \begin{align*}
        \|(\trch\trchb)^{(3)}\|_{L^2_{sc(S_{u,\ub})}}&\les \|(\trch)^{(3)}\trch\|_{L^2_{sc(S_{u,\ub})}}+\|\trch(\trchb)^{(3)}\|_{L^2_{sc(S_{u,\ub})}}\\
        &\les\frac{1}{|u|}\|\trch^{(3)}\|_{L^2_{sc(S_{u,\ub})}}\|\trchb^{(3)}\|_{L^2_{sc(S_{u,\ub})}}\\
        &\les\frac{1}{|u|}\frac{|u|^2}{a}\\
        &\les\frac{|u|}{a}.
    \end{align*}
    Applying Lemma \ref{3evolution}, we deduce from Proposition \ref{estbr} and \eqref{improvedGab} that for $i\leq 3$
    \begin{align*}
        &\quad\,\,|u|^{-1}\|\dkb^i(\om,\omd)\|_{L^2_{sc}(S_{u,\ub})}\\
        &\les|u_\infty|^{-1}\|\dkb^i(\om,\omd)\|_{L^2_{sc}(S_{u_\infty,\ub})}+\int_{u_\infty}^u\frac{a}{|u'|^3}\|(\K,\si)^{(3)}\|_{L^2_{sc}(S_{u',\ub})}du'\\
        &+\int_{u_\infty}^u\frac{a}{|u'|^3}\|(\trch\trchb)^{(3)}\|_{L^2_{sc}(S_{u',\ub})}+\frac{1}{|u'|^2}\|(\Gab\c\Gab)^{(3)}\|_{L^2_{sc}(S_{u',\ub})}du'\\
        &\les\frac{1}{|u_\infty|}+\left(\int_{u_\infty}^u\frac{a}{|u'|^2}\|(\K,\si)^{(3)}\|^2_{L^2_{sc}(S_{u',\ub})}du'\right)^\frac{1}{2}\left(\int_{u_\infty}^u\frac{a}{|u'|^4}du'\right)^\frac{1}{2}+\int_{u_\infty}^u\frac{a}{|u'|^3}du'\\
        &\les\frac{1}{|u|}+\frac{\af}{|u|^\frac{3}{2}}+\frac{a}{|u|^2}\\
        &\les\frac{1}{|u|}.
    \end{align*}
    Combining with Proposition \ref{ellipticLp}, this concludes the proof of Proposition \ref{estom}.
\end{proof}
\subsubsection{Estimate for \texorpdfstring{$\eta$}{}}
\begin{prop}\label{esteta}
    We have the following estimate:
    \begin{align*}
        \|\eta^{(3)}\|_{L^2_{sc}(S_{u,\ub})}\les 1.
    \end{align*}
\end{prop}
\begin{proof}
    We have from Proposition \ref{null}
    \begin{align*}
        \nabs_4\eta=-\b+\Gab\c\Gab.
    \end{align*}
    Differentiating it by $\dkb^i$ and applying Proposition \ref{commutation}, we infer
    \begin{align*}
        \nabs_4(\dkb^i\eta)=\b^{(i)}+(\Gab\c\Gab)^{(i)}.
    \end{align*}
    Applying Lemma \ref{evolution}, we deduce for $i\leq 3$
    \begin{align*}
        \|\dkb^i\eta\|_{L^2_{sc}(S_{u,\ub})}&\les\int_{0}^\ub\|\b^{(3)}\|_{L^2_{sc}(S_{u,\ub'})}+\|(\Gab\c\Gab)^{(3)}\|_{L^2_{sc}(S_{u,\ub'})}d\ub'\\
        &\les\left(\int_{0}^\ub\|\b^{(3)}\|^2_{L^2_{sc}(S_{u,\ub'})}d\ub'\right)^\frac{1}{2}+\int_0^\ub \frac{1}{|u|}\|\Gab^{(3)}\|_{L^2_{sc}(S_{u,\ub'})}\|\Gab\|_{L^\infty_{sc}(S_{u,\ub'})} d\ub'\\
        &\les 1+\frac{a}{|u|}\\
        &\les 1,
    \end{align*}
    where we used Proposition \ref{estrb} and \eqref{improvedGab} at the third step. Combining with Proposition \ref{ellipticLp}, this concludes the proof of Proposition \ref{esteta}.
\end{proof}
\subsubsection{Estimate for \texorpdfstring{$\trcht$}{}}
\begin{prop}\label{esttrcht}
    We have the following estimate:
    \begin{align*}
        \frac{|u|}{a}\|\trcht^{(3)}\|_{L^2_{sc}(S_{u,\ub})}\les 1.
    \end{align*}
\end{prop}
\begin{proof}
    We have from Proposition \ref{null}
    \begin{align*}
        \nabs_4\trcht=\Gab\c\Gab.
    \end{align*}
    Differentiating it by $\dkb^i$ and applying Proposition \ref{commutation}, we infer
    \begin{align*}
        \nabs_4(\dkb^i\trcht)=(\Gab\c\Gab)^{(i)}.
    \end{align*}
    Applying Lemma \ref{evolution}, we deduce for $i\leq 3$
    \begin{align*}
        \|\dkb^i\trcht\|_{L^2_{sc}(S_{u,\ub})}&\les\int_{0}^\ub\|(\Gab\c\Gab)^{(3)}\|_{L^2_{sc}(S_{u,\ub'})}d\ub'\\
        &\les\int_0^\ub \frac{1}{|u|}\|\Gab^{(3)}\|_{L^2_{sc}(S_{u,\ub'})}\|\Gab\|_{L^\infty_{sc}(S_{u,\ub'})} d\ub'\\
        &\les \frac{a}{|u|},
    \end{align*}
    where we used Proposition \ref{esthchhchb} at the last step. Combining with Proposition \ref{ellipticLp}, this concludes the proof of Proposition \ref{esttrcht}.
\end{proof}
\subsubsection{Estimate for \texorpdfstring{$\trchbt$}{}}
\begin{prop}\label{esttrchbc}
    We have the following estimate:
    \begin{align*}
        \left\|\trchbt^{(3)}\right\|_{L^2_{sc}(S_{u,\ub})}\les1.
    \end{align*}
\end{prop}
\begin{proof}
    We have from Proposition \ref{null}
    \begin{align*}
        \nabs_3(\trchbt)+\frac{1}{2}\trchb\,\trchbt=\frac{|u|^2}{a^2}\Gab\c\Gab.
    \end{align*}
    Differentiating it by $\dkb^i$ and applying Proposition \ref{commutation}, we infer
    \begin{align*}
        \nabs_3(\dkb^i\trchbc)+\frac{2+i}{2}\trchb(\dkb^i\trchbc)=\frac{|u|^2}{a^2}(\Gab\c\Gab)^{(i)}.
    \end{align*}
    Applying Lemma \ref{3evolution} and Proposition \ref{esthchhchb}, we deduce for $i\leq 3$
    \begin{align*}
        |u|^{-1}\|\dkb^i\trchbt\|_{L^2_{sc}(S_{u,\ub})}&\les|u_\infty|^{-1}\|\dkb^i\trchbt\|_{L^2_{sc}(S_{u_\infty,\ub})}+\int_{u_\infty}^u\frac{1}{a|u'|}\|(\Gab\c\Gab)^{(3)}\|_{L^2_{sc}(S_{u',\ub})}du'\\
        &\les|u_\infty|^{-1}+\int_{u_\infty}^u\frac{1}{|u'|^2}du'\\
        &\les |u|^{-1}.
    \end{align*}
    Combining with Proposition \ref{ellipticLp}, this concludes the proof of Proposition \ref{esttrchbc}.
\end{proof}
\subsubsection{Estimate for \texorpdfstring{$\etab$}{}}
\begin{prop}\label{estetab}
    We have the following estimate:
    \begin{align*}
        \|\etab^{(3)}\|_{L^2_{sc}(S_{u,\ub})}\les 1.
    \end{align*}
\end{prop}
\begin{proof}
    We have from Proposition \ref{null}
    \begin{align*}
        \nabs_3\etab+\frac{1}{2}\trchb\,\etab=\bb+\frac{1}{2}\trchb\,\eta+\frac{|u|}{a}\Gab\c\Gab.
    \end{align*}
    Differentiating it by $\dkb^i$ and applying Proposition \ref{commutation}, we infer
    \begin{align*}
        \nabs_3(\dkb^i\etab)+\frac{1+i}{2}\trchb(\dkb^i\etab)=\bb^{(i)}+\trchb(\dkb^i\eta)+\frac{|u|}{a}(\Gab\c\Gab)^{(i)}.
    \end{align*}
    Applying Lemma \ref{3evolution} and Proposition \ref{Holder}, we deduce for $i\leq 3$
    \begin{align*}
         |u|^{-1}\|\dkb^i\etab\|_{L^2_{sc}(S_{u,\ub})}&\les|u_\infty|^{-1}\|\dkb^i\etab\|_{L^2_{sc}(S_{u_\infty,\ub})}+\int_{u_\infty}^u\frac{a}{|u'|^3}\|\bb^{(3)}\|_{L^2_{sc}(S_{u',\ub})}du'\\
         &+\int_{u_\infty}^u\frac{a}{|u'|^4}\|\trchb\|_{L^\infty_{sc}(S_{u',\ub})}\|\eta^{(3)}\|_{L^2_{sc}(S_{u',\ub})}+\frac{1}{|u'|^2}\|(\Gab\c\Gab)^{(3)}\|_{L^2_{sc}(S_{u',\ub})}du'\\
         &\les\frac{1}{|u_\infty|}+\left(\int_{u_\infty}^u\frac{a}{|u'|^2}\|\bb^{(3)}\|_{L^2_{sc}(S_{u',\ub})}^2du'\right)^\frac{1}{2}\left(\int_{u_\infty}^u\frac{a}{|u'|^4}du'\right)^\frac{1}{2}\\
         &+\int_{u_\infty}^u\frac{1}{|u'|^2}\|\eta^{(3)}\|_{L^2_{sc}(S_{u',\ub})}+\frac{1}{|u'|^3}\|\Gab^{(3)}\|_{L^2_{sc}(S_{u',\ub})}\|\Gab\|_{L^\infty_{sc}(S_{u',\ub})}du'\\
         &\les\frac{1}{|u|}+\frac{\af}{|u|^\frac{3}{2}}+\frac{1}{|u|}+\frac{a}{|u|^2}\\
         &\les\frac{1}{|u|},
    \end{align*}
    where we used Propositions \ref{estrb} and \ref{esteta} and \eqref{improvedGab}. Combining with Proposition \ref{ellipticLp}, this concludes the proof of Proposition \ref{estetab}.
\end{proof}
\subsubsection{Estimate for \texorpdfstring{$(\omb,\ombd)$}{}}
\begin{prop}\label{estomb}
We have the following estimate:
    \begin{align*}
        \|(\omb,\ombd)^{(3)}\|_{L^2_{sc}(S_{u,\ub})}\les 1.
    \end{align*}
\end{prop}
\begin{proof}
    We have from Proposition \ref{null} and \eqref{dfomd}
    \begin{align}\label{nab4ombinproof}
        \nabs_4(\omb,-\ombd)=-\frac{1}{2}(\K,\si)-\frac{1}{8}(\trch\trchb,0)+\frac{|u|}{a}\Gab\c\Gab.
    \end{align}
    Note that we have
    \begin{align*}
        \K+\frac{1}{4}\trch\trchb&=\K+\frac{1}{4}\left(\trcht+\frac{2}{|u|}\right)\left(\trchbt-\frac{2}{\Om|u|}\right)\\
        &=\K-\frac{1}{\Om|u|^2}+\frac{1}{4}\trch\,\trchbt-\frac{1}{2\Om|u|}\trcht\\
        &=\Kc+\Wb+\Gag\c\Gag,
    \end{align*}
    where
    \begin{align}\label{Wbdf}
        \Wb:=\frac{\Om-1}{\Om|u|^2}-\frac{1}{2\Om|u|}\trcht.
    \end{align}
    Moreover, we have $s_2(\Wb)=1$. Thus, we have from Proposition \ref{Holder}\footnote{Note that here the signature $s_2$ of $|u|^{-2}$ and $|u|^{-1}$ are $1$, which are consistent with \eqref{Wbdf}.}
    \begin{align}
    \begin{split}\label{Wbest}
        \left\|\Wb^{(3)}\right\|_{L^2_{sc}(S_{u,\ub})}&\les\left\|\frac{(\log\Om)^{(3)}}{|u|^2}\right\|_{L^2_{sc}(S_{u,\ub})}+\left\|\frac{(\trcht)^{(3)}}{|u|}\right\|_{L^2_{sc}(S_{u,\ub})}\\
        &\les\frac{1}{|u|}\left\||u|^{-2}\right\|_{L^\infty_{sc}(S_{u,\ub})}\left\|(\log\Om)^{(3)}\right\|_{L^2_{sc}(S_{u,\ub})}\\
        &+\frac{1}{|u|}\left\||u|^{-1}\right\|_{L^\infty_{sc}(S_{u,\ub})}\left\|\trcht^{(3)}\right\|_{L^2_{sc}(S_{u,\ub})}\\
        &\les a^{-1}\left\|(\log\Om)^{(3)}\right\|_{L^2_{sc}(S_{u,\ub})}+\frac{|u|}{a}\left\|\trcht^{(3)}\right\|_{L^2_{sc}(S_{u,\ub})}\\
        &\les 1,
    \end{split}
    \end{align}
    where we used \eqref{B1} and Proposition \ref{esttrcht} at the last step. Hence, \eqref{nab4ombinproof} can be written in the following form:
    \begin{align*}
        \nabs_4(\omb,-\ombd)=-\frac{1}{2}(\Kc,\si)-\frac{1}{2}(\Wb,0)+\frac{|u|}{a}\Gab\c\Gab.
    \end{align*}
    Differentiating it by $\dkb^i$ and applying Proposition \ref{commutation}, we deduce
    \begin{align*}
        \nabs_4(\dkb^i(\omb,-\ombd))=(\Kc,\si)^{(i)}+\Wb^{(i)}+\frac{|u|}{a}(\Gab\c\Gab)^{(i)}.
    \end{align*}
    Applying Lemma \ref{evolution} and Proposition \ref{Holder}, we infer for $i\leq 3$
    \begin{align*}
        &\quad\;\,\|\dkb^i(\omb,\ombd)\|_{L^2_{sc}(S_{u,\ub})}\\
        &\les\int_0^\ub\|\Kc^{(3)}\|_{L^2_{sc}(S_{u,\ub'})}+\|\Wb^{(3)}\|_{L^2_{sc}(S_{u,\ub'})}+\frac{|u|}{a}\|(\Gab\c\Gab)^{(3)}\|_{L^2_{sc}(S_{u,\ub'})}d\ub'\\
        &\les\|\Kc^{(3)}\|_{L^2_{sc}(\cuv)}+\int_0^\ub\|\Wb^{(3)}\|_{L^2_{sc}(S_{u,\ub'})}+a^{-1}\|\Gab^{(3)}\|_{L^2_{sc}(S_{u,\ub'})}\|\Gab\|_{L^\infty_{sc}(S_{u,\ub'})}d\ub'\\
        &\les 1,
    \end{align*}
    where we used Proposition \ref{estrb}, \eqref{Wbest} and \eqref{improvedGab} at the last step.
\end{proof}
\subsubsection{Estimate for null lapse \texorpdfstring{$\log\Om$}{}}
\begin{prop}\label{estOm}
We have the following estimate:
    \begin{align*}
        \left\|\log\Om\right\|_{L^2_{sc}(S_{u,\ub})}\les 1.
    \end{align*}
\end{prop}
\begin{proof}
    We have from \eqref{nullidentities}
    \begin{align*}
        \nabs_4\log\Om=-2\om.
    \end{align*}
    Applying Lemma \ref{evolution} and Proposition \ref{estom}, we infer
    \begin{align*}
        \|\log\Om\|_{L^2_{sc}(S_{u,\ub})}\les \int_{0}^\ub \|\om\|_{L^2_{sc}(S_{u,\ub'})}d\ub'\les1.
    \end{align*}
    This concludes the proof of Proposition \ref{estOm}.
\end{proof}
\subsubsection{Estimate for null shift \texorpdfstring{$\bbb$}{}}
\begin{prop}\label{estbbb}
We have the following estimate:
    \begin{align*}
        \left\|\bbb\right\|_{L^2_{sc}(S_{u,\ub})}\les 1.
    \end{align*}
\end{prop}
\begin{proof}
    Notice that we have
    \begin{align*}
        \nabs_4\bbb^A=-4\Om\ze^A.
    \end{align*}
    Applying Lemma \ref{evolution}, Propositions \ref{esteta} and \ref{estetab}  and \eqref{nullidentities}, we infer
    \begin{align*}
        \|\bbb\|_{L^2_{sc}(S_{u,\ub})}\les \int_{0}^\ub \|\ze\|_{L^2_{sc}(S_{u,\ub'})}d\ub'\les 1.
    \end{align*}
    This concludes the proof of Proposition \ref{estbbb}.
\end{proof}
\subsubsection{Estimate for areal radius \texorpdfstring{$r$}{}}
\begin{prop}\label{estr}
    We have the following estimate:
    \begin{align*}
        |r+u|\les 1.
    \end{align*}
\end{prop}
\begin{proof}
    We have from Lemma \ref{dint}
    \begin{align*}
        |e_4(r+u)|=\left|\frac{\ov{\Om\trch}}{2\Om}r\right|\les 1.
    \end{align*}
    Recalling that $r=|u|$ holds on $\Hb_0$, we obtain
    \begin{align*}
        |r+u|\les \int_0^\ub |e_4(r+u)| d\ub'\les 1.
    \end{align*}
    This concludes the proof of Proposition \ref{estr}.
\end{proof}
\subsubsection{top-order estimate for \texorpdfstring{$\trcht$}{}}
\begin{prop}\label{fluxtrcht}
    We have the following estimate:
    \begin{align*}
        \left\|\trcht^{(4)}\right\|_{L^2_{sc}(S_{u,\ub})}\les\af.
    \end{align*}
\end{prop}
\begin{proof}
Denoting
\begin{align*}
    \Ga_l:=\{\trcht,\hch,\om,\Psi_4,\bF\},\qquad\quad \Ga_l^{(1)}:=(\af\nabs)^{\leq 1}\Ga_l\cup\{\b\},
\end{align*}
we have from Section \ref{secnorms} and \eqref{B1}
\begin{align*}
    \|\Ga_l^{(4)}\|_{L^2_{sc}(\cuv)}\les\af\bo,\qquad\quad \|\Ga_l^{(1)}\|_{L^\infty_{sc}(S_{u,\ub})}\les\af\bo.
\end{align*}
We have from Proposition \ref{nulles}
\begin{align*}
    \nabs_4\trcht=\Ga_l\c\Ga_l.
\end{align*}
Differentiating it by $\dkb^i$ and applying Proposition \ref{commutation}, we obtain
\begin{align*}
    \nabs_4(\dkb^i\trcht)=(\Gag\c\Gab)^{(i-1)}+(\Ga_l\c\Ga_l)^{(i)}.
\end{align*}
Applying Lemma \ref{evolution} and Proposition \ref{Holder}, we infer for $i\leq 4$
\begin{align*}
    \|\dkb^i\trcht\|_{L^2_{sc}(S_{u,\ub})}&\les\int_0^\ub\|(\Gag\c\Gab)^{(3)}\|_{L^2_{sc}(S_{u,\ub'})}+\|(\Ga_l\c\Ga_l)^{(4)}\|_{L^2_{sc}(S_{u,\ub'})}d\ub'\\
    &\les\int_0^\ub \frac{\af b^\frac{1}{2}}{|u|}+\frac{\af\bo}{|u|}\|\Ga_l^{(4)}\|_{L^2_{sc}(S_{u,\ub'})}d\ub'\\
    &\les\frac{\af b^\frac{1}{2}}{|u|}+\frac{\af\bo}{|u|}\left(\int_0^\ub\|\Ga_l^{(4)}\|^2_{L^2_{sc}(S_{u,\ub'})}d\ub'\right)^\frac{1}{2}\\
    &\les\frac{ab^\frac{1}{2}}{|u|}\\
    &\les\af.
\end{align*}
Combining with Proposition \ref{ellipticLp}, this concludes the proof of Proposition \ref{fluxtrcht}.
\end{proof}
\subsubsection{top-order estimate for \texorpdfstring{$\trchbc$}{}}
\begin{prop}\label{fluxtrchbc}
    We have the following estimate:
    \begin{align*}
        |u|^{-\frac{1}{2}}\left\|\trchbc^{(4)}\right\|_{L^2_{sc}(S_{u,\ub})}\les 1.
    \end{align*}
\end{prop}
\begin{proof}
Denoting
\begin{align*}
    \Ga_r:=\left\{\frac{a}{|u|}\trchbt,\frac{a}{|u|}\hchb,\omb,\Psi_3,\bbF\right\},\qquad\quad \Ga_r^{(1)}:=(\af\nabs)^{\leq 1}\Ga_r\cup\{\bb\},
\end{align*}
we have from Section \ref{secnorms} and \eqref{B1}
\begin{align*}
    \|\Ga_r^{(4)}\|_{L^2_{sc}(\ucuv)}\les\af\bo,\qquad\quad \|\Ga_r^{(1)}\|_{L^\infty_{sc}(S_{u,\ub})}\les\af\bo.
\end{align*}
Next, we have from Proposition \ref{nulles}
\begin{align*}
    \nabs_3\trchb+\frac{1}{2}(\trchb)^2=\frac{|u|^2}{a^2}\Ga_r\c\Ga_r.
\end{align*}
Differentiating it by $\dkb^i$ and applying Proposition \ref{commutation}, we deduce
\begin{align*}
    \nabs_3(\dkb^i\trchb)+\trchb(\dkb^i\trchb)=\frac{|u|^2}{a^2}(\Gag\c\Gab)^{(3)}+\frac{|u|^2}{a^2}(\Ga_r\c\Ga_r)^{(4)}.
\end{align*}
Applying Lemma \ref{3evolution} and Propositions \ref{Holder}, we infer for $1\leq i\leq 4$
\begin{align*}
    |u|^{-1}\|\dkb^i\trchb\|_{L^2_{sc}(S_{u,\ub})}&\les|u_\infty|^{-1}\|\dkb^i\trchb\|_{L^2_{sc}(S_{u_\infty,\ub})}+\int_{u_\infty}^u\frac{1}{a|u'|}\|(\Gag\c\Gab)^{(3)}\|_{L^2_{sc}(S_{u',\ub})}du'\\
    &+\int_{u_\infty}^u\frac{1}{a|u'|}\|(\Ga_r\c\Ga_r)^{(4)}\|_{L^2_{sc}(S_{u',\ub})}du'\\
    &\les\frac{\af}{|u|}+\int_{u_\infty}^u\frac{b^\frac{1}{2}}{\af|u'|^2}du'+\int_{u_\infty}^u\frac{\bo}{\af|u'|^2}\|\Ga_r^{(4)}\|_{L^2_{sc}(S_{u',\ub})}du'\\
    &\les\frac{\af}{|u|}+\bo\left(\int_{u_\infty}^u\frac{a}{|u'|^2}\|\Ga_r^{(4)}\|^2_{L^2_{sc}(S_{u',\ub})}du'\right)^\frac{1}{2}\left(\int_{u_\infty}^u\frac{du'}{a^2|u'|^2}\right)^\frac{1}{2}\\
    &\les\frac{\af}{|u|}+\frac{\bo}{a|u|^\frac{1}{2}}\|\Ga_r^{(4)}\|_{L^2(\ucuv)}\\
    &\les|u|^{-\frac{1}{2}}.
\end{align*}
Combining with Propositions \ref{ellipticLp}, this concludes the proof of Proposition \ref{fluxtrchbc}.
\end{proof}
Combining Propositions \ref{esthchhchb}, \ref{estom}--\ref{fluxtrchbc}, we obtain
\begin{equation}\label{estOG}
    \sup_{V_*}\sum_{i=0}^4\mo_i[G](u,\ub)\les 1.
\end{equation}
\subsection{Estimates for \texorpdfstring{$\mg_i(u,\ub)$}{} and \texorpdfstring{$\ug_i(u,\ub)$}{}}
In this section, we estimate the $L^2$--flux norms of all the Ricci coefficients.
\subsubsection{Estimate for \texorpdfstring{$\hch$}{}}
\begin{prop}\label{fluxhch}
    We have the following estimate:
    \begin{align*}
        \|\hch^{(4)}\|_{L^2_{sc}(\cuv)}\les\af.
    \end{align*}
\end{prop}
\begin{proof}
    We have from Proposition \ref{null}
    \begin{align*}
        \sdivs\hch=\frac{1}{2}\nabs\trch-\b+\Gab\c\Gab,
    \end{align*}
    which implies for $i\leq 3$
    \begin{align*}
        \dkb^{i+1}\hch=\dkb^{i+1}\trch-\af\dkb^{i}\b+\af(\Gab\c\Gab)^i.
    \end{align*}
    Applying Propositions \ref{ellipticLp} and \ref{Holder}, we infer
    \begin{align*}
        \|\hch^{(4)}\|_{L^2_{sc}(S_{u,\ub})}&\les\left\|\trcht^{(4)}\right\|_{L^2_{sc}(S_{u,\ub})}+\af\|\b^{(3)}\|_{L^2(S_{u,\ub})}+\af\|(\Gab\c\Gab)^{(3)}\|_{L^2_{sc}(S_{u,\ub})}\\
        &\les\af+\frac{a^\frac{3}{2}}{|u|}+\af\|\b^{(3)}\|_{L^2(S_{u,\ub})}\\
        &\les\af+\af\|\b^{(3)}\|_{L^2_{sc}(S_{u,\ub})}.
    \end{align*}
    where we used Proposition \ref{fluxtrcht} and \eqref{improvedGab} at the second step. Hence, we obtain from Proposition \ref{estbr}
    \begin{align*}
        \|\hch^{(4)}\|_{L^2_{sc}(\cuv)}\les \af+\af\|\b^{(3)}\|_{L^2_{sc}(\cuv)}\les\af.
    \end{align*}
    This concludes the proof of Proposition \ref{fluxhch}.
\end{proof}
\subsubsection{Estimate for \texorpdfstring{$\hchb$}{}}
\begin{prop}\label{fluxhchb}
    We have the following estimate:
    \begin{align*}
        \left\|\frac{a}{|u|}\hchb^{(4)}\right\|_{L^2_{sc}(\ucuv)}\les\af.
    \end{align*}
\end{prop}
\begin{proof}
    We have from Proposition \ref{null}
    \begin{align*}
        \sdivs\hchb=\frac{1}{2}\nabs\trchb+\bb-\frac{1}{2}\trchb\,\ze+\frac{|u|}{a}\Gab\c\Gab,
    \end{align*}
    which implies for $i\leq 3$
    \begin{align*}
        \dkb^{i+1}\hchb=\frac{1}{2}\dkb^{i+1}(\trchbc)+\af(\dkb^i\bb)-\frac{1}{2}\af\trchb(\dkb^i\ze)+\frac{|u|}{\af}(\Gab\c\Gab)^{(i)}.
    \end{align*}
    Applying Propositions \ref{ellipticLp} and \ref{Holder}, we infer
    \begin{align*}
        \|\hchb^{(4)}\|_{L^2_{sc}(\cuv)}&\les\left\|\trchbc^{(4)}\right\|_{L^2_{sc}(S_{u,\ub})}+\af\|\bb^{(3)}\|_{L^2_{sc}(S_{u,\ub})}+\|\trchb\c\ze^{(3)}\|_{L^2_{sc}(S_{u,\ub})}\\
        &+\frac{|u|}{\af}\|(\Gab\c\Gab)^{(3)}\|_{L^2_{sc}(S_{u,\ub})}\\
        &\les\left\|\trchbc^{(4)}\right\|_{L^2_{sc}(S_{u,\ub})}+\frac{1}{|u|}\|\trchb\|_{L^\infty_{sc}(S_{u,\ub})}\|\ze^{(3)}\|_{L^2_{sc}(S_{u,\ub})}\\
        &+\afd\|\Gab\|_{L^2_{sc}(S_{u,\ub})}\|\Gab^{(3)}\|_{L^2_{sc}(S_{u,\ub})}+\af\|\bb^{(3)}\|_{L^2_{sc}(S_{u,\ub})}\\
        &\les|u|^\frac{1}{2}+\af\|\bb^{(3)}\|_{L^2_{sc}(S_{u,\ub})},
    \end{align*}
    where we used Propositions \ref{esteta}, \ref{estetab} and \ref{fluxtrchbc} and \eqref{improvedGab} at the last step. Hence, we deduce
    \begin{align*}
        \frac{a}{|u|}\|\hchb^{(4)}\|_{L^2_{sc}(\cuv)}\les \frac{a}{|u|^\frac{1}{2}}+\frac{a^\frac{3}{2}}{|u|}\|\bb^{(3)}\|_{L^2_{sc}(S_{u,\ub})}.
    \end{align*}
    Thus, we infer
    \begin{align*}
        \left\|\frac{a}{|u|}\hchb^{(4)}\right\|_{L^2_{sc}(\ucuv)}^2&=\int_{u_\infty}^u\frac{a}{|u'|^2}\left\|\frac{a}{|u'|}\hchb^{(4)}\right\|_{L^2_{sc}(S_{u',\ub})}^2du'\\
        &\les \int_{u_\infty}^u\frac{a^3}{|u'|^3}+\frac{a^4}{|u'|^4}\|\bb^{(3)}\|_{L^2_{sc}(S_{u',\ub})}^2 du'\\
        &\les \frac{a^3}{|u|^2}+\frac{a^3}{|u|^2}\int_{u_\infty}^u \frac{a}{|u'|^2}\|\bb^{(3)}\|_{L^2_{sc}(S_{u',\ub})}^2du'\\
        &\les a,
    \end{align*}
    where we used \eqref{phifluxsc} and Proposition \ref{estrb} at the last step. This concludes the proof of Proposition \ref{fluxhchb}.
\end{proof}
\subsubsection{Estimates for \texorpdfstring{$\mu$}{} and \texorpdfstring{$\eta$}{}}
\begin{prop}\label{fluxmu}
    We have the following estimate:
    \begin{align*}
        \|\muc^{(3)}\|_{L^2_{sc}(\ucuv)}\les\frac{a}{|u|},\qquad\quad \|\eta^{(4)}\|_{L^2_{sc}(\ucuv)}\les\af.
    \end{align*}
\end{prop}
\begin{rk}
The fact that $\muc$ has better decay than $\mu$ plays an essential role in the proof of Proposition \ref{estrb}.
\end{rk}
\begin{proof}[Proof of Proposition \ref{fluxmu}]
We have from Proposition \ref{null}
\begin{align*}
    \nabs_4\mu=-\frac{1}{|u|^2}\trch+\afd(\Gab\c\Gab)^{(1)}.
\end{align*}
Differentiating it by $\dkb^i$ and applying Proposition \ref{commutation}, we obtain for $1\leq i\leq 3$
\begin{align*}
    \nabs_4(\dkb^i\mu)=-\frac{1}{|u|^2}\dkb^i\trcht+\afd(\Gab\c\Gab)^{(i+1)}.
\end{align*}
Applying Lemma \ref{evolution}, we infer for $1\leq i\leq 3$
\begin{align*}
    \|\dkb^i\mu\|_{L^2_{sc}(S_{u,\ub})}&\les \int_0^\ub a^{-1}\|\trcht^{(3)}\|_{L^2_{sc}(S_{u,\ub})}+\afd\|(\Gab\c\Gab)^{(4)}\|_{L^2_{sc}(S_{u,\ub'})}d\ub'\\
    &\les\frac{1}{|u|}+\int_0^\ub\afd\|(\Gab\c\Gab)^{(4)}\|_{L^2_{sc}(S_{u,\ub'})}d\ub',
\end{align*}
where we used Proposition \ref{esttrcht} at the last step. Multiplying it by $\frac{a}{|u|^2}$ and integrating it from $u_\infty$ to $u$, we deduce for $1\leq i\leq 3$
\begin{align*}
    \|\dkb^i\mu\|^2_{L^2_{sc}(\ucuv)}&\les\int_0^\ub\int_{u_\infty}^u\frac{a}{|u'|^4}+\frac{1}{|u'|^2}\|(\Gab\c\Gab)^{(4)}\|_{L^2_{sc}(S_{u',\ub'})}^2du'd\ub'\\
    &\les\frac{a}{|u|^3}+\frac{1}{|u|^2}\int_0^\ub\int_{u_\infty}^u\|(\Gab\c\Gab)^{(4)}\|_{L^2_{sc}(S_{u',\ub'})}^2du'd\ub'\\
    &\les\frac{a}{|u|^3}+\frac{ab}{|u|^2}\\
    &\les\frac{a^2}{|u|^2}.
\end{align*}
where we used Theorem \ref{Junkmanthm} in the third step. Applying Proposition \ref{ellipticLp}, we infer
\begin{equation}\label{mucest}
    \|\muc^{(3)}\|_{L^2_{sc}(\ucuv)}\les \frac{a}{|u|}.
\end{equation}
We recall from \eqref{murenor}
\begin{align*}
    \sdivs\eta=\Kc-\mu.
\end{align*}
Applying Propositions \ref{ellipticLp} and \ref{estbr} and \eqref{mucest}, we have
\begin{align*}
    \|\eta^{(4)}\|_{L^2_{sc}(\ucuv)}\les\af\|\Kc^{(3)}\|_{L^2_{sc}(\ucuv)}+ \af\|\mu^{(3)}\|_{L^2_{sc}(\ucuv)}\les\af+\frac{a^\frac{3}{2}}{|u|}\les\af.
\end{align*}
This concludes the proof of Proposition \ref{fluxmu}.
\end{proof}
\subsubsection{Estimates for \texorpdfstring{$\vkpb$}{} and \texorpdfstring{$(\omb,\ombd)$}{}}
\begin{prop}\label{fluxvkpb}
    We have the following estimate:
    \begin{align*}
        \|\vkpb^{(3)}\|_{L^2_{sc}(\ucuv)}\les 1,\qquad\quad \|(\omb,\ombd)^{(4)}\|_{L^2_{sc}(\ucuv)}\les\af.
    \end{align*}
\end{prop}
\begin{proof}
    We have from Proposition \ref{null}
    \begin{align*}
        \nabs_4\vkpb=\frac{|u|}{ab}(\Gab\c\Gab)^{(1)}.
    \end{align*}
    Differentiating it by $\dkb^i$ and applying Proposition \ref{commutation}, we obtain
    \begin{align*}
        \nabs_4(\dkb^i\vkpb)=\frac{|u|}{ab}(\Gab\c\Gab)^{(i+1)}.
    \end{align*}
    Applying Lemma \ref{evolution}, we infer for $i\leq 3$
    \begin{align*}
        \|\dkb^i\vkpb\|_{L^2_{sc}(S_{u,\ub})}\les\int_0^\ub\frac{|u|}{ab}\|(\Gab\c\Gab)^{(4)}\|_{L^2_{sc}(S_{u,\ub'})}d\ub'.
    \end{align*}
    Multiplying it by $\frac{a}{|u|^2}$ and integrating it from $u_\infty$ to $u$, we deduce for $i\leq 3$
    \begin{align*}
        \|\dkb^i\vkpb\|_{L^2_{sc}(\ucuv)}^2&\les\int_0^\ub\int_{u_\infty}^u\frac{a}{|u'|^2}\frac{|u'|^2}{a^2b^2}\|(\Gab\c\Gab)^{(4)}\|_{L^2_{sc}(S_{u',\ub'})}^2du'd\ub'\\
        &\les\frac{b}{a}+\frac{ab}{ab^2}\les 1.
    \end{align*}
    where we used Theorem \ref{Junkmanthm} in the second step. Applying Proposition \ref{ellipticLp}, we infer
    \begin{equation}\label{vkpbest}
        \|\vkpb^{(3)}\|_{L^2_{sc}(\ucuv)}\les 1.
    \end{equation}
    Next, we recall from \eqref{renorvkp}
    \begin{align*}
        \sld_1^*(\omb,\ombd):=-\vkpb+\frac{1}{2}\bb.
    \end{align*}
    Applying Propositions \ref{ellipticLp}, we obtain
    \begin{align*}
        \|(\omb,\ombd)^{(4)}\|_{L^2_{sc}(\ucuv)}&\les\af\|\vkpb^{(3)}\|_{L^2_{sc}(\ucuv)}+\af\|\bb^{(3)}\|_{L^2_{sc}(\ucuv)}+\|(\omb,\ombd)\|_{L^2_{sc}(\ucuv)}\\
        &\les\af,
    \end{align*}
    where we used \eqref{vkpbest} and Propositions \ref{estrb} and \ref{estomb} at the last step. This concludes the proof of Proposition \ref{fluxvkpb}.
\end{proof}
\subsubsection{Estimates for \texorpdfstring{$\mub$}{} and \texorpdfstring{$\etab$}{}}
\begin{prop}\label{fluxmub}
We have the following estimate:
\begin{align*}
    \|\mub^{(3)}\|_{L^2_{sc}(\cuv)}\les 1,\qquad\quad \|\etab^{(4)}\|_{L^2_{sc}(\cuv)}\les\af.
\end{align*}
\end{prop}
\begin{proof}
    We have from Proposition \ref{null}
    \begin{align*}
        \nabs_3\mub+\trchb\,\mub=-\frac{1}{|u|^2}\trchbt-\trchb\,\sdivs\eta+\frac{|u|}{ab}(\Gab\c\Gab)^{(1)}.
    \end{align*}
    Differentiating it by $\dkb^i$ and applying Proposition \ref{commutation}, we obtain
    \begin{align*}
        \nabs_3(\dkb^i\mub)+\trchb(\dkb^i\mub)=\frac{1}{|u|^2}\trcht^{(i)}+\trchb(\widecheck{\K}-\muc)^{(i)}+\frac{|u|}{ab}(\Gab\c\Gab)^{(i+1)}.
    \end{align*}
    Applying Lemma \ref{3evolution} and Proposition \ref{Holder}, we deduce for $i\leq 3$
    \begin{align*}
        |u|^{-1}\|\dkb^i\mub\|_{L^2_{sc}(S_{u,\ub})}&\les|u_\infty|^{-1}\|\dkb^i\mub\|_{L^2_{sc}(S_{u_\infty,\ub})}+\int_{u_\infty}^u\frac{1}{|u'|^3}\|\trchbt^{(3)}\|_{L^2_{sc}(S_{u',\ub})}+\frac{1}{|u'|^2}\|\muc^{(3)}\|_{L^2_{sc}(S_{u',\ub})}du'\\
        &+\int_{u_\infty}^u \frac{1}{|u'|^2}\|\Kc^{(3)}\|_{L^2_{sc}(S_{u',\ub})}+\frac{1}{b|u'|^2}\|(\Gab\c\Gab)^{(4)}\|_{L^2_{sc}(S_{u',\ub})}du',
    \end{align*}
    which implies from Propositions \ref{esttrchbc}, \ref{fluxmu} and \ref{estrb} that for $i\leq 3$
    \begin{align*}
        |u|^{-2}\|\dkb^i\mub\|_{L^2_{sc}(\cuv)}^2&\les|u_\infty|^{-2}\|\dkb^i\mub\|_{L^2_{sc}(\cuvi)}^2+|u|^{-4}\\
        &+\int_0^\ub\left(\int_{u_\infty}^u\frac{a}{|u'|^2}\|\muc^{(3)}\|^2_{L^2_{sc}(S_{u',\ub'})} du'\right)\left(\int_{u_\infty}^u\frac{1}{a|u'|^2}du'\right)d\ub'\\
        &+\int_0^\ub\left(\int_{u_\infty}^u\frac{1}{|u'|^2}\|\Kc^{(3)}\|^2_{L^2_{sc}(S_{u',\ub'})} du'\right)\left(\int_{u_\infty}^u\frac{1}{|u'|^2}du'\right)d\ub'\\
        &+\int_0^\ub\left(\int_{u_\infty}^u\frac{1}{b|u'|^2}\|(\Gab\c\Gab)^{(4)}\|_{L^2_{sc}(S_{u',\ub'})}du'\right)^2d\ub'\\
        &\les\frac{1}{|u|^2}+\frac{1}{a|u|}\int_0^\ub\|\muc^{(3)}\|_{L^2_{sc}(\Hb_{\ub'}^{(u_\infty,u)})}^2d\ub'+\frac{1}{|u|}\int_{u_\infty}^u\frac{du'}{|u'|^2}\|\Kc^{(3)}\|^2_{L^2_{sc}(H_{u'}^{(0,\ub)})}\\
        &+\frac{1}{b^2|u|^3}\int_0^\ub\int_{u_\infty}^u\|(\Gab\c\Gab)^{(4)}\|_{L^2_{sc}(S_{u',\ub'})}^2du'd\ub'\\
        &\les\frac{1}{|u|^2}+\frac{b}{|u|^3}+\frac{ab}{b^2|u|^3}\les\frac{1}{|u|^2},
    \end{align*}
    where we used Theorem \ref{Junkmanthm} in the third step. Combining with Propositions \ref{ellipticLp} and \ref{estR}, we infer
    \begin{equation}\label{estmub}
        \|\mub^{(3)}\|_{L^2_{sc}(\cuv)}\les 1.
    \end{equation}
    We recall from \eqref{murenor}
    \begin{align*}
        \sdivs\etab=\Kc-\mub.
    \end{align*}
    Applying Propositions \ref{ellipticLp} and \ref{estR}, we obtain
    \begin{align*}
        \|\etab^{(4)}\|_{L^2_{sc}(S_{u,\ub})}\les\af\|\Kc^{(3)}\|_{L^2_{sc}(\ucuv)}+ \af\|\mub^{(3)}\|_{L^2_{sc}(\ucuv)}\les\af.
    \end{align*}
    This concludes the proof of Proposition \ref{fluxmub}.
\end{proof}
\subsubsection{Estimates for \texorpdfstring{$\vkp$}{} and \texorpdfstring{$(\om,\omd)$}{}}
\begin{prop}\label{fluxvkp}
    We have the following estimate:
    \begin{align*}
        \|\vkp^{(3)}\|_{L^2_{sc}(\cuv)}\les 1,\qquad\quad \|(\om,\omd)^{(4)}\|_{L^2_{sc}(\cuv)}\les\af.
    \end{align*}
\end{prop}
\begin{proof}
    We have from Proposition \ref{null}
    \begin{align*}
        \nabs_3\vkp+\frac{1}{2}\trchb\vkp=\frac{1}{4}\trchb\,\b+\frac{|u|}{ab}(\Gab\c\Gab)^{(1)}.
    \end{align*}
    Differentiating it by $\dkb^i$ and applying Proposition \ref{commutation}, we obtain
    \begin{align*}
        \nabs_3(\dkb^i\vkp)+\frac{1}{2}\trchb(\dkb^i\vkp)=\trchb\,\b^{(i)}+\frac{|u|}{ab}(\Gab\c\Gab)^{(i+1)}.
    \end{align*}
    Applying Lemma \ref{3evolution} and Proposition \ref{Holder}, we deduce for $i\leq 3$
    \begin{align*}
        |u|^{-1}\|\dkb^i\vkp\|_{L^2_{sc}(S_{u,\ub})}&\les|u_\infty|^{-1}\|\dkb^i\vkp\|_{L^2_{sc}(S_{u,\ub})}+\int_{u_\infty}^u\frac{a}{|u'|^4}\|\trchb\|_{L^\infty_{sc}(S_{u',\ub})}\|\b^{(3)}\|_{L^2_{sc}(S_{u',\ub})}du'\\
        &+\int_{u_\infty}^u\frac{1}{b|u'|^2}\|(\Gab\c\Gab)^{(4)}\|_{L^2_{sc}(S_{u',\ub})}du',
    \end{align*}
    which implies from Proposition \ref{estbr} and Theorem \ref{Junkmanthm} that for $i\leq 3$
    \begin{align*}
        |u|^{-2}\|\dkb^i\vkp\|^2_{L^2_{sc}(\cuv)}&\les|u_\infty|^{-2}\|\dkb^i\vkp\|^2_{L^2_{sc}(\cuvi)}+\int_0^\ub\left(\int_{u_\infty}^u\frac{1}{|u'|^2}\|\b^{(3)}\|^2_{L^2_{sc}(S_{u',\ub'})}du'\right)^2d\ub'\\
        &+\int_0^\ub\left(\int_{u_\infty}^u\frac{1}{b|u'|^2}\|(\Gab\c\Gab)^{(4)}\|_{L^2_{sc}(S_{u',\ub'})}du'\right)^2d\ub'\\
        &\les\frac{1}{|u|^2}+\frac{1}{|u|}\int_{u_\infty}^u\int_0^\ub\frac{1}{|u'|^2}\|\b^{(3)}\|^2_{L^2_{sc}(S_{u',\ub'})}d\ub'du'\\
        &+\frac{1}{b^2|u|^3}\int_0^\ub\int_{u_\infty}^u\|(\Gab\c\Gab)^{(4)}\|_{L^2_{sc}(S_{u',\ub'})}^2du'd\ub'\\
        &\les\frac{1}{|u|^2}+\frac{b}{|u|^3}+\frac{ab}{b^2|u|^3}\les\frac{1}{|u|^2}.
    \end{align*}
    Combining with Proposition \ref{ellipticLp}, we obtain
    \begin{align}\label{vkpest}
        \|\vkp^{(3)}\|_{L^2_{sc}(\cuv)}\les 1.
    \end{align}
    Next, we recall from \eqref{renorvkp}
    \begin{align*}
        \sld_1^*(-\om,\omd):=\vkp+\frac{1}{2}\b.
    \end{align*}
    Applying Propositions \ref{ellipticLp}, we obtain
    \begin{align*}
        \|(\om,\omd)^{(4)}\|_{L^2_{sc}(\cuv)}&\les\af\|\vkp^{(3)}\|_{L^2_{sc}(\cuv)}+\af\|\b^{(3)}\|_{L^2_{sc}(\cuv)}+\|(\om,\omd)\|_{L^2_{sc}(\cuv)}\\
        &\les\af,
    \end{align*}
    where we used \eqref{vkpest} and Propositions \ref{estbr} and \ref{estom} at the last step. This concludes the proof of Proposition \ref{fluxvkpb}.
\end{proof}
Combining \eqref{estOG} and Propositions \ref{fluxhch}--\ref{fluxvkpb}, this concludes the proof of Theorem \ref{M3}.
\section{Conclusion}\label{sectrapping}
The goal of this section is to prove the following theorem.
\begin{thm}\label{thmtrappedsurface}
    Under the hypothesis of Theorem \ref{maintheorem}, we assume in addition that
    \begin{align}
        \int_0^1|u_\infty|^2\left(|\hch|^2+|\bF|^2+|\Psi_4|^2\right)(u_\infty,\ub',x^1,x^2)d\ub'&\geq a,\qquad \forall \;(x^1,x^2),\label{geqathm}\\
        \left|\int_{H_{u_\infty}^{(0,1)}}\Om\Im(\psi\Psi_4^\dag)\right|&\geq a.\label{J4geqa}
    \end{align}
    Then, we have:
    \begin{itemize}
        \item the $2$--sphere $S_{-\frac{a}{4},1}$ is a trapped surface,
        \item the Hawking mass of the sphere $S_{u_\infty,1}$ satisfies
        \begin{equation}\label{mass}
            m(S_{u_\infty,1})\simeq a,
        \end{equation}
        \item the electric charge of the sphere $S_{u_\infty,1}$ satisfies
    \begin{equation}\label{charge}
        |Q(S_{u_\infty,1})|\simeq\ef a.
    \end{equation}
    \end{itemize}
\end{thm}
\begin{rk}
    Theorem \ref{thmtrappedsurface} shows that a trapped surface forms in evolution, although the initial data is free of the trapped surface. Moreover, \eqref{mass} and \eqref{charge} imply that the charge-to-mass ratio of the sphere $S_{u_\infty,1}$ has size $\ef$, which is bounded above by $e_0$.
\end{rk}
The proof of Theorem \ref{thmtrappedsurface} is divided into three parts, which are proved respectively in Propositions \ref{trappedsurface}--\ref{charging}.
\subsection{Formation of trapped surface}
\begin{prop}\label{trappedsurface}
Under the hypothesis of Theorem \ref{thmtrappedsurface}, $S_{-\frac{a}{4},1}$ is a trapped surface.
\end{prop}
\begin{proof}
    We have from \eqref{basicnull}
    \begin{align}\label{nab3hch}
        \nabs_3\hch+\frac{1}{2}\trchb\,\hch=\nabs\hot\eta+2\omb\hch-\frac{1}{2}\trch\,\hchb+\eta\hot\eta+\Sslh.
    \end{align}
    Contracting it with $\hch$, we infer
    \begin{align*}
    e_3(|\hch|^2)+\trchb|\hch|^2=2\hch\c\nabs\hot\eta+4\omb|\hch|^2-\trch(\hch\c\hchb)+2\hch\c\eta\hot\eta+2\hch\c\Sslh.
    \end{align*}
    Thus, we obtain
    \begin{align*}
e_3(|u|^2|\hch|^2)=|u|^2\left(-\trchbt|\hch|^2+2\hch\c\nabs\hot\eta+4\omb|\hch|^2-\trch(\hch\c\hchb)+2\hch\c\eta\hot\eta+2\hch\c\Sslh\right).
    \end{align*}
    Applying Theorem \ref{maintheorem}, we have
    \begin{align*}
        \left|e_3(|u|^2|\hch|^2)\right|\les\frac{a}{|u|^2}+\frac{a^\frac{5}{2}}{|u|^3}.
    \end{align*}
    Recalling that
    \begin{align}\label{Ome3bbb}
        \Om e_3=\pr_u+\bbb^A\pr_A,
    \end{align}
    we deduce from Theorem \ref{maintheorem}
    \begin{align}\label{pruhch}
\left|\pr_u(|u|^2|\hch|^2)\right|\les\left|\bbb^A\pr_A(|u|^2|\hch|^2)\right|+\left|e_3(|u|^2|\hch|^2)\right|\les\frac{a^\frac{3}{2}}{|u|^2}+\frac{a}{|u|^2}+\frac{a^\frac{5}{2}}{|u|^3}\les\frac{a^\frac{3}{2}}{|u|^2}.
    \end{align}
    We also have from Proposition \ref{Maxwellequations}
    \begin{align*}
        \nabs_3\bF+\frac{1}{2}\trchb\bF=\nabs\rhoF-\dual\nabs\siF+2\omb\bF+2\rhoF\eta-2\siF\dual\eta+\hch\c\bbF -2\ef\Im\left(\psi\Psisl^\dag\right).
    \end{align*}
    Contracting it with $\bF$ and multiplying it by $|u|^2$, we infer
    \begin{align*}
        e_3\left(|u|^2|\bF|^2\right)&=-|u|^2\trchbt|\bF|^2+2|u|^2\bF\c\left(2\omb\bF+\nabs\rhoF-\dual\nabs\siF\right)\\
        &+2|u|^2\bF\c\left(2\rhoF\eta-2\siF\dual\eta+\hch\c\bbF -2\ef\Im\left(\psi\Psisl^\dag\right)\right).
    \end{align*}
    Applying Theorem \ref{maintheorem} and \eqref{Ome3bbb}, we obtain
\begin{align}\label{prubF}
\left|\pr_u\left(|u|^2|\bF|^2\right)\right|\les\left|e_3\left(|u|^2|\bF|^2\right)\right|+\left|\bbb^A\pr_A\left(|u|^2|\bF|^2\right)\right|\les\frac{a^\frac{3}{2}}{|u|^2}+\frac{\ef a^2}{|u|^2}.
\end{align}
Similarly, we have from Proposition \ref{waveequation}
\begin{align*}
    \nabs_3(\Psi_4)+\frac{1}{2}\trchb\Psi_4-2\omb\Psi_4+i\ef \Ub\Psi_4=\sdivs\Psisl+2\eta\c\Psisl+i\ef\Asl\c\Psisl-\frac{1}{2}\trch\Psi_3+i\ef\rhoF\psi,
\end{align*}
which implies
\begin{align*}
    e_3(|u|^2|\Psi_4|^2)&=(4\omb-\trchbt)|u|^2|\Psi_4|^2+2|u|^2\Psi_4^\dag\left(\sdivs\Psisl+2\eta\c\Psisl+i\ef\Asl\c\Psisl-\frac{1}{2}\trch\Psi_3+i\ef\rhoF\psi\right).
\end{align*}
Applying Theorem \ref{maintheorem} and \eqref{Ome3bbb}, we obtain
\begin{align}\label{pruPsi4}
\left|\pr_u\left(|u|^2|\Psi_4|^2\right)\right|\les\left|e_3\left(|u|^2|\Psi_4|^2\right)\right|+\left|\bbb^A\pr_A\left(|u|^2|\Psi_4|^2\right)\right|\les\frac{a^\frac{3}{2}}{|u|^2}+\frac{\ef a^2}{|u|^2}.
\end{align}
Combining \eqref{pruhch}, \eqref{prubF} and \eqref{pruPsi4}, we obtain 
\begin{align*}
    \left|\pr_u\left(|\hch|^2+|\bF|^2+|\Psi_4|^2\right)\right|\les\frac{a^\frac{3}{2}}{|u|^2}+\frac{\ef a^2}{|u|^2},
\end{align*}
which implies for $a\gg 1$ and $\ef\ll 1$
\begin{align*}
    \left|\pr_u\left(|\hch|^2+|\bF|^2+|\Psi_4|^2\right)\right|\leq \frac{a^2}{16|u|^2}.
\end{align*}
Integrating it from $u_\infty$ to $u$, we deduce
\begin{align*}
    |u|^2\left(|\hch|^2+|\bF|^2+|\Psi_4|^2\right)\geq|u_\infty|^2\left(|\hch|^2+|\bF|^2+|\Psi_4|^2\right)\bigg|_{H_{\infty}}-\frac{a^2}{16|u|}.
\end{align*}
Combining with \eqref{geqathm}, we obtain for $|u|\geq \frac{a}{4}$
\begin{align*}
    &\quad\,\int_0^1 |u|^2\left(|\hch|^2+|\bF|^2+|\Psi_4|^2\right)(u,\ub')d\ub'\\
    &\geq \int_0^1|u_\infty|^2\left(|\hch|^2+|\bF|^2+|\Psi_4|^2\right)(u_\infty,\ub')d\ub'-\frac{a^2}{16|u|}\\
    &\geq a-\frac{a^2}{16|u|}\\
    &\geq\frac{3a}{4}.
\end{align*}
Taking $u=-\frac{a}{4}$, we have
    \begin{align*}
        \int_0^1\left(|\hch|^2+|\bF|^2+|\Psi_4|^2\right)\left(-\frac{a}{4},\ub'\right)d\ub'\geq\frac{12}{a}.
    \end{align*}
    Finally, we have from Proposition \ref{nulles} and \eqref{nullidentities}
    \begin{align*}
        \pr_\ub(\Om^{-1}\trch)=-\frac{1}{2}(\trch)^2-|\hch|^2-\S_{44}\leq -|\hch|^2-|\bF|^2-|\Psi_4|^2.
    \end{align*}
    Recalling that we have on $\Hb_0$
    \begin{align*}
        \Om^{-1}\trch\left(-\frac{a}{4},0\right)=\frac{8}{a},
    \end{align*}
    we infer
    \begin{align}\label{trchne}
        \Om^{-1}\trch\left(-\frac{a}{4},1\right)\leq\frac{8}{a}-\int_0^1\left(|\hch|^2+|\bF|^2+|\Psi_4|^2\right)\left(-\frac{a}{4},\ub'\right)d\ub'\leq-\frac{4}{a}<0.
    \end{align}
    Next, we have from Theorem \ref{maintheorem} that
    \begin{align*}
        \Om\trchb+\frac{2}{|u|}\les \frac{a}{|u|^3}\les a^{-2},
    \end{align*}
    which implies for $a\gg 1$
    \begin{align*}
        \Om\trchb\leq -\frac{2}{|u|}+\frac{4}{a}.
    \end{align*}
    Taking $u=-\frac{a}{4}$, we obtain
    \begin{align}\label{trchbne}
        \Om\trchb\left(-\frac{a}{4},1\right)\leq -\frac{8}{a}+\frac{4}{a}=-\frac{4}{a}<0.
    \end{align}
    Combining \eqref{trchne} and \eqref{trchbne}, we conclude that $S_{-\frac{a}{4},1}$ is a trapped surface. This concludes the proof of Proposition \ref{trappedsurface}.
\end{proof}
\subsection{Mass increasing}
\begin{prop}\label{massincreasing}
Under the hypothesis of Theorem \ref{thmtrappedsurface}, we have
    \begin{equation}
        m(S_{u_\infty,1})\simeq a.
    \end{equation}
\end{prop}
\begin{proof}
We have from \eqref{Hawkingdf}
\begin{align*}
    \frac{2m}{r}&=\frac{1}{4\pi}\int_S \frac{1}{r^2}d\slg+\frac{1}{16\pi}\int_S \trch\trchb d\slg\\
    &=\frac{1}{16\pi}\int_S\left(\trch\trchb+\frac{4}{r^2}\right) d\slg\\
    &=\frac{1}{16\pi}\int_S \frac{2}{r}\left(\trch-\frac{2}{r}\right)+\trch\left(\trchb+\frac{2}{r}\right)d\slg,
\end{align*}
which implies from Theorem \ref{maintheorem}
\begin{align*}
    m\les \int_S\left(\left|\trch-\frac{2}{r}\right|+\left|\trchb+\frac{2}{r}\right|\right)d\slg\les a.
\end{align*}
    Next, recalling from \eqref{basicnull} that
    \begin{align*}
        \nabs_4\trch&=-\frac{1}{2}(\trch)^2-|\hch|^2-2\omb\trchb-\S_{44},\\
        \nabs_4\trchb&=-\frac{1}{2}\trch\trchb+2\om\trchb+2\rho-\hch\c\hchb+2\sdivs\etab+2|\etab|^2+\Tr\S,
    \end{align*}
    we have
    \begin{align*}
        \nabs_4(\trch\trchb)&=-\trchb\trch^2-\trchb|\hch|^2-2\omb\trchb^2-\trchb\,\S_{44}+2\om\trch\trchb\\
        &+\trch\left(2\rho-\hch\c\hchb+2\sdivs\etab+2|\etab|^2\right)+\trch\,\Tr\S.
    \end{align*}
    Thus, we obtain
    \begin{align*}
        \Om e_4\int_S(\trch\trchb)d\slg&=\int_S\left(-\Om\trchb|\hch|^2-2\Om\omb(\trchb)^2-\Om\trchb\,\S_{44}+2\Om\om\trch\trchb\right)d\slg\\
        &+\int_S\Om\trch\left(2\rho-\hch\c\hchb+2\sdivs\etab+2|\etab|^2+\Tr\S\right)d\slg,
    \end{align*}
    which implies from Theorem \ref{maintheorem}
    \begin{align}\label{e4mest}
        \left|\Om e_4\left(\int_S\trch\trchb d\slg\right)-\frac{2}{r}\int_S\left(|\hch|^2+\S_{44}\right)d\slg\right|\les\frac{1}{|u|}.
    \end{align}
    Notice that we have from \eqref{Hawkingdf}
    \begin{align*}
        \Om e_4(2m)&=\Om e_4(r)\frac{2m}{r}+\frac{r}{16\pi}\Om e_4\left(\int_S\trch\trchb d\slg\right)\\
        &=\Om e_4(r)\frac{2m}{r}+\frac{1}{8\pi}\int_S\left(|\hch|^2+\S_{44}\right)d\slg\\
        &+\frac{r}{16\pi}\left(\Om e_4\left(\int_S\trch\trchb d\slg\right)-\frac{2}{r}\int_S\left(|\hch|^2+\S_{44}\right)d\slg\right),
    \end{align*}
    which implies from \eqref{e4mest}
    \begin{align*}
        \left|\Om e_4(m)-\frac{1}{16\pi}\int_S\left(|\hch|^2+\S_{44}\right)d\slg\right|\les 1.
    \end{align*}
    Hence, we infer\footnote{Here, we used the fact that $m=0$ on $\Hb_0$.}
    \begin{align*}
        \left|m(S_{u_\infty,1})-\frac{1}{16\pi}\int_0^1\int_{S_{u_\infty,\ub'}}\left(|\hch|^2+\S_{44}\right)d\slg d\ub'\right|\les 1.
    \end{align*}
    Combining with \eqref{geqathm}, we obtain
    \begin{align*}
        m(S_{u_\infty,1})\ges a.
    \end{align*}
    This concludes the proof of Proposition \ref{massincreasing}.
\end{proof}
\subsection{Charging process}
\begin{prop}\label{charging}
Under the hypothesis of Theorem \ref{thmtrappedsurface}, we have
    \begin{equation}
        |Q(S_{u_\infty,1})|\simeq \ef a.
    \end{equation}
\end{prop}
\begin{proof}
    We have from Proposition \ref{Maxwellequations} and \eqref{nullidentities}
    \begin{align*}
    \Om\nabs_4\rhoF+\Om\trch\rhoF=\sdivs(\Om\bF)+2\ef\Im\left(\psi\Om\Psi_4^\dag\right).
    \end{align*}
    Applying Lemma \ref{dint}, we have
    \begin{align*}
    \pr_\ub\left(\int_{S_{u,\ub}}\rhoF d\slg\right)&=\int_{S_{u,\ub}} \left(\Om \nabs_4(\rhoF)+\Om\trch \rhoF\right)d\slg\\
    &=\int_{S_{u,\ub}} \left(\sdivs(\Om\bF)+2\ef\Im\left(\psi\Om\Psi_4^\dag\right)\right)d\slg\\
    &=2\ef\int_{S_{u,\ub}}\Om\Im\left(\psi\Psi_4^\dag\right)d\slg.
    \end{align*}
    Integrating it on $\cuvi$, we obtain
    \begin{align*}
        4\pi Q(S_{u_\infty,1})=2\ef\int_{\cuvi}\Om\Im\left(\psi\Psi_4^\dag\right).
    \end{align*}
    Recalling from Theorem \ref{maintheorem}
    \begin{align*}
        |\Om|\simeq 1,\qquad\quad|\psi|\les \frac{\af}{|u|},\qquad\quad|\Psi_4|\les \frac{\af}{|u|},
    \end{align*}
    we infer
    \begin{align*}
        |Q(S_{u_\infty,1})|\les \ef\int_{\cuvi}|\Om||\psi||\Psi_4|\les  \ef\int_0^1\int_{S_{u,\ub}}\frac{a}{|u|^2}d\slg d\ub'\les\ef a.
    \end{align*}
    We also have from \eqref{J4geqa}
    \begin{align*}
        |Q(S_{u_\infty},1)|\geq \frac{\ef}{2\pi}\left|\int_{H_{u_\infty}^{(0,1)}}\Om\Im\left(\psi\Psi_4^\dag\right)\right|\gtrsim \ef a.
    \end{align*}
    This concludes the proof of Theorem \ref{charging}.
\end{proof}
Combining Propositions \ref{trappedsurface}--\ref{charging}, this concludes the proof of Theorem \ref{thmtrappedsurface}.
\section{Scaling argument}\label{secscaling}
\subsection{Spacetime rescaling}\label{secscaling1}
Recall that throughout this paper, we use the coordinate system $x^\a := \{ u, \ub, x^1, x^2 \}$ adapted to the double null foliation, where $(x^1, x^2)$ are stereographic coordinates on $\mathbb{S}^2$. In these coordinates, we study the spacetime region
$$
u_{\infty}\leq u\leq-\frac{a}{4},\qquad\quad 0\leq\ub\leq 1.
$$
The Lorentzian metric $\g$ takes the following form:
$$
\g=-2\Om^2(du\otimes d\ub+d\ub\otimes du)+\slg_{AB}(dx^A-\bbb^A du)\otimes(dx^B-\bbb^B du).
$$
Following \cite{An}, we use a rescaled coordinate system ${x'}^{\a}:=\{u',\ub',{x'}^1,{x'}^2\}$, where
\begin{equation}\label{rescale}
u'=\de\c u,\qquad \ub'=\de\c \ub,\qquad {x'}^A:=\de\c x^A.
\end{equation}
Under the rescaling (\ref{rescale}), it follows
$$\g'({{x'}^\a})=\de^2\c\g(x^\a).$$ 
In $\{{x'}^{\a}\}$--coordinates, we let
$$
\g'({x'}^\a)=-2{\Om'}^2(du'\otimes d\ub'+d\ub'\otimes du')+\slg'_{A'B'}(d{x'}^A-{\bbb'}^A du')\otimes(d{x'}^B-{\bbb'}^Bdu').
$$
Comparing with
$$
\g(x^{\a})=-2\Om^2(du\otimes d\ub+d\ub\otimes du)+\slg_{AB}(d\th^A-\bbb^A du)\otimes(d\theta^B-\bbb^B du),
$$
we obtain
\begin{align*}
    d{x'}^{\a}&=\de\c d x^\a,\qquad\quad e'_\mu({x'}^{\a})=\de^{-1}e_\mu(x^\a),\\
    {\Om'}({x'}^{\a})&=\Om(x^{\a}), \qquad \slg'_{A'B'}({x'}^{\a})=\slg_{AB}(x^{\a}), \qquad {\bbb'}^A({x'}^{\a})=\bbb^A(x^{\a}).
\end{align*}
\subsection{Rescaled quantities}\label{secscaling2}
We have the following propositions.
\begin{prop}\label{Prop rescale 1}
For $\Gamma\in \{\hch, \trch, \hchb, \trchb, \eta, \etab, \zeta, \om,\omb\}$ written in two different coordinate systems $\{{x'}^{\a}\}$ and $\{x^{\a}\}$ related by \eqref{rescale}, we have
$$\Ga'({x'}^{\a})=\de^{-1}\c\Ga(x^{\a}).$$
\end{prop}
\begin{proof}
See Proposition 8.1 in \cite{An}.
\end{proof}
\begin{prop}\label{Prop rescale 2}
For $\W\in \{\a, \b, \rho, \si, \bb, \aa\}$ written in two different coordinate systems $\{{x'}^{\a}\}$ and $\{x^{\a}\}$ related by \eqref{rescale}, we have the following identity:
$$\W'({x'}^{\a})=\de^{-2}\c \W(x^\a).$$
\end{prop}
\begin{proof}
    See Proposition 8.2 in \cite{An}.
\end{proof}
\begin{prop}\label{Prop rescale 3}According to Einstein equations
\begin{equation}\label{G=T}
    \begin{split}
    \Ric_{\mu\nu}-\frac{1}{2}\R\g_{\mu\nu}&=\T^{EM}_{\mu\nu}+\T^{CSF}_{\mu\nu},\\
    \T^{EM}_{\mu\nu}&=2\left(\F_{\mu\a}{\F_{\nu}}^\a-\frac{1}{4}\g_{\mu\nu}\F_{\a\b}\F^{\a\b}\right),\\
    \T^{CSF}_{\mu\nu}&= 2 \left(\Re\left(\Psi_\mu \Psi_\nu^\dag\right)-\frac{1}{2}\g_{\mu\nu} \Psi_\a (\Psi^\a)^\dag\right),
    \end{split}
\end{equation}
any quantity $T\in\{\bF, \rhoF, \siF, \bbF, \Psi_4, \Psisl, \Psi_3\}$, when written in two different coordinate systems $\{{x'}^\a\}$ and $\{x^\a\}$ related by the rescaling \eqref{rescale}, satisfies
$$T'({x'}^\a)=\de^{-1}\c T(x^\a),$$
while the components of $\A$ and the complex scalar wave $\psi$ transform under \eqref{rescale} as
$$
\Ub'({x'}^\a)=\Ub(x^{\a}),\qquad \Asl'_{A'}({x'}^\a)=\Asl_A(x^\a),\qquad \psi'({x'}^\a)=\psi(x^\a). 
$$
\end{prop}
\begin{proof}
    It's obvious that Einstein tensor remains invariant under \eqref{rescale}:\footnote{$\Ric-\frac{1}{2}\R\g$ is invariant as a tensor, so applying to $(e_\mu,e_\nu)$, we see $\Ric'_{\mu\nu}({x'}^\a)-\frac{1}{2}\R'({x'}^\a)\g'_{\mu\nu}({x'}^\a)=\Ric_{\mu\nu}(x^\a)-\frac{1}{2}\R(x^{\a})\g_{\mu\nu}(x^{\a})$.}
\begin{equation}
    \Ric'({x'}^\a)-\frac{1}{2}\R'({x'}^\a)\g'({x'}^\a)=\Ric(x^\a)-\frac{1}{2}\R(x^{\a})\g(x^{\a}).
\end{equation}
From \eqref{G=T}, we observe that the Maxwell field \(\F\) and, consequently, the electromagnetic potential \(\A\) rescale as follows:
\begin{equation}\label{F rescaling1}
    \F'_{\mu\nu}({x'}^\a)=\de \c  \F_{\mu\nu}(x^{\a}),\quad \A'_\mu({x'}^\a)=\de \c \A_\mu(x^\a).
\end{equation}
Additionally, from \eqref{G=T}, the gauge covariant derivatives \(\Psi\) and, hence, the complex scalar field \(\psi\) transform under \eqref{rescale} as:
\begin{equation}\label{Psi rescaling1}
    \Psi'_{\mu}({x'}^\a)=\Psi_{\mu}(x^{\a}),\quad \psi'({x'}^\a)=\psi(x^\a).
\end{equation}
Hence, we have
\begin{align*}
\F'_{\mu'\nu'}({x'}^\a)&=\F'({x'}^\a)\left(e_\mu',e_\nu' \right)=\de^{-2}\F'({x'}^\a)\left(e_\mu,e_\nu\right)=\de^{-2}\de \c  \F_{\mu\nu}(x^{\a})=\de^{-1}\F_{\mu\nu}(x^{\a}),\\
\A'_{\mu'}({x'}^\a)&=\A'({x'}^\a)\left(e_\mu'\right)=\de^{-1}\A'({x'}^\a)(e_\mu)=\de^{-1}\de\c\A_\mu(x^\a)=\A_\mu(x^\a),\\
\Psi'_{\mu'}({x'}^\a)&=\Psi'({x'}^\a)\left(e_\mu'\right)=\de^{-1}\Psi'({x'}^\a)(e_\mu) =\de^{-1}\Psi_{\mu}(x^{\a}).
\end{align*}
This concludes Proposition \ref{Prop rescale 3}.
\end{proof}
\begin{prop}\label{Prop rescale 4}
    The Maxwell equation with a fixed coupling constant
    \begin{equation}\label{DF=ej}
    \D^\mu\F_{\mu\nu}=-2\ef\Im(\psi \Psi_\nu^\dag)    
    \end{equation}
    is not scale invariant when coupled to Einstein equations \eqref{G=T}. In particular, the scaling \eqref{rescale} converts \eqref{DF=ej} into the following Maxwell equation
    \begin{equation}\label{DF=ej'}
    \D^{'\mu'}\F'_{\mu'\nu'}=-2\ef'\Im(\psi' \Psi_{\nu'}^{'\dag}),    
    \end{equation}
    with a different coupling constant
    \begin{equation}
        \ef'=\de^{-1}\ef.
    \end{equation}
\end{prop}
\begin{proof}
    We have from Proposition \ref{Prop rescale 3},
    \begin{align*}
    ({\D'}^{\mu'}\F'_{\mu'\nu'})(x^{'\a})&={\g'}^{\mu'\la'}({x'}^{\a})(\D'_{\la'}\F'_{\mu'\nu'})({x'}^{\a})\\
    &=\g^{\mu\la}(x^{\a})(\de^{-1}\D_{\la}(\de^{-1}\F_{\mu\nu}))(x^{\a})\\
    &=\de^{-2}(\D^{\mu}\F_{\mu\nu})(x^{\a}),
    \end{align*}
    and
    \begin{align*}
    \Im(\psi'\Psi_{\nu'}^{'\dag})({x'}^{\a})&=\Im(\psi(\de^{-1}\Psi_{\nu})^{\dag})(x^{\a})=\de^{-1}\Im(\psi\Psi_\nu^\dag)(x^\a).
    \end{align*}
    Then, we deduce from the original Maxwell equation \eqref{DF=ej}
    \begin{align*}
    ({\D'}^{\mu'}\F'_{\mu'\nu'})(x^{'\a})=\de^{-2}(\D^{\mu}\F_{\mu\nu})(x^{\a}) = -2\de^{-2}\ef\Im(\psi \Psi_\nu^\dag)(x^\a)=-2\de^{-1}\ef\Im(\psi' {\Psi'}_{\nu'}^{\dag})({x'}^{\a}).
    \end{align*}
    This finishes the proof of Proposition \ref{Prop rescale 4}.
\end{proof}
\subsection{Main theorem in rescaled version}
Unlike the Einstein--Maxwell or Einstein–scalar field systems, the EMCSF system loses scale invariance due to the coupling between the Maxwell field and the charged scalar field, as shown in Proposition \ref{Prop rescale 4}. With a fixed physical coupling constant, the system is not scale invariant unless the coupling constant itself rescales. We define the rescaled EMCSF' system as follows:
\begin{align}
    \begin{split}\label{EMCSF'}
        \Ric'_{\mu'\nu'}-\frac{1}{2}\R'\g'_{\mu'\nu'}&={\T'}^{EM}_{\mu'\nu'}+{\T'}^{CSF}_{\mu'\nu'},\\
        {\D'}^{\mu'}\F_{\mu'\nu'}&=-2\ef'\Im\left(\psi'(D'_{\nu'}\psi)^\dag\right),\\
        {\g'}^{\mu'\nu'}D_{\mu'}D_{\nu'}\psi'&=0,
    \end{split}
\end{align}
where the coupling constant is given by:
\begin{align}\label{dfef'}
\ef'=\de^{-1}\ef.
\end{align}
We are now prove the following theorem.
\begin{thm}\label{thmscaling}
    Consider the EMCSF' system \eqref{EMCSF'} with $\ef'$ defined by \eqref{dfef'}. Let $\de>0$ be a fixed constant. Then, there exists a sufficiently large $a_0(\de)>0$ and a sufficiently small $e_0(\de)>0$. For $0<a_0<a$ and $0<\ef'<e_0'(\de)$, with an initial data that satisfies: \begin{itemize}
    \item The following assumption holds along $u'=u_\infty'$:
    $$
    \sum_{i\leq 10,k\leq 4}\afd|u_\infty'|\left\|(\de\nabs'_{4'})^k(|u_\infty'|\nabs')^i\left(\hch',\bF',\psi'\right)\right\|_{L^\infty(S_{u_\infty',\ub'})}\leq 1.
    $$
    \item Minkowskian initial data along $\ub'=0$.
    \item The following lower bound condition holds along $u'=u_\infty'$:
    \begin{align*}
        \int_0^\de |u_\infty'|^2\left(|\hch'|^2+|\bF'|^2+|\Psi_4'|^2\right)(u_\infty',\ub',x^{'1},x^{'2})d\ub'\geq \de a,\qquad\forall\; (x^{'1},x^{'2}).
    \end{align*}
    \item The following lower bound condition holds along $u'=u_\infty'$:
    \begin{align*}
    \left|\int_{H_{u_\infty'}^{(0,\de)}}\Om'\Im(\psi'\Psi_{4'}^{'\dag})\right|\geq\de a.
    \end{align*}
    \end{itemize}
    \begin{figure}[H]
    \centering
    \begin{tikzpicture}[scale=0.9]
\fill[yellow] (1,1)--(2,2)--(4,0)--(3,-1)--cycle;

\draw[very thick] (1,1)--(2,2)--(4,0)--(3,-1)--cycle;

\node[scale=1, black] at (4.1,-1) {$H_{u_\infty}$};
\node[scale=1, black] at (1.1,1.9) {$H_{-\frac{\de a}{4}}$};
\node[scale=1, black] at (3.2,1.4) {$\Hb_\de$};
\node[scale=1, black] at (1.7,-0.3) {$\Hb_0$};

\draw[dashed] (0,-4)--(0,4);
\draw[dashed] (0,-4)--(4,0)--(0,4);
\draw[dashed] (0,0)--(2,2);
\draw[dashed] (0,2)--(3,-1);

\draw[->] (3.3,-0.6)--(3.6,-0.3) node[midway, sloped, above, scale=1] {$e_4$};
\draw[->] (2.4,-0.3)--(1.7,0.4) node[midway, sloped, above, scale=1] {$e_3$};

\filldraw[red] (2,2) circle (2pt);
\filldraw[red] (4,0) circle (2pt);

\draw[thin] (2.1,2.5) node[above, scale=1] {$S_{-\frac{\de a}{4},\de}$} -- (2,2);
\draw[thin] (4.1,0.5) node[above, scale=1] {$S_{u_\infty,\de}$} -- (4,0);
\end{tikzpicture}
    \caption{{\footnotesize The initial conditions in Theorem \ref{thmscaling} lead to trapped surface ($S_{-\frac{\de a}{4},\de}$) formation in the future of the $H_{u_\infty}$ and $\Hb_0$.}}
    \label{scalingfigure}
\end{figure}
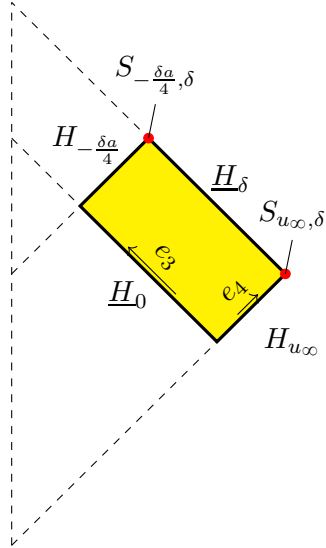
    Then, Einstein--Maxwell--charged scalar field system \eqref{EMCSF} admits a unique smooth solution in the colored region of Figure \ref{scalingfigure}. Moreover, we have:
\begin{itemize}
        \item The $2$--sphere $S_{-\frac{\de a}{4},\de}$ is a trapped surface.
        \item The Hawking mass of the sphere $S_{u_\infty',\de}$ satisfies
        \begin{align*}
            m'(S_{u_\infty',1})\simeq\de a.
        \end{align*}
        \item The electric charge of the sphere $S_{u_\infty',\de}$ satisfies
        \begin{align*}
            |Q'(S_{u_\infty',1})|\simeq\de^2\ef' a.
        \end{align*}
        \item The following estimates hold in the colored region of Figure \ref{scalingfigure}:
    \begin{align*}
    |\a'|&\les\frac{\de^{-1}\af}{|u'|},\quad\; |\b'|\les\frac{a^\frac{1}{2}}{|u'|^2}, \quad |\rho',\si'|\les\frac{\de a}{|u'|^3},\quad\; |\bb'|\les\frac{\de^2 a^\frac{3}{2}}{|u'|^4},\quad\; |\aa'|\les\frac{\de^3 a^2}{|u'|^5},\\
    |\hch'|&\les\frac{\af}{|u'|},\qquad\qquad\;\;\,|\om',\trch',\log\Om'|\les\frac{1}{|u'|},\quad\;\;\;|\eta',\etab',\ze',\hchb',\bbb',\Psi'_{3'}|\les\frac{\de\af}{|u'|^2},\\
    |\trchbt',\omb'|&\les\frac{\de^2a}{|u'|^3},\qquad\qquad\;\;\;\,
    |\bF',\Psi'_{4'},\Asl'|\les\frac{\af}{|u|'},\qquad\;|\rhoF',\siF',\Psisl',\Ub'|\les\frac{\de a}{|u'|^2},\\
    |\bbF',\Psit'_{3'}|&\les\frac{\de^2 a^\frac{3}{2}}{|u'|^3}, \qquad\qquad\qquad\qquad\quad|\psi'|\les \frac{\de a^{\frac{1}{2}}}{|u'|}.
\end{align*}
\end{itemize}
\end{thm}
Theorem \ref{thmscaling} is proved in Section \ref{rescale bounds}, based on the proof of Theorem \ref{maintheorem} and standard scaling arguments, which we recall in Sections \ref{secscaling1} and \ref{secscaling2}. See also Section 8 in \cite{An} and Section 9 in \cite{AnAth} for similar arguments.
\subsection{Proof of Theorem \ref{thmscaling}}\label{rescale bounds}
We now prove Theorem \ref{thmscaling} by applying the rescaling properties established in Propositions \ref{Prop rescale 1}–\ref{Prop rescale 3} and transferring the estimates derived previously to the rescaled coordinates. Applying Proposition \ref{Prop rescale 1} and Proposition \ref{Prop rescale 2}, we establish the connection to \cite{AnLuk,Chr}. Take $\hch$ as an example. Applying Proposition \ref{Prop rescale 1}, \eqref{esthch} and the fact that $u'=\de u$, we have 
\begin{equation*}
\begin{split}
&|\hch'_{A'B'}({x'}^\a)|=\de^{-1}\c  |\hch_{AB}(x^{\a})|\leq \de^{-1}\c  \frac{\af}{|u|}=\frac{\af}{\de|u|}=\frac{\af}{|u'|}. 
\end{split}
\end{equation*}
In the same fashion, we have
\begin{align*}
|\hchb'_{A'B'}({x'}^\a)|&=\de^{-1}\c  |\hchb_{AB}(x^{\a})|\leq \de^{-1}\c  \frac{\af}{|u|^2}=\frac{\de\af}{\de^2|u|^2}=\frac{\de\af}{|u'|^2}, \\
|\trch'({x'}^\a)|&=\de^{-1}\c  |\trch(x^{\a})|\leq \de^{-1}\c  \frac{1}{|u|}=\frac{1}{\de|u|}=\frac{1}{|u'|}, \\
|\eta'_{A'}({x'}^\a)|&=\de^{-1}\c  |\eta_{A}(x^{\a})|\leq \de^{-1}\c  \frac{\af}{|u|^2}=\frac{\de\af}{\de^2|u|^2}=\frac{\de\af}{|u'|^2},\\
|\etab'_{A'}({x'}^\a)|&=\de^{-1}\c  |\etab_{A}(x^{\a})|\leq \de^{-1}\c  \frac{\af}{|u|^2}=\frac{\de\af}{\de^2|u|^2}=\frac{\de\af}{|u'|^2}, \\
|\om'({x'}^\a)|&=\de^{-1}\c  |\om(x^{\a})|\leq \de^{-1}\c  \frac{1}{|u|}=\frac{1}{\de|u|}=\frac{1}{|u'|}, \\
|\omb'({x'}^\a)|&=\de^{-1}\c  |\omb(x^{\a})|\leq \de^{-1}\c  \frac{a}{|u|^3}=\frac{\de^2 a}{\de^3|u|^3}=\frac{\de^2 a}{|u'|^3},\\
|\trchbt'({x'}^\a)|&=\de^{-1}\c|\trchbt|\leq \de^{-1}\c  \frac{a}{|u|^3}=\frac{\de^2 a}{\de^3|u|^3}=\frac{\de^2 a}{|u'|^3}.
\end{align*}
In the same manner, by Proposition \ref{Prop rescale 2} and with the help that $|u'|\geq  \de a/4$ we have
\begin{align*}
|\a'_{A'B'}({x'}^\a)|&=\de^{-2}\c|\a_{AB}(x^{\a})|\leq\de^{-2}\c\frac{\af}{|u|}=\frac{\de^{-1}\af}{\de |u|}=\frac{\de^{-1}\af}{|u'|},\\
|\b'_{A'}({x'}^\a)|&=\de^{-2}\c|\b_{A}(x^{\a})|\leq \de^{-2}\c\frac{\af}{|u|^2}=\frac{\af}{\de^2|u|^2}=\frac{\af}{|u'|^2},\\
|\rho'({x'}^\a)|&=\de^{-2}\c|\rho(x^{\a})|\leq \de^{-2}\c\frac{a}{|u|^3}=\frac{\de a}{\de^3|u|^3}=\frac{\de a}{|u'|^3},\\
|\si'({x'}^\a)|&=\de^{-2}\c|\si(x^{\a})|\leq \de^{-2}\c\frac{a}{|u|^3}=\frac{\de a}{\de^3|u|^3}=\frac{\de a}{|u'|^3}, \\
|\bb'_{A'}({x'}^\a)|&=\de^{-2}\c |\bb_{A}(x^{\a})|\leq \de^{-2}\c\frac{a^{\frac{3}{2}}}{|u|^4}=\frac{\de^2 a^{\frac{3}{2}}}{\de^4|u|^4}=\frac{\de^2 a^{\frac{3}{2}}}{|u'|^4},\\
|\aa'_{A'B'}({x'}^\a)|&=\de^{-2}\c |\aa_{AB}(x^{\a})|\leq \de^{-2}\c\frac{a^2}{|u|^5}=\frac{\de^3 a^2}{\de^5|u|^5}=\frac{\de^3 a^2}{|u'|^5}.
\end{align*}
Similarly, applying Proposition \ref{Prop rescale 3} and the fact that $|u'|\geq\de a/4$, we have
\begin{align*}
|\bF'_{A'}({x'}^\a)|&=\de^{-1}\c|\bF_{A}(x^{\a})|\leq \de^{-1}\c\frac{\af}{|u|}=\frac{\af}{\de|u|}=\frac{\af}{|u'|}, \\
|\bbF'_{A'}({x'}^\a)|&=\de^{-1}\c|\bbF_{A}(x^{\a})|\leq \de^{-1}\c\frac{a^\frac{3}{2}}{|u|^3}=\frac{\de^2 a^\frac{3}{2}}{\de^3|u|^3}=\frac{\de^2 a^\frac{3}{2}}{|u'|^3},\\
|\rhoF'({x'}^\a)|&=\de^{-1}\c|\rhoF(x^{\a})|\leq \de^{-1}\c\frac{a}{|u|^2}=\frac{\de a}{\de^2|u|^2}=\frac{\de a}{|u'|^2},\\
|\siF'({x'}^\a)|&=\de^{-1}\c|\siF(x^{\a})|\leq \de^{-1}\c \frac{a}{|u|^2}=\frac{\de a}{\de^2|u|^2}=\frac{\de a}{|u'|^2},\\
|\Psi'_{4'}({x'}^\a)|&=\de^{-1}\c|\Psi_{4}(x^\a)|\leq \de^{-1}\c\frac{\af}{|u|}=\frac{\af}{\de|u|}=\frac{\af}{|u'|},\\
\left|\Psit'_{3'}({x'}^\a)\right|&=\de^{-1}\c\left|\Psit_3(x^{\a})\right|\leq \de^{-1}\c\frac{a^{\frac{3}{2}}}{|u|^3}=\frac{\de^2 a^{\frac{3}{2}}}{\de^3|u|^3}=\frac{\de^2 a^{\frac{3}{2}}}{|u'|^3},\\
|\Psisl'_{A'}({x'}^\a)|&=\de^{-1}\c|\Psisl_{A}(x^{\a})|\leq\de^{-1}\c\frac{a}{|u|^2}=\frac{\de a}{\de^2|u|^2}=\frac{\de a}{|u'|^2}.
\end{align*}
By the above rescaling argument, we hence transfer the bounds derived in the previous sections into new bounds for the rescaled EMSCF' \eqref{EMCSF'}, holding in the spacetime region:
\begin{align*}
    -\de u_\infty\leq u'\leq -\de a,\qquad\quad 0\leq \ub\leq \de.
\end{align*}
Repeating the arguments in Section \ref{proofmain} and in Sections \ref{secF}--\ref{sectrapping}, this concludes the proof of Theorem \ref{thmscaling}.
\appendix
\section{Proof of results in Section \ref{secpre}}
In this appendix, we prove the Maxwell equations and complex wave equations state in Section \ref{secpre}.
\subsection{Proof of Proposition \ref{Maxwellequations}}\label{AppendixM}
\begin{lem}\label{lem:DivF}
In double null foliation, $\D^\mu\F_{\mu\nu}=-2\ef\Im\left(\psi (D_\nu\psi)^\dag \right)$ can be written as the following using null components:
\begin{equation}
\begin{split}\label{eqn:DivF}
\nabs_3\rhoF+\trchb\rhoF&=-\sdivs\bbF-\frac{\eta+\etab}{2}\c\bbF-2\ef\Im\left(\psi \Psi_3^\dag\right),\\
\nabs_4\rhoF+\trch\rhoF&=\sdivs\bF+\frac{\eta+\etab}{2}\c\bF+2\ef\Im\left(\psi \Psi_4^\dag\right),\\
\nabs_3\bF+\nabs_4\bbF&=-\frac{1}{2}\trchb\bF+2\omb\bF+\hchb\c\bF-\frac{1}{2}\trch\bbF+2\om\bbF\\
&+\hch\c\bbF+2\rhoF(\eta-\etab)-2\siF(\dual\eta+\dual\etab)-2\dual\nabs\siF\\
&-4\ef\Im\left(\psi\Psisl^\dag\right).
\end{split}
\end{equation}
\end{lem}
\begin{proof}
    See Section 2.3.3 in \cite{Giorgi} for the calculation of $\D^\mu\F_{\mu\nu}$ without gauge fixing. Imposing the double null gauge and combining with the null components of $-2\ef\Im\left(\psi (D_\nu\psi)^\dag\right)$ then yields the equations directly.
\end{proof}
\begin{lem}\label{lem:CurlF}
    In double null foliation, the Bianchi identities $\D_{[\mu}\F_{\nu\lambda]}=0$\footnote{A consequence of $\F=d\A$.} of $\F$ can be written as the following using null components:
    \begin{equation}\label{eqn:CurlF}
        \begin{split}
\nabs_3\bF-\nabs_4\bbF&=-\frac{1}{2}\trchb\bF+2\omb\bF-\hchb\c\bF+\frac{1}{2}\trch \bbF-2\om\bbF\\
&+\hch\c\bbF+2\rhoF(\eta+\etab)-2\siF(\dual\eta-\dual\etab)+2\nabs\rhoF,\\ 
\nabs_3\siF+\trchb\siF&=\curls\bbF-(\ze-\eta)\wedge\bbF,\\
\nabs_4\siF+\trch\siF&=\curls\bF+(\ze+\etab)\wedge\bF.
        \end{split}
    \end{equation}
\end{lem}
\begin{proof}
    See Section 2.3.3 in \cite{Giorgi} for the full derivation without gauge fixing. Imposing the double null gauge directly yields the equations.
\end{proof}
\begin{proof}[Proof of Proposition \ref{Maxwellequations}]
    Taking addition and subtraction of \eqref{eqn:DivF} and \eqref{eqn:CurlF}, we get
    \begin{align*}
        \nabs_3\bF+\frac{1}{2}\trchb\bF-2\omb\bF&=\nabs\rhoF-\dual\nabs\siF+2\rhoF\eta-2\siF\dual\eta+\hch\c\bbF -2\ef\Im\left(\psi \Psisl^\dag\right),\\
        \nabs_4\bbF+\frac{1}{2}\trch\bbF-2\om\bbF&=-\nabs\rhoF-\dual\nabs\siF-2\rhoF\etab-2\siF\dual\etab+\hchb\c\bF -2\ef\Im\left(\psi \Psisl^\dag\right).
    \end{align*}
    This concludes the proof of Proposition \ref{Maxwellequations}.
\end{proof}
\subsection{Proof of Proposition \ref{lem:FdAeq}}\label{AppendixA}
The electromagnetic $2$-form $\F$ is given by $\F=d\A$ for electromagnetic potential $1$-form $\A$. We recall the null components of $\F$ are
\begin{equation}
    \F_{A4}=\bF,\qquad \F_{A3}=\bbF,\qquad\F_{34}=2\rhoF,\qquad\F_{AB}=-\eps_{AB}\siF.
\end{equation}
We have from $\F=d\A$ that
\begin{equation}
    \F_{\mu\nu}=\D_\mu\A_\nu-\D_\mu \A_\nu=(\D_\mu\A)(e_\nu)-(\D_\nu\A)(e_\mu).
\end{equation}
We first expand $\rhoF$:
\begin{align*}
    \rhoF&=\frac{1}{2}\F_{34}=\frac{1}{2} (\D_3\A)(e_4)-\frac{1}{2} (\D_4\A)(e_3)\\
    &=\frac{1}{2} \D_3(\A(e_4))-\frac{1}{2}\A(\D_3e_4) -\frac{1}{2} \D_4(\A(e_3))+\frac{1}{2}\A(\D_4e_3)\\
    &=\frac{1}{2} \D_3U-\frac{1}{2}\A(2\omb e_4 +2\eta^B e_B) -\frac{1}{2}\D_4\Ub+\frac{1}{2}\A(2\om e_3 +2\etab^B e_B)\\
    &=\frac{1}{2}\nabs_3U-\frac{1}{2}\nabs_4\Ub-\omb U+\om\Ub +(\etab-\eta)\c \Asl.
\end{align*}    
Now for $\siF$, by symmetry of Christoffel symbols on any manifolds, we have
\begin{align*}
    -\slep_{BC}\siF&=\F_{BC}\\
    &=\D_B\A_C-\D_C \A_B\\
    &=\pr_B(\A(e_C))-\pr_C(\A(e_B))\\
    &=\nabs_B\Asl_C-\nabs_C\Asl_B.
\end{align*}
Contracting with $\slep^{BC}$, we deduce
\begin{align*}
    \siF=-\curls\Asl.
\end{align*}
Next, we rewrite $\bF$:
\begin{align*}
    \bF_B&=\F_{B4}=(\D_B\A)(e_4)-(\D_4\A)(e_B)\\
    &=\D_B(\A(e_4))-\A(\D_Be_4) - \D_4(\A(e_B))+\A(\D_3e_4)\\
    &=\nabs_B U -\A (\chi_{BC}e_C -\ze_B e_4) -\pr_4(\Asl_B)+\A (\nabs_4 e_B +\etab_B e_4 )\\
    &=\nabs_B U-(\pr_4(\Asl_B)-\Asl(\nabs_4 e_B))-\chi_{BC}\Asl_B+U\ze_B +U\etab_B\\
    &=\nabs_B U-\nabs_4\Asl_B-\chi_{BC}\Asl_B+U\ze_B +U\etab_B.
\end{align*}
The equation of $\bbF$ is similar. This concludes the proof of Proposition \ref{lem:FdAeq}.
\subsection{Proof of Proposition \ref{waveequation}}\label{AppendixPsi}
First notice that $D_\mu\Psi_\nu=D_\mu D_\nu\psi$ is not symmetric in $\mu,\nu$:
\begin{align*}
    D_\mu D_\nu\psi
    &=(\D_\mu+i\ef \A_\mu)(\D_\nu+i\ef\A_\nu)\psi\\
    &=(\D_\mu+i\ef \A_\mu)(\D_\nu\psi+i\ef \A_\nu \psi)\\
    &=\D_\mu\D_\nu\psi+i\ef \A_\nu (\D_\mu\psi)+i\ef (\D_\mu\A_\nu) \psi+i\ef \A_\mu(\D_\nu\psi)-\ef^2\A_\mu\A_\nu \psi,
\end{align*}
which implies
\begin{align*}
    D_\mu\Psi_\nu-D_\nu\Psi_\mu=i\ef (\D_\mu\A_\nu-\D_\nu\A_\mu) \psi=i\ef \F_{\mu\nu}\psi.
\end{align*}
In particular,
\begin{align*}
    D_4\Psi_3=D_3\Psi_4+i\ef \F_{43}\psi=D_3\Psi_4-2i\ef \rhoF \psi.
\end{align*}
We now rewrite $\g^{\mu\nu}D_\mu D_\nu\psi=0$ to get the $\nabs_3(\Psi_4)$-equation:
\begin{align*}
    0&=\g^{\mu\nu}D_\mu D_\nu\psi\\
    &=-\frac{1}{2}D_3\Psi_4-\frac{1}{2}D_4\Psi_3+\slg^{BC}D_C\Psi_B\\
    &=-D_3\Psi_4 +\slg^{BC}D_C\Psi_B+i\ef \rhoF \psi\\
    &=-\left((\D_3\Psi)(e_4)+i\ef\Ub\Psi_4\right)+\slg^{BC}\left((\D_C\Psi)(e_B)+ i\ef \Asl_C\Psisl_B\right)+i\ef \rhoF \psi\\
    &=-\nabs_3(\Psi_4)+\Psi(\D_3e_4)+\slg^{BC}\left(\pr_C(\Psisl_B)- \Psi(\D_Ce_B)\right)-i\ef\Ub\Psi_4+i\ef \Asl\c\Psisl+i\ef\rhoF\psi\\
    &=-\nabs_3(\Psi_4)+\Psi(2\omb e_4+2\eta^Be_B)+\slg^{BC}\left(\pr_C(\Psisl_B)- \Psi(\nabs_Ce_B+\frac{1}{2}\chi_{CB}e_3+\frac{1}{2}\chib_{CB}e_4)\right)\\
    &-i\ef\Ub\Psi_4+i\ef \Asl\c\Psisl+i\ef\rhoF\psi\\
    &=-\nabs_3(\Psi_4)+\sdivs\Psisl+2\omb\Psi_4+2\eta\c\Psisl-\frac{1}{2}\trch\Psi_3-\frac{1}{2}\trchb\Psi_4-i\ef\Ub\Psi_4+i\ef\Asl\c\Psisl+i\ef\rhoF\psi.
\end{align*}
Therefore,
\begin{align*}
    \nabs_3(\Psi_4)+\frac{1}{2}\trchb\Psi_4-2\omb\Psi_4+i\ef \Ub\Psi_4=\sdivs\Psisl+2\eta\c\Psisl+i\ef\Asl\c\Psisl-\frac{1}{2}\trch\Psi_3+i\ef\rhoF\psi,
\end{align*}
as desired.
From Lemma \ref{comm}, we know
\begin{align*}
    [\Om\nabs_4,\nabs]\psi=-\Om\chi\c\nabs\psi.
\end{align*}
Therefore, we can compute the $\nabs_4\Psisl$--equation:
\begin{align*}
    \Om\nabs_4\Psisl&=\Om\nabs_4\left(\nabs\psi+i\ef\Asl\psi\right)\\
    &=[\Om\nabs_4,\nabs]\psi+\nabs(\Om e_4\psi) +i\ef(\Om\nabs_4\Asl)\psi +i\ef \Asl \Om e_4\psi\\
    &=\nabs \left(\Om\Psi_4-i\ef\Om U\psi\right)-\Om\chi\c\nabs\psi +i\ef (\Om\nabs_4\Asl)\psi +i\ef\Om \Asl e_4\psi\\
    &=\nabs(\Om\Psi_4)-\Om\chi\c\nabs\psi-i\ef U\Om\nabs\psi+i\ef\Om\Asl e_4\psi+i\ef (\Om\nabs_4\Asl-\nabs(\Om U))\psi.
\end{align*}
By Lemma \ref{lem:FdAeq}, we know
\begin{align*}
    \Om\nabs_4\Asl-\nabs(\Om U)-\Om\chi\c\Asl=-\Om\bF,
\end{align*}
hence,
\begin{align*}
    \Om\nabs_4\Psisl+\frac{1}{2}\Om\trch\Psisl+\Om\hch\c\Psisl+i\ef \Om U\Psisl=\nabs(\Om\Psi_4)+i\ef\Om\Asl\Psi_4-i\ef\Om\bF\psi.
\end{align*}
Finally, we compute using $\curls\Asl=-\siF$ from Lemma \ref{lem:FdAeq}
\begin{align*}
    \curls\Psisl&=\curls(\nabs\psi+i\ef\Asl\psi)\\
    &=i\ef (\curls\Asl) \psi -i\ef\Asl\wedge\nabs\psi\\
    &=-i\ef \siF\psi -i\ef\Asl\wedge\Psisl.
\end{align*}
    We first write down the difference between $\Psi_3$ and $\Psit_3$
    \begin{align*}
\Psit_3=e_3\psi+i\ef\Ub\psi-\frac{1}{|u|\Om}\psi=\Psi_3-\frac{1}{\Om|u|}\psi.    
    \end{align*}
Therefore,
\begin{align*}
&\quad\;\,\nabs_4(\Psi_3)+\frac{1}{2}\trchb\Psi_4+\frac{1}{2}\trch\Psi_3-2\om\Psi_3\\
&=\nabs_4\left(\Psit_3+\frac{1}{\Om|u|}\psi\right)+\frac{1}{2}\trchb \Psi_4+\frac{1}{2}\trch\left(\Psit_3+\frac{1}{\Om|u|}\psi\right)-2\om\left(\Psit_3+\frac{1}{\Om|u|}\psi\right)\\
&=\nabs_4(\Psit_3)+\left(\frac{1}{2}\trchb + \frac{1}{\Om |u|} \right)\Psi_4-\frac{\nabs_4\Om}{\Om^2 |u|}\psi+\frac{1}{2}\trch\Psit_3+\frac{1}{2\Om|u|}\trch\psi -2\om\Psit_3-\frac{2\om}{\Om |u|}\psi\\
&=\nabs_4(\Psit_3)+\frac{1}{2}\trch\Psit_3-2\om\Psit_3+\left(\frac{1}{2}\trchb + \frac{1}{\Om |u|} \right)\Psi_4+\frac{1}{2\Om|u|}\trch\psi.
\end{align*}
Combining with $\nabs_4(\Psi_3)$-equation, we get the desired $\nabs_4(\Psit_3)$-equation. Now we can also compute
\begin{align*}
&\quad\nabs(\Psi_3)+\frac{\eta+\etab}{2}\Psi_3+i\ef\Asl\Psi_3\\
&=\nabs\left(\Psit_3+\frac{1}{\Om|u|}\psi\right)+\frac{\eta+\etab}{2}\left(\Psit_3+\frac{1}{\Om|u|}\psi\right)+i\ef \Asl\left(\Psit_3+\frac{1}{\Om|u|}\psi\right)\\
&=\nabs(\Psit_3)+\frac{1}{\Om|u|}\nabs\psi-\frac{\nabs\Om}{\Om^2 |u|}\psi +\frac{\eta+\etab}{2}\left(\Psit_3+\frac{1}{\Om|u|}\psi\right)+i\ef \Asl\left(\Psit_3+\frac{1}{\Om|u|}\psi\right)\\
&=\nabs(\Psit_3)+\frac{1}{\Om|u|}\Psisl +\frac{\eta+\etab}{2}\Psit_3+i\ef \Asl \Psit_3\\
&=\nabs(\Psit_3)-\frac{1}{2}\trchb \Psisl +\frac{\eta+\etab}{2}\Psit_3+i\ef \Asl \Psit_3 + \left(\frac{1}{2}\trchb+\frac{1}{\Om|u|} \right)\Psisl.
\end{align*}    
Combining with $\nabs_3\Psisl$--equation, we get the desired $\nabs_3\Psisl$--equation in terms of $\Psit_3$.
This concludes the proof of Proposition \ref{waveequation}.

\vspace{0.1cm}
\small{Dawei Shen: Department of Mathematics, Columbia University, New York, NY, 10027. \\
Email: \textit{ds4350@columbia.edu}\\ \\
Jingbo Wan: Department of Mathematics, Columbia University, New York, NY, 10027.\\
Email: \textit{jingbowan@math.columbia.edu}}
\end{document}